\documentclass[12pt, twoside, letterpaper]{article}
\usepackage[utf8x]{inputenc}
\usepackage{comment}
\usepackage{amsthm}
\usepackage{multirow}
\usepackage{mathtools} 
\usepackage{array}
\newcolumntype{P}[1]{>{\centering\arraybackslash}p{#1}}
\usepackage{longtable}
\usepackage[top=2.5cm,bottom=2.5cm,left=2cm,right=2cm]{geometry}
\usepackage{amsmath,amssymb}
\usepackage[mathscr]{euscript}
\usepackage[usenames,dvipsnames]{color}
\usepackage{booktabs}
\usepackage[colorlinks,linkcolor=blue]{hyperref}
\usepackage{pstricks}
\usepackage{graphicx,caption}
\usepackage{dsfont}
\usepackage[all]{xy}
\usepackage{tabularx}
\usepackage{bbm}
\usepackage{todonotes}
\usepackage{hyperref}
\usepackage{fancyhdr}

\usepackage[font=footnotesize, labelfont=sc, width=0.8\textwidth]{caption}
\usepackage{tikz}
\usepackage{tikz-3dplot}
\usepackage{amsmath}
\usetikzlibrary{arrows.meta,decorations.markings}
\usetikzlibrary{3d,decorations.markings}

\tikzset{
  ->-/.style={decoration={
    markings,
    mark=at position #1 with {\arrow{>}}},postaction={decorate}},
  ->-/.default=0.5
}

\newcommand{\dd}{\mathrm{d}}

\renewcommand{\geq}{\geqslant}
\renewcommand{\leq}{\leqslant}

\newcommand{\E}{ \mathbb{E} } 
 
\newcommand{\N}{ \mathbb{N} } 
\newcommand{\R}{ \mathbb{R} }

\newcommand{\Nc}{ \mathcal{N} }

\newcounter{chronopara}
\newcommand{\myparagraph}[1]{%
  \refstepcounter{chronopara}
  \paragraph{$\mathcal{P}$\arabic{chronopara}. #1}
  \label{para:\arabic{paragraph}}
}

\newcommand{\Tr}{ \mathrm{Tr} }

\newcommand{\gfrak}{ \mathfrak{g} }
\newcommand{\g}{\mathfrak{g}}

\newcommand{\ind}{ \mathrm{ind} } 

\newtheorem{remark}{Remark}
\numberwithin{equation}{section}

\newtheorem{thm}{Theorem}[section]
\newtheorem{proposition}[thm]{Proposition}
\newtheorem{corollary}[thm]{Corollary}

\newtheorem{definition}[thm]{Definition}
\newtheorem{example}[thm]{Example}
\newtheorem{lemma}[thm]{Lemma}

\newcommand{\EN}[1]{{\color{blue} #1}}

\title{Semiclassical analysis for Yang--Mills \\ random connections on compact surfaces}
\author{Nguyen Viet {\sc{Dang}\footnote{IRMA, Université de Strasbourg, 7 rue René Descartes, 67084 Strasbourg Cedex, France AND Institut Universitaire de France, Paris, France. Email : nvdang@unistra.fr} } and Elias \sc{Nohra}\footnote{Sorbonne Université, Université Paris Cité, CNRS, Laboratoire de Probabilités, Statistique et Modélisation, LPSM, F-75005, France. Email : elias.nohra@sorbonne-universite.fr}}
\date{\today}

\begin{document}
\maketitle
\thispagestyle{empty}
\begin{abstract}

We introduce anisotropic Banach spaces of distributional $1$-forms on compact surfaces designed to capture the fine regularity properties of Morse gauge-fixed Yang--Mills random connections. \par Using singular connections supported on unstable curves of a Morse gradient flow, we provide a new description of the moduli space of flat connections and of its canonical Atiyah--Bott--Goldman symplectic form. When the group is the special unitary group, we prove that in the zero-area limit, the Yang--Mills random connection converges, within these anisotropic spaces, to a random distributional $1$-form whose law coincides with our Morse-theoretic representative of the Atiyah--Bott--Goldman measure.
\par Our approach extends the works of Witten, Forman, Liu, and Sengupta by establishing the zero-area limit of the Yang--Mills measure at the level of random distributional connections, rather than only at the level of holonomies. 
This answers a question of Thierry Lévy on the semiclassical limits of Yang--Mills random connections.

\end{abstract}
\tableofcontents
\noindent\rule{\linewidth}{0.4pt}
\section{Introduction}
The Yang--Mills equation on Riemann surfaces occupies a central position in gauge theory, following the seminal work of Atiyah and Bott~\cite{AB}. Starting from the space $\mathcal{A}$ of unitary connections on a Riemann surface and the action of the gauge group $\mathcal{G}$, they introduced the Yang--Mills functional $S_{\mathsf{YM}}$~\cite[def p.~548]{AB} on the quotient $\mathcal{A}/\mathcal{G}$. A fundamental insight of~\cite{AB} is that $S_{\mathsf{YM}}$ behaves as a perfect equivariant Morse function on the infinite-dimensional space $\mathcal{A}$. This leads to a stratification of the space of connections into cells indexed by topological types of bundles~\cite[p.~536--539]{AB}, whose lowest stratum identifies with the moduli space of stable holomorphic bundles. In particular, solving the Yang--Mills equations is equivalent, via Morse theory, to the classification of stable bundles. Using this analytic approach, they recovered the fact that the moduli space of stable bundles over a Riemann surface can be identified with irreducible unitary representations of the surface group~\cite[(8.1), p.~569]{AB}, an important result originally due to Narasimhan--Seshadri~\cite{NS65}.\footnote{In fact, their work was motivated by the results of Narasimhan--Seshadri and Harder--Narasimhan, with the goal of recovering them from a geometric-analytic perspective.} They also determined the cohomology of the moduli space of semi-stable bundles~\cite[Section~9, p.~576]{AB}, thereby recovering earlier results of Harder--Narasimhan~\cite{HS75}; see also~\cite[Section~11, p.~593]{AB}.  
\par These irreducible unitary representations of the surface group can be identified with the moduli space of flat connections, namely the \emph{strict minima} of the Yang--Mills functional \( S_{\mathsf{YM}} \).
Another fundamental aspect of Atiyah--Bott~\cite[p.~587]{AB} is the discovery of a \emph{natural symplectic structure} on the space of connections. The action of the gauge group on \( \mathcal{A} \) is \emph{Hamiltonian}, with the curvature map
\[
A \in \mathcal{A} \longmapsto F(A)
\]
playing the role of a \emph{moment map}. 
Thus, by infinite-dimensional symplectic reduction, the space of flat connections \( \{F(A)=0\} \subset \mathcal{A} \) is identified as the zero level set of the moment map, and the moduli space
\[
\{F(A)=0\} / \mathcal{G}
\]
naturally inherits a symplectic structure from \( \mathcal{A} \). Since the moduli space of flat connections is finite-dimensional, it is natural to ask what its \emph{symplectic volume} is. This symplectic-geometric viewpoint on the moduli space \( \{F(A)=0\} / \mathcal{G} \) later played a central role in the work of Goldman~\cite{Goldman1}.

\par This question about the \emph{symplectic volume} motivates a fundamental connection with quantum gauge theory via semiclassical limits.

Roughly speaking, quantum Yang--Mills theory can be viewed as the study of the Gibbs measure
\[
\exp\left({-\frac{S_{\mathsf{YM}}(A)}{\mathrm{area}}}\right) \, \mathcal{D}A
\]
on the space of connections \( \mathcal{A} \), regarded as the configuration space of a statistical system. Letting the area tend to zero is \emph{analogous to taking a semiclassical or low-temperature limit} of the system.

In his seminal work, Witten~\cite{Witten1,Witten2} observed that the small-area regime of the Yang--Mills measure on a compact surface allows one to compute the symplectic volume of the moduli space of flat connections~\cite[(3.11)--(3.13), p.~178]{Witten1}, thereby revealing a deep connection between the zero-area limit of the Yang--Mills measure and the geometry of flat bundles. At a heuristic level, as the area tends to zero, one can see that the Yang--Mills measure concentrates on configurations minimizing the energy, namely flat connections, and therefore expects the limiting object to be closely related to the Atiyah--Bott--Goldman (ABG) measure~\cite{AB,Goldman1}.
\par Following the rigorous construction of the two-dimensional Yang--Mills measure as a \( G \)-valued stochastic process indexed by loops, due to Driver~\cite{driver1989ym2}, Sengupta~\cite{Sengupta92,Sengupta97}, and L\'evy~\cite{Levyphd,LevyMarkov}, this convergence has been understood at the level of holonomies by Sengupta and King~\cite{KS2,KS3}, as well as by Forman~\cite{Forman}. Subsequent work by Jeffrey~\cite{Je}, Biswas--Guruprasad~\cite{BG}, Sengupta~\cite{Sen2002}, and others further developed the theory of flat connections on surfaces with boundary and punctures. In all of the works mentioned above, in the small-area limit, the measure on holonomies collapses to a finite-dimensional object reflecting the combinatorial structure of holonomies along a finite family of lassos drawn on the surface. Recent contributions in these directions include works by Dahlqvist, François, Lévy, Lemoine, Tarrago, and García-Zelada ~\cite{Dahlq, DahlqLem, FT, FGLT}.

\par While this approach to the Yang--Mills measure provides a complete description of the behavior of Wilson loop observables, it does not yield an underlying random connection, and therefore does not allow for a direct analysis of the semiclassical limit at the level of gauge potentials. 
\par A functional approach to the Yang--Mills measure has emerged in recent years, aiming to construct it as a genuine random distribution. On the flat torus, Chevyrev~\cite{Chevyrev_2019} obtained such a construction, opening the way to a functional-analytic study of Yang--Mills fields. Building on ideas from dynamical systems and Morse theory, the works~\cite{BCDRT,DN} extended this construction to compact surfaces of arbitrary genus. In this framework, the Yang--Mills measure is realized as a probability measure on distributional connections in a specific gauge, the \emph{Morse gauge}, where the geometry of the surface is encoded via a Morse flow and its associated stable and unstable manifolds. This distributional framework provides a natural setting in which to revisit the small-area limit. 
\par From a geometric point of view, this limit can be classified into three different regimes: the case of genus $0$, where the character variety is trivial; the case of genus $1$, where the character variety exhibits singular behavior; and the case of genus greater than or equal to $2$, in which the character variety and its symplectic form are rather regular. The constructions in~\cite{BCDRT,DN} therefore make it possible to explore in a distributional setting and in all these different behaviors, the interplay between the Yang--Mills measure and the symplectic geometry of the moduli space of flat connections.
\par We can highlight the main contributions of our paper as follows.
\begin{enumerate}
\item \textbf{A new anisotropic functional space for the regularity analysis of the Yang--Mills random field.} Inspired by the recent progress on the
analysis of Ruelle spectra in hyperbolic dynamics~\cite{Baladi18, Baladiquest, Liverani, FaureSjostrand},
we introduce a family of low regularity \emph{anisotropic Banach spaces of distributions} tailored to the geometry of the Morse gauge, in which the Yang--Mills measure can be realized as a Borel probability measure. This provides a fine regularity analysis of the Yang--Mills random field and establishes a fundamental decomposition of this field into a noise-like term and a finite-dimensional component encoding topological degrees of freedom.
\item \textbf{A novel Morse-theoretic description of the moduli space of flat connections.} This can be viewed as a non-abelian analogue of the Morse--Witten complex.  
Using the spectral gap for the transfer operator associated to Morse--Smale gradient flows proved by  Rivière and  the first author~\cite{DR19}, we use a dynamical construction to interpolate between smooth flat connections and some distributional connections 
supported on unstable curves of the Morse flow.
This is reminiscent of the geometric measure theory approach to the Morse--Witten complex of Harvey--Lawson~\cite{HL01},  and can be seen as its generalization to non-abelian gauge theories. In this picture, flat connections admit canonical representatives supported on unstable manifolds, leading to a concrete realization of the character variety in terms of distributional connections.
\item \textbf{The semi-classical analysis of the Yang--Mills random fields.} As a consequence of the first two contributions, we establish that, for arbitrary compact surfaces of any genus and for the special unitary groups, the Yang--Mills measure converges in the small-area limit to a measure supported on a finite-dimensional space. Moreover this limit is the (normalized) Atiyah--Bott--Goldman symplectic volume form expressed in the Morse gauge. This provides a direct realization of the semi-classical limit at the level of gauge potentials, rather than only through holonomies, and gives a new probabilistic interpretation of the ABG measure. This result, stated precisely in Theorem~\ref{Thm1}, answers a question raised by Thierry Lévy during a lecture at  \emph{Collège de France}, available online at the time of writing on \href{https://www.youtube.com/watch?v=c3JZt91i2fU}{this link} (check 51:45--53:00).  
\end{enumerate}

\section*{Acknowledgments}
\par N.V.D heartily thanks Némo Sauvion for many explanations and discussions that helped him understand the beautiful results of Forman and Liu. He also thanks Julien Marché
and Arnaud Maret for answering several questions on character varieties. N.V.D acknowledges the support of the Institut Universitaire de France.
\par E.N. expresses his deepest gratitude to his PhD advisor, Thierry Lévy, for constant guidance, generosity, and encouragement. E.N. is profoundly indebted to his insight, patience, and unwavering support, both scientific and personal, throughout his PhD. E.N. also warmly thanks David García-Zelada, François Jacopin, and Thibaut Lemoine for many stimulating and enjoyable discussions related to this work. Finally, he gratefully acknowledges the financial support of the ``Fondation CFM pour la Recherche" and thanks the foundation for providing excellent working conditions during his PhD.

\par Both authors would like to thank Vladimir Fock, Thierry Lévy, and Némo Sauvion for enlightening discussions and suggestions.

\section{Main results}

In this section, we will describe in more details the main theorems of the present work. There are three main results, each discussed in a separate section.

\subsection{Anisotropic functional spaces and the regularity of the Yang--Mills random fields}
\label{ss:firstpart}
\par Previous works~\cite{BCDRT,DN} construct a candidate for the Yang--Mills measure as a probability measure on distributional \(1\)-forms. These works rely on a novel gauge-fixing method, itself based on a Morse function, which can be viewed as a generalization of the axial gauge. Accordingly, we refer to the Yang--Mills measure in this setting as the \emph{Morse gauge-fixed Yang--Mills measure}, or simply as the Yang--Mills random connection or field.

The strategy in~\cite{BCDRT,DN}  is to start from a surface with one boundary component, endowed with free boundary conditions, and define the random variable \( A_{\Sigma,\sigma}^{\mathsf{Free}} \), that is the Yang--Mills measure on the surface with no restriction on the holonomy of the boundary. One then conditions the boundary holonomy to be equal to \(1_G \) in order to obtain the Yang--Mills measure on the corresponding closed surface. With the aim of providing the reader with a global picture of both, the field itself and the functional space supporting it, we give a heuristic description of this random field.

The random field \( A_{\Sigma,\sigma}^{\mathsf{Free}} \) admits a decomposition into two components. The first term is a noise term, and is essentially given by the integral of a white noise in the direction of the gradient of a Morse function \( f \). The second term chooses a flat connection at random, and consists of integration currents supported on the unstable manifolds of \( f \). This structure motivates the introduction, in Section~\ref{SpacesBYM}, of Banach spaces of distributions \( \mathcal{B}_{\mathsf{YM}}^{\alpha,p,s,\ell} \), for $\alpha\in (\frac{1}{3},\frac{1}{2}), p\geqslant 2, s< -1, \ell<0  $ which provide a natural functional framework for these random fields.
\par The space~$\mathcal{B}_{\mathsf{YM}}^{\alpha,p,s,\ell}$ decomposes as a direct sum~$\mathcal{W}^{\alpha,p,s,\ell}\oplus\mathcal{F}$ of two subspaces of the space of $1-$currents on the punctured surface~$\mathcal{D}'^1(\Sigma^0,\mathfrak{g})$. The space~$\mathcal{W}^{\alpha,p,s,\ell}$ consists of distributions with anisotropic regularity. The parameters~$(\alpha,p)$ measure local anisotropic Sobolev regularity in the Gagliardo sense: $p$ is an integrability parameter, while~$\alpha$ is a differentiability parameter, indicating regularity~$\alpha$ in the direction of the flow lines of~$f$ and~$\alpha-1$ in the transverse direction.
\par The two remaining parameters, $\ell$ and $s$, are used to describe the behavior of the Yang--Mills connection near singular points on the surface. We use a Morse function to describe the surface, and in the associated local Morse coordinates, the area measure is distorted near the critical points of~$f$. Since the Yang--Mills random field is intrinsically related to the area form, it is naturally distorted near these critical points.
\par These critical points fall into two types, and different parameters are required to capture the behavior of the local norms as one approaches them. We use~$s$ and~$\ell$ to quantify the weighted behavior near saddle points and near extremal points (maxima and minima), respectively. A schematic illustration is provided in Figure~\ref{fig:reg}.

\begin{figure}
    \centering
    \includegraphics[width=\linewidth, trim={0 0 1cm 0}]{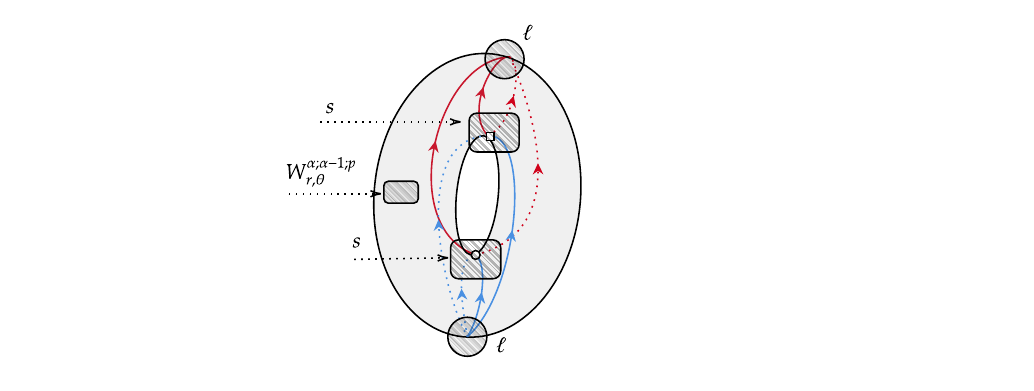}
    \caption{Visual description of the spaces~$\mathcal{W}^{\alpha,p,s,\ell}$: $(\alpha,p)$ encode local Sobolev regularity, with~$\alpha$ measuring differentiability along flow lines and~$\alpha-1$ transversely; $s$ and~$\ell$ control the behavior near saddle points and extremal points respectively.}
    \label{fig:reg}
\end{figure}
\par The second component~$\mathcal{F}$ is a finite-dimensional space of rough $1$-currents supported on 1-dimensional submanifolds of~$\Sigma$, encoding the zero modes of~$A_{\Sigma,\sigma}$.  The space~$\mathcal{F}$ arises canonically in the Morse-gauge description of flat connections, and our representative of the Atiyah--Bott--Goldman measure is supported on~$\mathcal{F}$ as described in the next paragraph. Working in these spaces allows us to extend the decomposition in~\cite[Proposition~1.2]{DN} from the free-boundary case to the closed surface case, and is reminiscent of the decomposition in~\cite[Theorem~1.1]{Lemoine2026}. More precisely, we prove the following result (See also Theorem~\ref{ClosedYMM}, and Lemma~\ref{structureLemma}).

\begin{thm}\label{MainRes1}
For $\alpha<\frac12$, $p\geq 2$, $\ell<0$, and $s<-1$, the Yang--Mills measure can be realized as a Borel probability measure on~$\mathcal{B}_{\mathsf{YM}}^{\alpha,p,s,\ell}$. Moreover, it admits the decomposition
\[
A_{\Sigma,\sigma}=A_{\mathcal{N},\Sigma,\sigma}+\sum_{a=1}^{2g}x_aU_a,
\]
where $A_{\mathcal{N},\Sigma,\sigma}\in \mathcal{W}^{\alpha,p,s,\ell}$ is the noise component, with regularity comparable to that of the Gaussian free field, each $U_a\in\mathcal{F}$ is an integration current on some unstable curve of a Morse gradient flow on~$\Sigma$, and almost surely
\[
\forall a=1,\dots,2g, \quad x_a\in B_{\mathfrak{g}}\bigl(0,\mathsf{diam}(G)\bigr).
\]
\end{thm}

\subsection{Morse-theoretic description of the moduli space of flat connections}

We now turn to the second main contribution of this paper: a new Morse-theoretic model for the character variety arising from dynamical considerations. This construction can be viewed as a non-abelian generalization of a result proved by Harvey and Lawson, later strengthened by the first author and Rivière~\cite{DR19,DR21}, which we briefly recall in the next paragraph.

\myparagraph{The Harvey--Lawson dynamical construction of the Morse--Witten complex.}
Let $M$ be a smooth compact manifold, and let $(f,h)$ be a Morse--Smale pair consisting of a Morse function $f$ and a Riemannian metric $h$. Denote by $\varphi_f^t : M \to M$ the flow generated by the gradient vector field $V = \nabla_h f$, and assume that it satisfies Smale's transversality condition.

Given a differential form $\alpha \in \Omega^\bullet(M)$, define $
u(t) := \big(\varphi_f^{-t}\big)^*\alpha$.
Then $u$ solves the transport equation
\[\begin{cases}
    \partial_t u + \mathcal{L}_V u = 0\\
    u(0)=\alpha
\end{cases}
\]
where $\mathcal{L}_V$ denotes the usual Lie derivative along the vector field $V$.
\par The result of Harvey and Lawson describes the asymptotic behavior of $u(t)$ as $t \to +\infty$, providing an interpolation between smooth differential forms and a finite-dimensional space of de Rham currents defined via Morse theory. More precisely, under the assumption that the metric $h$ is flat in local Morse charts, they proved that, in the sense of currents,
\begin{equation*}
\underbrace{u(t) = \varphi_f^{-t*}\alpha}_{\text{smooth form}}
\;\xrightarrow[t \to +\infty]{}\;
\underbrace{\sum_{a \in \mathrm{Crit}(f)} \left( \int_{W^s(a)} \alpha \right) U_a}_{\text{Morse-theoretic current}}.
\end{equation*}
Here $U_a$ denotes the current of integration over the unstable manifold $W^u(a)$ associated with the critical point $a$. These unstable manifolds are well-defined and have finite volume. The term on the right is a smooth form that collapses, in the limit of large $t$, to a combinatorial object : a distribution supported on $1$-dimensional sub-manifolds.

This result was later extended in~\cite{DR19,DR21} to general Morse--Smale flows, removing the flatness assumption on the metric near critical points. Moreover, these works establish exponential convergence to equilibrium, together with a full asymptotic expansion of $\varphi_f^{-t*}\alpha$ as $t \to +\infty$.

Defining
\[
\Pi_0(\alpha) := \lim_{t \to +\infty} \varphi_f^{-t*}\alpha = \sum_{a \in \mathrm{Crit}(f)} \left( \int_{W^s(a)} \alpha \right) U_a,
\]
one obtains a realization of the Morse--Witten complex in terms of de Rham currents. This perspective was first discovered by Laudenbach~\cite{Laudenbach94} and later recovered by Harvey--Lawson~\cite{HL01}. In particular, if $\alpha \in \Omega^\bullet(M)$ is closed, then $\Pi_0(\alpha)$ defines a current representing the same cohomology class as $\alpha$. 
\par We next move to our second main theorem that adapts these ideas to the dynamical generation of the character variety of a surface.

\myparagraph{A dynamical approach to the combinatorial model of character varieties.}
Consider a compact surface $\Sigma$, a Morse--Smale pair $(f,h)$ on $\Sigma$ where $f$ is a \emph{perfect Morse function}. As before,
$V:=\nabla_h f$ is the gradient vector field on $\Sigma$, and assume that $V$ is $C^1$ linearizable near critical points of $f$. Denote by $\mathsf{Crit}_1(f)$ the saddle points of $f$, that is, critical points of degree $1$. For $a\in \mathsf{Crit}_1(f)$ we denote by $U_a$ the current of integration along the unstable curve $W^u(a)$ associated with a.
\par We use Morse--Smale dynamics to produce a family of gauge transformations that modify a connection and turns it asymptotically into a singular, combinatorial object, as stated --- in a slightly simplified and informal way --- in the following theorem (For the precise statement, see Theorem~\ref{Slicing}). 
\begin{thm}[Morse gauge representation of moduli space of flat connections]
Let $E=\Sigma\times \mathbb{C}^N\rightarrow\Sigma$ be a $G=\mathsf{SU}(N)$-vector bundle of rank $N$~\footnote{Since we are on surfaces and $\mathsf{SU}(N)$ simply connected, any such bundle is trivial.}, and let $\nabla=\mathrm \dd+A$, $A\in \Omega^1\big (\Sigma,\mathfrak{su}(N)\big)$ be a flat connection on $E$.
Define the following twisted transport equation
\begin{equation}
   \begin{cases}
        \partial_tg_t+\mathcal{L}_Vg_t+A(V)g_t=0
        \\
        g_0=1_G
   \end{cases}
\end{equation}
with Cauchy initial data.     
Then, the following holds.
\begin{enumerate}
\item The solution $(g_t)_{t\geqslant 0}$  to (2.1) provides a family of gauge transformations in $C^\infty(\Sigma,G)$ verifying, \emph{in the sense of currents}
\[
\underset{C^\infty \text{ connection}}{ \underbrace{ g_t^{-1}\mathrm d g_t+g_t^{-1}Ag_t}} \ \ 
\xrightarrow[t\to\infty]{}  \ \ 
\underset{ \text{Morse-theoretic flat connection} }{ \underbrace{ A_\infty
=
\sum_{a\in \mathsf{Crit}_1(f)}
\left(\int_{W^s(a)}A\right)U_a}}.
\]
\item If $A_1,A_2\in  \Omega^1\big (\Sigma,\mathsf{SU}(N)\big )$ are gauge equivalent, then
\[
A_{1,\infty}=g^{-1}A_{2,\infty}g
\]
for some group element $g\in \mathsf{SU}(N)$.
\item  The moduli space \[\mathcal M_g:=\Big\{ A\in \Omega^1\big (\Sigma,\mathfrak{su}(N)\big);\ \dd A+A\wedge A=0  \Big\}\Big / C^\infty\big (\Sigma,\mathsf{SU}(N)\big )\] of flat connections modulo gauge admits a description
\[
\mathcal M_g= \Big\{ \mathsf{Hol}_{\partial\mathcal S}\big( (g_a)_{a\in \mathsf{Crit}_1(f)} \big)=1_G  \Big\}\Big/G,
\]
where $\mathsf{Hol}_{\partial\mathcal S} : (g_a)_{a\in \mathsf{Crit}_1(f)}\in G^{2\mathsf{genus}(\Sigma)} \mapsto \mathsf{Hol}_{\partial\mathcal S}\big( (g_a)_{a\in \mathsf{Crit}_1(f)} \big)\in  G   $ is a monomial map and the quotient is taken with respect to the adjoint action of $G$.
\end{enumerate}
\end{thm}

A quick recollection on Morse theory and more details on our geometric settings can be found in Subsection~\ref{ss:slicing}, and the algebraic map $\mathsf{Hol}_{\partial\mathcal S} $ is defined in more detail in Definition \ref{def:combihol}.
 Moreover, we  give a formula for the Atiyah--Bott--Goldman symplectic form for flat connections in Morse gauge as shown in the following proposition. (For more details, see Section~\ref{sss:ABGsymplecticform}).

\begin{proposition}[The ABG symplectic form via Morse gauge]\label{ABGMeasure}
For $a\in \mathsf{Crit}_1(f)$,
denote $[W^u(a)]$ the cycle  corresponding  to $W^u(a)$, viewed an element in  $H_1(\Sigma,\mathbb{Z})$.
Assume $g=(g_a)_{a\in \mathsf{Crit}_1(f) }$ is a regular point of \[\left\{ \mathsf{Hol}_{\partial\mathcal S}(g)=1_G \right\}\]
which we identify as a point in the moduli space $\mathcal{M}_g$.    
Given two elements $(v_a)_{a\in \mathsf{Crit}_1(f)}$ and $(w_a)_{a\in \mathsf{Crit}_1(f)}$ in $\mathfrak{g}^{2\mathsf{genus}(\Sigma)}$ that we can identify with elements in the tangent space $T_{(g_a)_{a\in \mathsf{Crit}_1(f)}}\mathcal{M}_g$, the symplectic form writes:
\begin{equation*}\label{eq:symplecticABG}
\omega_g\left(v,w\right)=\sum_{a,b\in \mathsf{Crit}_1(f)} \mathrm{Tr}\Big( \big\{ \dd\log\left(g_a \right)\left(v_ag_a \right)\big\}\big\{\dd\log\left(g_b \right)(v_bg_b) \big\}   \Big)  \Big( \big[W^u(a)\big]\cap \big[W^u(b)\big]\Big),
\end{equation*}
where the quantity $\big[W^u(a)\big]\cap \big[W^u(b)\big] \in \mathbb{Z} $ is the oriented intersection number of the two cycles $[W^u(a)]$ and  $ [W^u(b)]$. 
\end{proposition}

\subsection{Semi-classical analysis of the Yang--Mills random fields}

\par Recall that $\sigma$ is the area form on $\Sigma$, and that $A_{\Sigma,\sigma}$ is the associated Yang--Mills random field. If we multiply the area form by a parameter $t>0$, we get a random connection $A_{\Sigma,t\sigma}$ that is the Yang--Mills random connection at temperature $t\sigma(\Sigma)$. The third result concerns the study of the small-area limit of the family $(A_{\Sigma,t\sigma})_{t>0}$ (when $t\rightarrow 0^+
$) in the functional spaces $\mathcal{B}_{\mathsf{YM}}^{\alpha,p,s,\ell}$ of the previous section. The following theorem appears also as Theorem~\ref{Thm3}.
\begin{thm}\label{Thm1}
Let $\Sigma$ a compact surface. Consider the case of the group $G=\mathsf{SU}(N)$. For $\alpha\in (\frac{1}{3},\frac{1}{2}),p\geqslant 2,s<0,\ell<-1$, seen as a family of measures on $\mathcal{B}_{\mathsf{YM}}^{\alpha,p,s,\ell}$, $(A_{\Sigma,t\sigma})_{t>0}$ converges  to a measure supported on $\mathcal F$. Moreover, this limit coincides with our representative of the normalized Atiyah--Bott--Goldman symplectic volume measure on $\mathcal F$ of Proposition~\ref{ABGMeasure}.
\end{thm}
\par When $G$ is not simply connected, or less generally, not the special unitary group, geometric considerations become technical, and extra care should be taken. We emphasize that the probabilistic results, however, still hold for general compact Lie groups, but the geometric interpretations become complicated. To avoid geometric technicalities, we suppose that $G=\mathsf{SU}(N)$.
 \par Moreover, Theorems~\ref{MainRes1} and~\ref{Thm1} together define a new class of observables for the Yang--Mills and the ABG measures. Namely, any bounded measurable function \(\mathcal{B}_{\mathsf{YM}}^{\alpha,p,s,\ell} \to \mathbb{R}\) is an admissible observable for the Yang--Mills measure, and, by the convergence theorem~\ref{Thm1}, any continuous function \(\mathcal{B}_{\mathsf{YM}}^{\alpha,p,s,\ell} \to \mathbb{R}\) is an observable for the ABG measure.

\par We mention that this convergence result is also analogous to the beautiful small-temperature limit of Brownian bridges on a Riemannian manifold, studied by Hsu~\cite{Hsu}. There, the semi-classical limit of the Brownian bridge concentrates on geodesics minimizing the kinetic energy. Here, the Yang--Mills measure concentrates on connections minimizing the Yang--Mills energy --- the flat connections--- and the limit is described explicitly in the Morse gauge.  
\par Let us finally clarify how the  small-temperature limit becomes, in the case of Yang--Mills, a small-area limit, or a semiclassical limit, or even a strong-coupling limit. The Yang--Mills action functional is given by  
\[
S_{\mathsf{YM},\sigma}(A)=-\int_\Sigma \mathrm{Tr}\left( F(A)\wedge \star F(A) \right),
\]  
where the subscript $\sigma$ indicates dependence on the area form $\sigma$. The action depends on the Riemannian metric on $\Sigma$ only through the area form. Indeed, the curvature $F(A)$ is a $2$-form, and its Hodge star $\star F(A)$ depends solely on $\sigma$, not on the full Riemannian metric used to define $\sigma$. Concretely, if two Riemannian metrics $h_1,h_2$ induce the same area form $\sigma$ (without necessarily being isometric), then the resulting Yang--Mills functional is identical.

Under a conformal rescaling of the area form, $\sigma \mapsto e^{2\rho} \sigma$ for $\rho\in C^\infty(\Sigma)$, the Hodge star acting on $2$-forms rescales as $\star F(A) \mapsto e^{-2\rho}\star F(A)$. In the equilibrium statistical mechanics interpretation of the Yang--Mills measure, the Gibbs measure describing the canonical ensemble is  
\[
\exp\big({-\beta S_{\mathsf{YM},\sigma}(A) } \big) \mathcal{D}A,
\]  
where $\beta$ is interpreted as an inverse temperature $\beta = \frac{1}{k_B T}$ ($T$ the temperature, $k_B$ Boltzmann's constant). In field theory language, $\beta$ can also be viewed as a coupling constant or as an inverse semiclassical parameter $1/\hbar$.
Under this rescaling we have
\[
S_{\mathsf{YM},e^{2\rho}\sigma}(A) = -\int_\Sigma \mathrm{Tr}\left( F(A)\wedge \star_{e^{2\rho}\sigma} F(A) \right) = e^{-2\rho} S_{\mathsf{YM},\sigma}(A),
\]  
and the effective coupling becomes $e^{-2\rho}\beta$, which tends to $+\infty$ as $\rho\to -\infty$. Consequently, for two dimensional Yang--Mills, the small-area limit is equivalent to the strong-coupling, semiclassical, and small-temperature limits.
\subsection{Organization of the paper}
The paper is organized as follows. Section \ref{s2} reviews the moduli space of flat connections, presents its Morse-gauge description, and in this setting, computes its canonical symplectic form. 
Section \ref{SpacesBYM} introduces the anisotropic Banach spaces of distributional connections and proves the fundamental decomposition theorem into noise and flat components. 
Section \ref{s3} recalls the construction of the Yang--Mills random connection on compact surfaces appearing in \cite{DN,BCDRT}, and establishes that the Yang--Mills measure can be realized as a Borel probability measure on the Banach spaces introduced in \ref{SpacesBYM}.
Section \ref{semiclas} concludes the proof of the small area limit of the Yang--Mills measure. This goes into two steps. First, the tightness of the family $(A_{\Sigma,t\sigma})_{t>0}$. Second, the identification of the limiting measure with our representative of the normalized
Atiyah–Bott–Goldman measure.

\section{The moduli space of flat connections in the Morse gauge} \label{s2}
This section addresses the geometric core of the paper. Our goal is to give an explicit description of the space of flat connections modulo gauge transformations, together with its natural symplectic structure, using the Morse gauge.

We begin by recalling the definition of flat connections, the moduli space of flat connections modulo gauge, and the character variety of a compact surface. Next, we show how the Morse gauge yields an identification between the last two. This involves running a certain dynamics on smooth flat connections to transform them into combinatorial objects containing the finite-dimensional information needed to identify the gauge equivalence class of the connection.

Our next goal is to investigate the symplectic structure of this moduli space using the Morse gauge. To that end, we need to define tangent spaces to this set and express the symplectic $2$-form.

We then turn to the symplectic volume measure induced by this symplectic form. In this direction, two types of points arise: regular and singular. For genus \(g \geq 2\), it is known --- and we recall this fact --- that the set of singular points has zero symplectic volume; hence, understanding the symplectic volume essentially reduces to understanding it on regular points. This is precisely where the notion of \emph{torsion} enters: as we will show, torsion provides a link between the geometry of regular points and the global volume form.

For \(g = 1\), however, every point is singular in a precise sense, requiring a separate, dedicated treatment. This is kept to the end of this section.

\subsection{General definitions}
Let $\Sigma$ be a smooth, closed, compact surface of genus $g$, and let $G=\mathsf{SU}(N)$. Given a $G$-vector bundle $E \to \Sigma$, the gauge group 
$\mathcal{G} \coloneqq C^\infty(\Sigma, G)$ acts by automorphisms on the vector bundle $E$, i.e., by linear invertible maps on each fiber. 
Smooth sections of $E$ are denoted by $C^\infty(\Sigma, E)$, and $E$-valued differential forms of degree $k$ are denoted by $\Omega^k(\Sigma, E)$.
A smooth connection on $E$ is a linear differential operator $\nabla: C^\infty(\Sigma, E) \to \Omega^1(\Sigma, E)$ satisfying the Leibniz rule (or Koszul rule):
\[
\forall f \in C^\infty(\Sigma, \mathbb{C}),\; \forall s \in C^\infty(\Sigma, E), \qquad 
\nabla(fs) = \mathrm{d}f \otimes s + f \nabla s \in \Omega^1(\Sigma, E).
\]
The space of connections $\mathcal{A}$ is an affine space modeled on $\Omega^1\big(\Sigma,\operatorname{End}(E)\big)$. 
Two connections $\nabla_1, \nabla_2$ on the vector bundle $E \to \Sigma$ are said to be \emph{gauge equivalent} if $\nabla_1 = g^{-1} \nabla_2 g$ for some $g \in \mathcal{G}$. 
The orbit under the gauge group of a smooth connection $\nabla \in \mathcal{A}$ is denoted by $\mathcal{G} \cdot \nabla$ and is given by 
\[
\mathcal{G} \cdot \nabla \coloneqq \{ g^{-1} \nabla g : g \in C^\infty(\Sigma, G) \}.  
\]
The square $\nabla \circ \nabla: C^\infty(\Sigma, E) \to \Omega^2(\Sigma, E)$ acts on sections as multiplication by an element $F \in \Omega^2\big(\Sigma, \operatorname{End}(E)\big)$ called the curvature of $\nabla$. In homological terms, we may think of the curvature as measuring the obstruction of 
\[
\Omega^\bullet(\Sigma, E) \xrightarrow{\nabla} \Omega^{\bullet+1}(\Sigma, E) \xrightarrow{\nabla} \Omega^{\bullet+2}(\Sigma, E)
\]
to being a cochain complex. A connection $\nabla$ is called \emph{flat} if its curvature vanishes, \textit{i.e.}, $F = 0$.
In particular, a flat connection defines a twisted de Rham cochain complex
\[
\Omega^\bullet(\Sigma, E) \xrightarrow{d^\nabla} \Omega^{\bullet+1}(\Sigma, E) \xrightarrow{d^\nabla} \Omega^{\bullet+2}(\Sigma, E)
\]
that we later denote by $\big(\Omega^\bullet(\Sigma, E), d^\nabla \big)$, where $d^\nabla \coloneqq \nabla$ is the twisted de Rham differential. We denote by $\mathcal{M}_g$ the moduli space of flat connections modulo the action of the gauge group $\mathcal{G}$
\[
\mathcal{M}_g = \{ \nabla : F(\nabla) = 0 \} / \mathcal{G}.
\] 
We emphasize that for the moment we do not discuss any topological issues, which are quite subtle. It is well known that these quotients are not well behaved in this direction, for example at reducible connections.
\par The fact that $G$ is \emph{connected and simply connected} ensures that all flat $G$-bundles $E$ are trivial. 
Given a flat connection $\nabla$ on $E$,
the moduli space $\mathcal{M}_g$ can be identified 
with the moduli space $\operatorname{Hom}\big(\pi_1(\Sigma), G\big)/G$ of representations of 
the surface group $\pi_1(\Sigma)$, modulo the adjoint action of $G$~\cite[Thm 13.2 p.~159]{Taubes} and \cite[Thm 2.9 p.~56]{Morita}. One direction of this correspondence can be easily described. In fact, for every based closed loop $\gamma$ representing a given homotopy class, denote by $\rho_\nabla(\gamma) \in G$ the holonomy of $\nabla$
along $\gamma$. Under this identification, 
the map $\rho_\nabla$ sends $\pi_1(\Sigma)$ to the group $G$.
Concretely, the fundamental group of $\Sigma$, which we denote by $\pi_1(\Sigma)$ and call the \emph{surface group}, has the usual presentation as the group generated by $2g$ generators with one relation:
\[
\pi_1(\Sigma) \coloneqq \big\langle a_1, b_1, \dots, a_g, b_g : [a_1, b_1] \cdots [a_g, b_g] = 1_G \big \rangle.
\] 
Then the character variety $\operatorname{Hom}\big(\pi_1(\Sigma), G\big)$ is defined
as
\[
\mathcal{M}_g \coloneqq \{ (a_1, b_1, \dots, a_g, b_g) \in G^{2g} : [a_1, b_1] \cdots [a_g, b_g] = 1_G \} / G,
\]
where we quotient by the adjoint action of $G$ defined by
\[
g \cdot (a_1, b_1, \dots, a_g, b_g) \coloneqq \big(g^{-1} a_1 g,\; g^{-1} b_1 g,\; \dots,\; g^{-1} a_g g,\; g^{-1} b_g g\big).    
\]
Later we will use Morse theory to suggest, and prove, that the surface group admits another similar representation. 

\subsection{Slicing the space of flat connections via the Morse gauge}
\label{ss:slicing}
In this subsection, we will give yet another model of the character variety, relying on the Morse gauge introduced in~\cite{BCDRT,DN}. We start by recalling some facts about Morse theory that will be used throughout the paper.
\myparagraph{Recollection on Morse functions.}
A smooth function $f: \Sigma \to \mathbb{R}$ is a Morse function if it has only finitely many critical points and these critical points are nondegenerate. Denote by $V = \nabla_h f$ the gradient vector field of $f$ with respect to some Riemannian metric $h$ on $\Sigma$, which is chosen in such a way that 
$h$ is flat in the Morse coordinates --- given by the Morse Lemma --- near all critical points. We denote by $\varphi^t_f: \Sigma \to \Sigma$ the corresponding flow.
For every critical point $a \in \mathsf{Crit}(f)$, we define the stable and unstable manifolds of $a$
(check Figure~\ref{fig:Morse_flow}) as \label{StableUnstable}
\[
W^{s/u}(a) = \{ x \in \Sigma : \lim_{t \to \pm\infty} \varphi_f^{t}(x) = a \}.     
\]
In the case of surfaces, when the critical point $a$ has Morse index $1$ --- in which case $a$ is called a saddle point --- the corresponding stable and unstable manifolds are $1$-dimensional curves. The set of saddle points will be denoted by $\mathsf{Crit}_1(f)$ \label{Crit}.

\begin{figure}
    \centering
    \includegraphics[width=\linewidth]{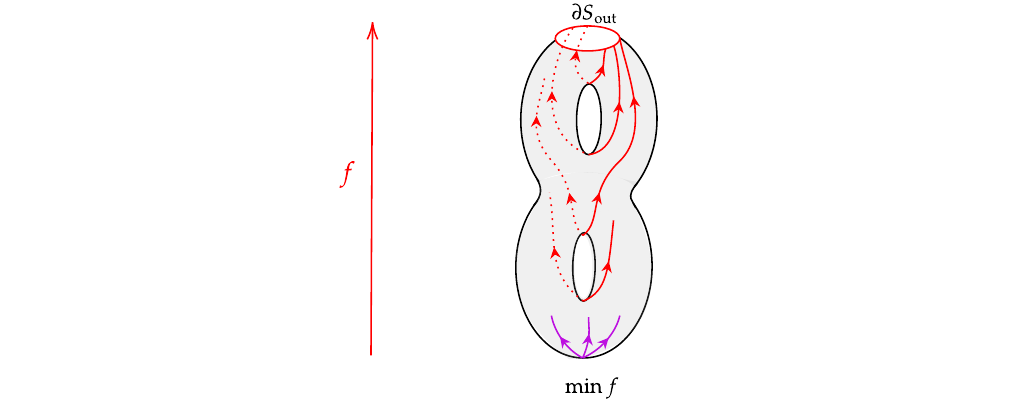}  
    \caption{Morse flow on $\Sigma$ and blow-up at $\max(f)$.}
    \label{fig:Morse_flow}
\end{figure}

For the sequel, we assume the following: 
\begin{itemize}
\item The Morse function $f$ is a \emph{perfect Morse function}, meaning that it has exactly one minimum, one maximum, and $2g$ saddle points.
\item The pair $(f, h)$ consisting of the Morse function and the Riemannian metric satisfies Smale's transversality condition~\cite[2.2.b p.~38]{AudinDamian}. This means that every unstable curve $W^u(a)$ corresponding to a saddle point $a$ is the union of two flow lines connecting $a$ and $\max(f)$.
\item The vector field $V = \nabla_h f$ is linear in Morse charts near every critical point. 
\end{itemize}

We refer the reader to the introductory sections in ~\cite{BCDRT,DN} for more recollection on the needed background in Morse theory.

\myparagraph{Morse-theoretic interpretation of surface group representations.}
\label{sss:morsemoduli}

Morse theory provides an alternative interpretation of representations of the surface group. Fix a Morse function $f$ verifying the assumptions of the previous paragraph. Let $\mathcal{S}$ be the surface with boundary obtained by blowing up $\Sigma$ at the maximum of $f$ and replacing it with a disk. Consider a singular connection of the form $\dd + A$, where
\[
A = \sum_{a \in \mathsf{Crit}_1(f)} \log(g_a) \, U_a,
\]
and the $U_a$ are currents of integration along the unstable curves. We will introduce the notion of combinatorial holonomy of such singular connection, and we will see later that, up to conjugation by $G$, each such singular connection encode the information about one and only equivalence class of flat connections.
\par Parametrize the boundary circle $\partial\mathcal{S}$ by an angle $\theta \in [0, 2\pi]$. Let $0 \leq \theta_1 < \dots < \theta_{4g} \leq 2\pi$ be the angles at which the $2g$ unstable curves $W^u(a)$ (with $\operatorname{ind}(a) = 1$) intersect $\partial\mathcal{S}$. Each unstable curve meets $\partial\mathcal{S}$ exactly twice. Orient each unstable curve $W^u(a)$ for $a \in \mathsf{Crit}_1(f)$. For each $i \in \{1, \dots, 4g\}$, we associate a critical point $a_{k_i} \in \mathsf{Crit}_1(f)$ such that $W^u(a_{k_i})$ intersects $\partial\mathcal{S}$ at time $\theta_i$, and we set $\varepsilon_i = 1$ if the curve is outgoing and $\varepsilon_i = -1$ if it is incoming.

This data yields a product of group elements
\[
g_{a_{k_{4g}}}^{\varepsilon_{4g}} \cdots g_{a_{k_2}}^{\varepsilon_2} g_{a_{k_1}}^{\varepsilon_1},
\]
where each $a_{k_i} \in \mathsf{Crit}_1(f)$ and each element $g_a$ appears exactly twice in the product, with opposite signs, corresponding to the two intersections of $W^u(a)$ with $\partial\mathcal{S}$.

\begin{definition}[Combinatorial holonomy]\label{def:combihol}
The combinatorial holonomy of the flat singular connection
\[
\sum_{a \in \mathsf{Crit}_1(f)} \log(g_a) \, U_a
\]
around $\partial\mathcal{S}$ is given by
\[
\mathsf{Hol}_{\partial\mathcal{S}}\!\left[\sum_{a \in \mathsf{Crit}_1(f)} \log(g_a) \, U_a\right]
= g_{a_{k_{4g}}}^{\varepsilon_{4g}} \cdots g_{a_{k_2}}^{\varepsilon_2} g_{a_{k_1}}^{\varepsilon_1}.\]
\end{definition}

More generally, let $\gamma: [0,1] \to \Sigma$ be any closed curve that is everywhere transverse to the vector field $V = \nabla f$. Denote by $t_1 < \dots < t_k$ the ordered times at which $\gamma$ intersects the union $\bigcup_{a \in \mathsf{Crit}_1(f)} W^u(a)$ of unstable curves. For each $i \in \{1, \dots, k\}$, let $a_i \in \mathsf{Crit}_1(f)$ be the saddle point whose unstable curve $W^u(a_i)$ meets $\gamma$ at time $t_i$, and define $\varepsilon_i = \int_\Sigma [\gamma] \wedge [U_{a_i}] = \pm 1$ depending on the orientation number. Then the holonomy $\mathsf{Hol}_\gamma(A)$ is defined as
\[
\mathsf{Hol}_\gamma(A) \coloneqq g_{a_{i_k}}^{\varepsilon_{i_k}} \cdots g_{a_{i_1}}^{\varepsilon_{i_1}}.
\]
The first theorem of this article specializes a result from~\cite{BCDRT} to the moduli space of flat connections. It gives a complete characterization of the set of flat connections modulo gauge using the combinatorial holonomy map.

\begin{thm}[Slicing flat connections with a Morse gauge]\label{Slicing}
Let $E = \Sigma \times \mathbb{C}^N \to \Sigma$ be a $G = \mathsf{SU}(N)$-vector bundle of rank $N$, and let $\nabla = \dd + A$ with $A \in \Omega^1(\Sigma, \mathfrak{su}(N))$ be a flat connection on $E$. Consider the twisted transport equation
\begin{equation}\label{eq:twistedtransport1}
    \partial_t g_t + \mathcal{L}_V g_t + A(V) g_t = 0,
\end{equation}
with initial condition $g_0 = 1_G$. Then the following hold.

\begin{enumerate}
\item The family of solutions $(g_t)_{t \geq 0}$ to (2.1) yields a family of gauge transformations in $C^\infty(\Sigma, G)$ such that
\[
\underbrace{g_t^{-1} \dd g_t + g_t^{-1} A g_t}_{\text{in $C^\infty$ connection}} \;\xrightarrow[t \to \infty]{}\; 
\underbrace{A_\infty \coloneqq \sum_{a \in \mathsf{Crit}_1(f)} \left( \int_{W^s(a)} A \right) U_a}_{\text{Morse-theoretic flat connection}},
\]
where convergence holds in $\mathcal{D}^{\prime,1}_{\Gamma_u}(\mathcal{S})$, and each $U_a$ (for $a \in \mathsf{Crit}_1(f)$) is the current of integration along the unstable curve $W^u(a)$.

\item If $A_1, A_2 \in \Omega^1(\Sigma, \mathfrak{su}(N))$ are gauge equivalent, then
\[
\exp\left(A_{1,\infty}\right) = g^{-1}\exp\left( A_{2,\infty}\right) g
\]
for some fixed group element $g \in \mathsf{SU}(N)$. This means that the holnomy representations induced by the two connections are conjugated by $g\in G$.

\item The moduli space of flat connections modulo gauge,
\[
\mathcal{M}_g := \left\{ A \in \Omega^1(\Sigma, \mathfrak{su}(N)) : \dd A + A \wedge A = 0 \right\} \big/ C^\infty(\Sigma, \mathsf{SU}(N)),
\]
admits the description
\[
\mathcal{M}_g = \left\{ \mathsf{Hol}_{\partial\mathcal{S}}\!\big( (g_a)_{a \in \mathsf{Crit}_1(f)} \big) = 1_G \right\} \big/ G,
\]
where the quotient is taken with respect to the adjoint action of $G$.
\end{enumerate}
\end{thm}

In particular, this theorem recovers the classical results on character varieties in a purely analytical way.
\begin{proof}
We adopt the dual point of view and start from the canonical orthonormal moving frame $(\mathbf{e}_i)_{i=1}^N$ of the trivial $G$--bundle $E$ of rank $N$. The flat connection $\nabla$ induces a twisted Lie derivative $\mathcal{L}_V^\nabla \coloneq  [\nabla,\iota_V]: C^\infty(\Sigma,E)\rightarrow C^\infty(\Sigma,E) $ acting on smooth sections of $E$. Recall that given $\nabla$, the connection $1$--form in a given moving frame $(\mathbf{e}_i)_{i=1}^N$ is defined as $ \left\langle \mathbf{e}_j, \nabla \mathbf{e}_i \right\rangle=A_{ji}\in \Omega^1(\Sigma)$, $(i,j)\in \{1,\dots,N\}^2$.
Then, we use the flow twisted by the flat connection $\nabla$ to deform the moving frame as follows.
$$ \partial_t \mathbf{e}_i(t)=-\mathcal{L}_V^\nabla \mathbf{e}_i(t), \,\ \mathbf{e}_i(0)=\mathbf{e}_i. $$ 
Define $g_t\in C^\infty(\Sigma,G)$ by $\mathbf{e}_i(t)=g_{t,i}^j\mathbf{e}_j $ for all $t\geqslant 0$.
We verify that $g_t$ satisfies the transport equation~\ref{eq:twistedtransport1}.
By construction 
\[\partial_t \mathbf{e}_i(t)=\partial_t g_{t,i}^j \mathbf{e}_j=-\mathcal{L}_V^\nabla g_{t,i}^j \mathbf{e}_j= -\mathcal{L}_Vg_{t,i}^j\mathbf{e}_j -g_{t,i}^kA(V)_k^j \mathbf{e}_j\]
since the moving frame $(\mathbf{e}_j)_{j=1}^N$ is constant which implies that $g_t$ satisfies the transport equation~\ref{eq:twistedtransport1}.

The initial connection $\nabla$ viewed in the moving frame $(\mathbf{e}_i(t))_{i=1}^N$
has connection $1$--form
\begin{align*}
 A_{ji}(t)=(g_t^{-1}\dd g_t)_{ji}+\big(g_t^{-1}A g_t\big)_{ji}.
\end{align*}
Moreover, using the fact that $\nabla$ is a flat unitary connection, we also find:
\begin{align*}
    A_{ji}(t)&= \left\langle \mathbf{e}_j(t),\nabla \mathbf{e}_i(t)\right\rangle=
\left\langle e^{-t\mathcal{L}_V^\nabla} \mathbf{e}_j,\nabla e^{-t\mathcal{L}_V^\nabla} \mathbf{e}_i\right\rangle\\
&=\left\langle e^{-t\mathcal{L}_V^\nabla} \mathbf{e}_j, e^{-t\mathcal{L}_V^\nabla}\left(\nabla \mathbf{e}_i \right) \right\rangle=\big(\varphi_f^{-t}\big)^*\big\langle  \mathbf{e}_j,\nabla  \mathbf{e}_i\big\rangle=\big(\varphi_f^{-t}\big)^*A_{ji}.
\end{align*}
Then the identity
\[
\big(g_t^{-1}\dd g_t\big)_{ji}+\big(g_t^{-1}A g_t\big)_{ji}=    \varphi_f^{-t*}A_{ji}
\]
converges at the limit when $t\rightarrow +\infty$ 
to the current $\sum_{a\in \mathsf{Crit}_1(f)} \left(\int_{W^s(a)} A\right) U_a $ by the convergence result of \cite{DR19}. Now assume we have started from another gauge equivalent flat connection $g^{-1}\nabla g$ for some arbitrary gauge transformation $g\in C^\infty(\Sigma,G)$.

Let us prove the second item on gauge equivalence. Assume we are given two flat connections $\nabla_1$ and $\nabla_2=u^{-1} \nabla_1 u$ for $u\in C^\infty(\Sigma)$ which are gauge equivalent.
Assume we are given two connection $1$--forms $A$ and $u^{-1}du+u^{-1}Au$ s.t. 
$u\in C^\infty(\Sigma,G)$.
Observe that for any $\gamma:[0,1]\mapsto \Sigma$, for any $A\in \Omega^1(\Sigma,\mathfrak{g})$, $\mathsf{P}_{\gamma}(\varphi_f^{-t*}A)=\mathsf{P}_{\varphi_f^{-t}\left(\gamma\right)}(A)$ where $\mathsf{P}_\gamma(A)$ is the solution $v(1)$ to the ODE 
$$ \frac{\dd v}{\dd }=A(\gamma(t),\gamma^\prime(t))v(t); v(0)=1_G .$$ 
Also for any  
 $A\in \Omega^1(\Sigma,\mathfrak{g})$ and $u\in C^\infty(\Sigma,G)$,
 \begin{align*}
 \mathsf{P}_{\gamma} \left( u^{-1}du+u^{-1}Au \right)=  u(\gamma(1))\mathsf{P}_{\gamma}(A)u(\gamma(0))^{-1}. 
 \end{align*}

Now for every saddle point $a\in \mathsf{Crit}_1(f)$, we combine the two above
identities choosing $\gamma$ to be the closure of stable curve $\overline{W^s(a)}$ parametrized in such a way that $\gamma(0)=\gamma(1)=\min(f)$.
This gives
\begin{align*}
\mathsf{Hol}_{\overline{W^s(a)}}\left( \varphi^{-t*}_f(  u^{-1}du+u^{-1}Au ) \right)&=\mathsf{Hol}_{\varphi^{-t}_f(\overline{W^s(a))}}\left(  u^{-1}du+u^{-1}Au  \right)\\
&= \mathsf{Hol}_{\overline{W^s(a) } }\left(  u^{-1}du+u^{-1}Au  \right)\\
&=u(\min(f))  \mathsf{Hol}_{\overline{W^s(a)}}\left( A\right)u(\min(f))^{-1},
\end{align*}
where the equalities hold true for all $t\geqslant 0$.
From this we deduce the second claim.

Let us move to the proof of the third item: namely that the Morse gauge provides an identification with the character variety. Start from a representation $\rho:\pi_1(\Sigma)\mapsto G$ of the surface group. By~\cite[Thm 13.2 p.~159]{Taubes} and \cite[Thm 2.9 p.~56]{Morita}, there always exists a smooth flat connection $\nabla$ s.t. the Holonomy 
representation $\rho_\nabla$ coincides with $\rho$ up to the adjoint action of $G$ (just apply the result of \cite[top of p.~60]{Morita} which gives a flat principal $G$-bundle realizing the given element $\rho\in \mathsf{Hom}(\pi_1(\Sigma),G)$ then pass to the vector bundle picture by using the usual representation of $G=\mathsf{SU}(N)$ on $\mathbb{C}^N$). The smooth connection $\nabla$ reads $\dd+A$. Then applying the deformation argument from the first part to $\dd+A$ produces a $1$--parameter family $\nabla_t:= \dd+A_t$ of flat connections such that
for every $t$, $\rho_{\nabla_t}\simeq\rho_{\nabla} $ up to conjugation and the limiting connection $1$-form $A_\infty$
has the form
$\sum_{a\in \mathsf{Crit}_1(f)} \log(g_a)U_a $. The fact that both $\rho$ et $\rho_{\nabla_\infty}$ are conjugated comes from the second item.
\end{proof}

Finally, since the combinatorial holonomy map from Definition~\ref{def:combihol}:
\[
[\gamma] \in \pi_1(\Sigma) \longmapsto \mathsf{Hol}_{\gamma}\!\left[ \sum_{a \in \mathsf{Crit}_1(f)} \log(g_a) \, U_a \right] \in G,
\]
defines a representation of the surface group $\pi_1(\Sigma)$, we can freely use the results of Goldman~\cite{Goldman1, Goldman2}, which are formulated in this language.

\subsection{The  ABG symplectic form via the Morse gauge}
\label{sss:ABGsymplecticform}
Now that we have given a realization of $\mathcal{M}_g$, we will define the tangent space at a point of $\mathcal{M}_g$. This will let us explicit a formula for the Atiyah--Bott--Goldman (ABG) symplectic form\footnote{We thank Vladimir Fock for very interesting discussions on how our expression 
for the symplectic form is closely related to the Poisson structures appearing in his work with Rosly~\cite{FockRosly}; although the point of view there emphasizes more the Poisson bracket.}.
\par The moduli space reads \label{ModuSpace}
$$\mathcal{M}_g=\left\{\mathsf{Hol}_{\partial\mathcal{S}}\Bigg[\sum_{a\in \mathsf{Crit}_1(f)} \log(g_a)U_a \Bigg] = g_{a_{k_{4g}}}^{\varepsilon_{4g}} \dots  g_{a_{k_2}}^{\varepsilon_2} g_{a_{k_1}}^{\varepsilon_1}=\text{Id}_G\right\} \Bigg/G $$ 
where we quotient by the adjoint action of $G$.
Consider an element $(g_a)_{a\in \mathsf{Crit}(f)_1}$ of $\mathcal{M}$.  Following Goldman~\cite[paragraph 1.2 p.~203]{Goldman1}, we
choose any smooth curve 
\[
t \in (-1, 1) \mapsto \big( g_a(t) \big)_{a \in \mathsf{Crit}_1(f)} \in \mathcal{M}_g
\]
such that $g_a(0) = g_a$. Then we view 
$\left( \frac{d g_a}{\dd t}(0) \right)_{a \in \mathsf{Crit}_1(f)}$ as an element of $T_{(g_a)_{a \in \mathsf{Crit}_1(f)}} \mathcal{M}_g$.
Each vector $\frac{d g_a}{\dd t}(0)$ belongs to $T_{g_a} G$, which we identify with 
$v_a g_a$ for some $v_a \in \mathfrak{g}$. Now, differentiating the 
curve 
\[
t \mapsto \sum_{a \in \mathsf{Crit}_1(f)} \log\big( g_a(t) \big) \, U_a 
\]
in $\mathcal{D}^{\prime,1}(\Sigma) \otimes \mathfrak{g}$, we obtain
\[
\frac{\dd}{\dd t}\Bigg|_{t=0} \sum_{a \in \mathsf{Crit}_1(f)} \log\big( g_a(t) \big) \, U_a
= \sum_{a \in \mathsf{Crit}_1(f)} \dd \log\big( g_a(0) \big)(v_a g_a(0)) \, U_a.
\] 
We may therefore write the following definition.
\begin{definition}[Tangent space to $\mathcal{M}_g$]\label{TangCone}
    Let $(g_a)_{a\in\mathsf{Crit}_1(f)}\in\mathcal{M}_g$. The tangent space to $\mathcal{M}_g$ is given at $(g_a)_{a\in\mathsf{Crit}_1(f)}$ is given by 
    \[
\bigl( (v_a)_{a\in\mathsf{Crit}_1(f)} \bigr) \in T_{(g_a)_{a\in\mathsf{Crit}_1(f)}}\mathcal{M}_g
\qquad\Longleftrightarrow\qquad
\begin{cases}
(v_a)_{a \in \mathsf{Crit}_1(f)} \in \mathfrak{g}^{2g} \\[4pt]
\bigl( v_a g_a \bigr)_{a \in \mathsf{Crit}_1(f)} \in \ker\,\!\Bigl( \dd\,\mathsf{Hol}_{\partial\mathcal{S}}\bigl( (g_a)_{a \in \mathsf{Crit}_1(f)} \bigr) \Bigr).
\end{cases}
\]
\end{definition}
Let us now express the ABG symplectic form. Note that this definition only makes sense modulo the choice of a logarithm. Since the final goal is to identify the small area limit of the Yang--Mills random field with the symplectic volume induced by the ABG symplectic form, all we need is to be able to define this form almost everywhere, which is the case here.
\begin{proposition}[ABG symplectic form in Morse gauge] \label{ABGSymp}
  Given two tangent vectors 
  \[v=(v_a)_{a\in  \mathsf{Crit}(f)_1}, w=(w_a)_{a\in  \mathsf{Crit}(f)_1} \in T_{(g_a)_{a\in  \mathsf{Crit}(f)_1}}\mathcal{M}_g,\] 
  the symplectic form writes
\begin{equation}\label{eq:symplecticABG}
\omega\left(v,w\right)=\sum_{a,b\in \mathsf{Crit}(f)_1} \mathrm{Tr}\Big(  \dd\log\left(g_a \right)\left(v_ag_a \right)\  \dd\log\left(g_b \right)(v_bg_b)   \Big)  \Big( [W^u(a)]\cap [W^u(b)] \Big),
\end{equation}
where $[W^u(a)]\cap [W^u(b)] \in \mathbb{Z} $ denotes the oriented intersection numbers of the two cycles $[W^u(a)]$,  $ [W^u(b)]$ viewed as elements in $H_1(\Sigma,\mathbb{Z})$.  
\end{proposition}
\begin{proof}
    It is an adaptation of the computation carried in the beginning of this paragraph. Moreover, 
the fact that the above form is closed follows from \cite[Thm 1.7 p.~208]{Goldman1} and the fact that Morse theory gives us a way to represent $\operatorname{Hom}(\pi_1(\Sigma),G)$.
\end{proof}

\subsection{The symplectic volume}
\label{sss:ABGvol}
The purpose of this paragraph is to make the first link between the symplectic volume induced by the symplectic form $\omega$, and the low temperature limit of the Yang--Mills measure. The fundamental result that we would like to have is Theorem ~\ref{YMABG}. This theorem, and its proof deal with two separate cases, the case $g=1$ and the case $g\geq 2$. Note that for $g=0$ the treatment is trivial.
\par For the case $g\geq 2$, this result was proved by Forman, and Kefeng Liu in ~\cite{Forman, LiuHeat1}, based on ideas from Witten. Here, we recall the main steps and ingredients of this proof.  
\par For the case $g=1$, it was done by Sengupta~\cite{SenguptaTorus}.  We dedicate a special section for this case where we recall the proof of Sengutpa's result. 

\begin{thm}\label{YMABG}

  Let $N=2(g-1)\dim(G)$ if $g\geq 2$, $N=2\sqrt{\dim G+1}-2$ if $g=1$, and $N=0$ if $g=0$. Define $\dd \nu=Z^{-1} \omega^{\wedge^\frac{N}{2}}$, where $Z$ is a normalization constant that turns $\nu$ into a probability measure. Let 
    \[Z_\epsilon(\Sigma)=\int_{G^{2g}}p_{\epsilon}\left(\mathsf{Hol}_{\partial S}\left[\sum_{a\in\mathsf{Crit}_1(f)}\log g_a U_a\right]\right)\dd g.\]
    Then, there exists a measure $\tilde{\nu}$ supported on $\mathsf{Hol}_{\partial \mathcal{S}}^{-1}(1_g)$ such that
    \[\dd \tilde{\nu}(g_1,\cdots,g_{2g}) =\lim_{\epsilon\to 0} \frac{ p_{\epsilon}\left(\mathsf{Hol}_{\partial S}\left[\sum_{a\in\mathsf{Crit}_1(f)}\log g_a U_a\right]\right)}{Z_\epsilon(\Sigma)}\bigotimes_{a\in\mathsf{Crit}_1(f)}\dd g_a\]
        in the sense of weak convergence of probability measure. Moreover, the measure $\tilde{\nu}$ descends to the quotient space $\mathsf{Hol}_{\partial \mathcal{S}}^{-1}(1_g)/G$, and the quotient measure is equal to $\nu$. 
\end{thm}

The need to separate into two cases comes from singular points. In fact, for $g \geq 2$, there are only few singular points, and the set of singular points has measure $0$. This lets us use usual techniques of change of variables using the almost everywhere submersion $\mathsf{Hol}_{\partial \mathcal{S}}$. However, for the case $g = 1$, all the points are singular in the sense we will shortly define (see Definition~\ref{SingPoint}), and special treatment is required. 
\par Another difference between these two cases is that the partition function that we denoted by $Z_{\epsilon}(\Sigma)$ in the previous theorem diverges for $g=1$ and converges for $g\geq 2$. In fact, For $G$ a compact semisimple Lie group, the Witten zeta function is defined as:
\[
    \zeta_G(s) \coloneq  \sum_{\rho\in \widehat{G}}\frac{1}{\dim(V_\rho)^{s}}.
\]
It follows from   \cite[Theorem 15.7, Equation 15.3.2]{Komori_Matsumoto_Matsumoto} that for a compact connected
semi-simple Lie group $G$, $\zeta_G(2k)\in \mathbb{Q}\pi^{2k p}, \forall k\in \mathbb{N}^*$ where $p$ denotes the number of positive roots of $Lie(G)$. This implies that $\zeta_G(2g-2)<+\infty$ for $g\geq 2$. Therefore, the partition function associated with the area function $\epsilon\sigma $ is equal to
\[ Z_\epsilon(\Sigma)= \sum_{\rho\in \widehat{G}} \exp\big(-\varepsilon\sigma(\Sigma) c_2(\rho)\big) \frac{1}{d_\rho^{2g-2}}
\]
and converges, when $\varepsilon\rightarrow 0^+$, to the Witten zeta function $\zeta_G(2g-2)$. When $g\leq 1$, $\zeta_G(2g-2)=\infty$, and the partition function does not converge as $\epsilon\to0$, and singular behavior is expected in this case. 
\par Let us first discuss singular points of $\mathcal{M}_g$.

\subsubsection{A discussion about singular points}
\label{sss:recollectionsingular}
Observe that Definition~\ref{TangCone} does not distinguish between singular and regular points of the moduli space $\mathcal{M}_g$. This is why we used the term ``tangent cone'' instead of ``tangent space''; Goldman uses the terminology \emph{Zariski tangent space}. We refer the reader to~\cite[2.3 p.~31-32]{MaretCharacter} for a very clear exposition of the Zariski tangent space. Of course, at regular points of $\mathcal{M}_g$, the two notions coincide.

Let us recall precisely the notion of singular points as stated in the work of Goldman. The Zariski tangent space at $(g_a)_{a \in \mathsf{Crit}_1(f)}$, as defined above, has dimension
\[
(2g - 1) \dim(G) + \dim Z_{(g_a)_{a \in \mathsf{Crit}_1(f)}},
\]
where $Z_{(g_a)_{a \in \mathsf{Crit}_1(f)}}$ denotes the centralizer of the corresponding representation of the surface group~\cite[Proposition on top of p.~204]{Goldman1}. Concretely, the centralizer consists of all group elements $h \in G$ that commute with every element
\[
\mathsf{Hol}_{\gamma}\!\left[\sum_{a \in \mathsf{Crit}_1(f)} \log(g_a) \, U_a \right],
\]
where $[\gamma]$ runs over all free homotopy classes in $\pi_1(\Sigma)$. The following definition of regular points is taken from ~\cite[line 9, p.~204; and equation (1.4) p.~204]{Goldman1}.
\begin{definition}[Regular points.]\label{SingPoint}
    A regular point of $g=(g_a)_{a\in\mathsf{Crit}_1(f)}\in \mathsf{Hol}_{\partial\mathcal{S}}^{-1}(1_G)$ corresponds to an element $(g_a)_{a \in \mathsf{Crit}_1(f)}$ for which $\dim Z_{(g_a)_{a \in \mathsf{Crit}_1(f)}}$ coincides with the dimension $\dim Z(G)$ of the center $Z(G)$ of $G$.
\end{definition}
 In fact, in the third paragraph of p.~204 in~\cite{Goldman1}, Goldman describes the stratification of the character variety by successively taking centralizers of centralizers of representations.
In~\cite[line 8 p.~46]{Forman}, the singular points of $\mathcal{M}_g$ are defined to be the elements $(g_a)_{a \in \mathsf{Crit}_1(f)} \in \mathsf{Hol}_{\partial\mathcal{S}}^{-1}(\operatorname{Id}_G)$ that are \emph{critical points} of the map $\mathsf{Hol}_{\partial\mathcal{S}}$; the regular points are the smooth points in the sense of~\cite[def. 2.4.1 p.~36]{MaretCharacter}. For $g\geqslant 2$, it is proved by Goldman that the set of regular points is open and dense in $\mathsf{Hol}_{\partial\mathcal{S}}^{-1}(1_G)$.
Moreover, we also use the result due to Goldman~\cite[see paragraph 1.3 p.~204 and the first sentences of paragraph 1.4 p.~205]{Goldman1} that regular points of $\mathsf{Hol}_{\partial\mathcal{S}}^{-1}(1_G)$ automatically descend to regular points of the moduli space $\mathcal{M}_g\simeq\mathsf{Hol}_{\partial\mathcal{S}}^{-1}(1_G)/G$.
Later, in Lemma~\ref{NegSing}, we will recall a central result that shows that for $g\geq 2$, the set of singular points has zero symplectic volume. However, for $g=1$, we will see that all the points of $\mathsf{Hom}(\pi_1(\Sigma),G)$ is singular in the sense of definition \ref{SingPoint} requiring a separate treatment.

\subsubsection{The case of genus $\geq 2$}

 The following is a recollection of the main steps is the case of genus $\geq 2$.

\begin{enumerate}
    \item Near regular points of $\mathsf{Hol}^{-1}_{\partial\mathcal{S}}(1_G)$, we can pull back the heat kernel $p_\varepsilon(\cdot)$ by the holonomy map. This pullback is well-defined when we let $\varepsilon \rightarrow 0^+$ by the pullback theorem of H\"ormander. In fact, $\operatorname{Hol}_{\partial\mathcal{S}}$ is a submersion near regular points. (See Paragraph~\ref{PullBackHeat})
    \item The tangent space to any regular point can be identified with certain first cohomology group $H^1$ of a twisted de Rham complex. (See Paragraph~\ref{TangentCohomology}).
    \item On the tangent space to any regular point, we can define a volume element. Using the above cohomological interpretation of the tangent space, this volume element is identified both with the Reidemeister torsion and also with some symplectic volume on $H^1$. This is based on fundamental ideas of Witten~\cite{Witten1}.
    \item We deduce that the symplectic volume of the regular part of the moduli space $\mathcal{M}_g$ has full mass as shown by Forman based on~\cite{Witten1}.
\end{enumerate}

\myparagraph{Pull back of the heat kernel near regular points of $\mathsf{Hol}_{\partial\mathcal{S}}^{-1}(1_G)$. } \label{PullBackHeat}
The group $G$ with its left invariant metric can be viewed as a Riemannian manifold. This convey to $G^{2g}$ the Riemannian structure and induces a Riemannian volume $\dd V_{\mathsf{Hol}_{\partial\mathcal{S}}^{-1}(1_G)}$ on the regular part of  $\mathsf{Hol}_{\partial\mathcal{S}}^{-1}(1_G)\subset G^{2g}$.  
For any collection of closed curves $\gamma_{i}, i=1,\dots,N$, we denote
by $\rho_{\gamma_i}$ the corresponding element of the surface group representation 
and $\mathsf{Hol}_{\gamma_i}(A)$ the holonomy under the flat connection $A$. By results of  Robin Forman~\cite{Forman} and Kefeng Liu~\cite{LiuHeat1}  which are heavily inspired by the work of Witten~\cite{Witten1}, we have the fundamental identity 
\begin{align}\label{eq:FromheattoABG}
 &\lim_{\epsilon\to 0}\int_{G^{2g}} \prod_{i=1}^N f_i\big(\mathsf{Hol}_{\gamma_i}(A)\big) p_{\epsilon}\big( \mathsf{Hol}_{\partial\mathcal{S}}(a)  \big) \dd^{2g}a \\
\nonumber \\
&\nonumber \hspace{2cm}= \mathsf{Vol}(G) \int\limits_{\mathsf{Hol}_{\partial\mathcal{S};\mathsf{reg}}^{-1}(1_G) } \prod_{i=1}^N f_i\big(\mathsf{Hol}_{\gamma_i}(A)\big) {\det}^{-\frac{1}{2}}\Big( \big(\dd \mathsf{Hol}_{\partial\mathcal{S}}(a)|_{N_a}\big)^*  \dd \mathsf{Hol}_{\partial\mathcal{S}}(a)|_{N_a} \Big)\dd V_{F^{-1}(1_G)}(a)  \\
\nonumber &\hspace{4cm}=\frac{\mathsf{Vol}(G)^2}{\sharp Z(G)} \int\limits_{\mathcal{M}^{\mathsf{reg}}_{g}}\prod_{i=1}^N f_i\left(\rho_{\gamma_i}\right) \dd\nu,
\end{align} 
where in the second line, we only integrate on the regular part $\mathsf{Hol}_{\partial\mathcal{S};\mathsf{reg}}^{-1}(1_G)$ of  $\mathsf{Hol}_{\partial\mathcal{S}}^{-1}(1_G)$, in the third line we only integrate on the regular part $\mathcal{M}^{\mathsf{reg}}_{g}$ of $\mathcal{M}_g$. The number
$\sharp Z(G)$ is the cardinality of the center of $G$,
and on the r.h.s. the holonomies are defined combinatorially by the intersection of each $\gamma_i$ with the unstable curves. The measure $\dd\nu$ will be identified with the symplectic volume 
on $\mathcal{M}_g$. The key reason why the symplectic volume appears
comes from some identification of the factor ${\det}^{-\frac{1}{2}}\Big( \big(\dd \mathsf{Hol}_{\partial\mathcal{S}}(a)|_{N_a}\big)^*  \dd \mathsf{Hol}_{\partial\mathcal{S}}(a)|_{N_a} \Big)$ with Reidemeister torsion
for non acyclic complexes, as seen next.

\myparagraph{Identifying the volume element $\dd\nu$.}
We reformulate the computations on the volume element $\dd\nu$ appearing in \cite{LiuHeat1,Forman}. On the regular part of the space of surface group representation $\mathsf{Hol}_{\partial\mathcal{S}}^{-1}\left(1_G\right)$,
one can define the measure
\[
\frac{\bigwedge_{a\in \mathsf{Crit}_1(f)} \dd g_a }  { \mathsf{Hol}_{\partial\mathcal{S}}^\star \dd g }    \]
where the numerator is the wedge product of Haar measure on $G^{2g}$ (equivalently, it is the same measure as the wedge product of each volume form) and the denominator is the pull--back $\mathsf{Hol}_{\partial\mathcal{S}}^\star \dd g$  of the Haar volume form on $G$ by the map $\mathsf{Hol}_{\partial\mathcal{S}}:G^{2g} \rightarrow G $.  The division of differential forms is to be understood in the sense of Gelfand--Leray~\cite[Def Lemma 5.11 p.~123]{Zoladek2006}. 
This measure descends from the pre-image $\mathsf{Hol}_{\partial\mathcal{S}}^{-1}(1_G)$ to the moduli space $ \mathsf{Hol}_{\partial\mathcal{S}}^{-1}(1_G)/G $. In fact, there is a gauge action of
$G$ on $G^{2g}$, defined
by 
\[u\cdot( g_a )_{a\in \mathsf{Crit}_1(f)}=\big (u^{-1}g_au\big)_{a\in \mathsf{Crit}_1(f)}\in G^{2g}.\]
We denote by $\mathbf{Ad}$ this action.
We choose a polyvector field\footnote{Recall that a polyvector field on $G$ of degree $k$ is a smooth section $C^\infty(\Lambda^kTG)$ of the bundle $\Lambda^kTG$ of exterior powers of degree $k$ of $TG$. } $V$ which is the canonical volume element on $G$ that is normalized to have volume $1$ on each fiber. This is possible since the Riemannian metric on $G$ is fixed. This choice 
induces
a polyvector 
\[\mathbf{Ad}_*(V )\in C^\infty\big(\Lambda^{\dim(\mathfrak{g})} TG^{2g}\big).\]  
Consider the contraction
operator 
$\iota_{\mathbf{Ad}_*(V )}$ defined by the above polyvector field.  The fundamental formula defining the symplectic volume on the moduli--space $\mathcal{M}_g=\mathsf{Hol}_{\partial\mathcal{S}}^{-1}(1_G)/G$ writes
\[\dd\nu \coloneq  C\iota_{\mathbf{Ad}_*(V )}    \frac{\bigwedge_{a\in \mathsf{Crit}_1(f)} \dd g_a }  {\mathsf{Hol}_{\partial\mathcal{S}}^\star \dd g},    \]
where the constant $C$ is irrelevant since we are only interested 
in the symplectic volume normalized to be of total volume $1$.

The striking fact is that this volume element $\dd\nu$ calculated at $\rho$ is exactly the Reidemeister torsion viewed as a volume element on $H^1\big(C^\bullet_\rho\big)$ following Forman~\cite[p.~50--51]{Forman} and Liu~\cite[Lemma 4 p.~753--754]{LiuHeat1}. This happens to have a symplectic interpretation as well. Let us review the necessary facts needed to state this identification. We start with identifying the tangent space $T_\rho \mathcal{M}_g$ to $\mathcal{M}_g$ at a regular point $\rho\in \mathcal{M}_g$ with the first cohomology group $H^1$ of a certain cochain complex.

\myparagraph{The tangent space at regular points of the character variety and cohomology.} \label{TangentCohomology}

In the sequel, we will freely use the identification: character variety $\simeq$ moduli space of flat connections.
It follows from the work of Goldman~\cite{Goldman1,Goldman2} that the character variety is an orbifold whose regular points correspond exactly to \emph{irreducible} flat connections which means that  
\[H^0\big(\Omega^\bullet(\Sigma,\mathfrak{g}),\mathbf{ad}_\nabla\big)=H^2\big(\Omega^\bullet(\Sigma,\mathfrak{g}),\mathbf{ad}_\nabla\big)=0.\]
We refer to ~\cite[p. 206-207]{Goldman1} where this is discussed in the language of group cohomology but
we mention that in \cite[ paragraph 1.8 p. 208]{Goldman1}, it is described how to go from the language of group
cohomology to de Rham theory. By the de Rham Theorem, one could deduce it from the argument
of \cite[bottom p. 50 and top p. 51]{Forman}. Another way to equivalently define irreducibility is that there are no $\nabla$-invariant sub-bundles of $E$. In this subsection, there are three pictures:
\begin{enumerate}
\item The de Rham picture which identifies tangent spaces of the character varieties to the first cohomology group $H^1\big(\Sigma,\operatorname{End}(E)\big)$ of the following cochain complex
\[
0\longrightarrow\Omega^0\big(\Sigma,\operatorname{End}(E)\big) \longrightarrow \Omega^1\big(\Sigma,\operatorname{End}(E)\big)  \longrightarrow \Omega^2\big(\Sigma,\operatorname{End}(E)\big)  \longrightarrow 0 
\]
equipped with the differential $\mathbf{ad}_\nabla $. This is described in a different language in \cite[1.3 p. 204]{Goldman1}
The flatness of $\nabla$ ensures that the differential squares to $0$, indeed $\mathbf{ad}_\nabla\circ \mathbf{ad}_\nabla=\mathbf{ad}_{[\nabla,\nabla]}=0$. In this case, one can identify the tangent space $T_m\mathcal{M}_{g,0}$ with the first cohomology group 
$H^1\big( \Omega^\bullet(\Sigma,\mathfrak{g}),\mathbf{ad}_\nabla \big)$ whose dimension is $(2g-2)\dim(G)$ by a simple application of the Atiyah--Bott--Lefschetz fixed point formula. We refer to Appendix~\ref{appendix:Atiyah-Bott-Lefschetz}.

\item The Betti picture which identifies it to infinitesimal deformations of representation of surface groups. 
This point of view is related to the deformation complex of the flat $G$--bundle in~\cite[p.~753]{LiuHeat1}.
Recall we defined the map: 
\[\mathsf{Hol}_{\partial\mathcal{S}} : (g_a)_{a\in \mathsf{Crit}_1(f)}\in G^{2g}\longmapsto  \mathsf{Hol}_{\partial\mathcal{S}}\big( (g_a)_{a\in \mathsf{Crit}_1(f)}\big)\in  G.\]
The pre-image $\mathsf{Hol}_{\partial\mathcal{S}}^{-1}\left( 1_G \right)$  is $\mathbf{Ad}_G$ invariant. The moduli space $\mathcal{M}_g \coloneq F^{-1}\left( 1_G \right)/G$ is the quotient by the $G$ action. Assume we are given a regular element $\rho\in \mathcal{M}_g $. As already seen above, this coincides with a representation of the surface group whose corresponding flat connection $\nabla_\rho$ is such that \[H^0\big(\Omega^\bullet\left(\Sigma,\mathfrak{g}\right),\mathbf{ad}_{\nabla_\rho}\big)=H^2\big(\Omega^\bullet\left(\Sigma,\mathfrak{g}\right),\mathbf{ad}_{\nabla_\rho}\big)=0.\] 
Then the tangent space $T_\rho\mathcal{M}_g$ at the point $\rho$ can be identified with 
the first homology of the short exact sequence~:
\[
0\longrightarrow {} T_{1_G}G \xrightarrow[]{\dd\mathrm{Ad}} T_\rho G^{2g}  \xrightarrow[]{\dd F} T_{1_G}G  \longrightarrow 0.    
\]
In the notations, we identify the element $\rho$ in the moduli space with its representative in $G^{2g}$. We now restrict the above complex to the fiber above $1_G\in G$, $\rho\in G^{2g}$ and $1_G\in G$ respectively.

\item The CW cohomology picture that identifies these infinitesimal deformations to
twisted cocycles on some cochain complex twisted by the representation of the surface group. The last picture is related to the first one via the de Rham isomorphism Theorem in the twisted setting and it is related to the Betti picture if one has a CW decomposition of the surface with only one cell with $4g$ edges.
\end{enumerate}

\myparagraph{Twisted cohomology and torsion.}
The torsion for non acyclic complexes was defined in the seminal work of Bismut--Zhang~\cite{BZ} to be a norm on the determinant line $\det\big( H^\bullet(C^\bullet) \big)$ of the homology of a chain complex which is induced by the choice of volume elements, for every $(C^i)_i$, of the cochain complex. This relies on the canonical isomorphism between determinant lines of Grothendieck and Knudsen--Mumford~\cite{KM}
\begin{equation}
\det\big(H(C^\bullet)\big)\simeq \det(C^\bullet).   
\end{equation}
We refer the reader to \cite[eq (1.5) p.~54]{BGS1} for the proof of the canonical isomorphism. In practice, both the chain and cochain complexes associated to the Morse--Witten complex of a manifold have canonical volume elements. This is one of the key ingredient in the work of Bismut--Zhang~\cite[(1.5) p.~22]{BZ}.

For each degree $i$, we choose $\left(e^i_1,\dots, e^i_{r_i}\right)\in C^i$
such that $\left(\dd e^i_1,\dots, \dd e^i_{r_i}\right) $ forms a basis of $\mathrm{Im}(\dd ^i)\subset C^{i+1} $, as
in~\cite[p.~49-50]{Forman}.
For any cohomology classes $f^i_1,\dots,f^i_s\in H^i(C)$, 
choose 
 $\tilde{f}^i_1,\dots,\tilde{f}^i_s$ 
 some representative in $C^i$ of the chosen cohomology classes.
The torsion is defined by the fundamental formula

\begin{align*}
    &\tau_C\left( \bigotimes_{i=0}^n (f_1\wedge\dots \wedge f_{s_i} )^{(-1)^{i+1}}  \right)  \\
    &\hspace{4cm}\coloneq  \frac{\prod\limits_{i \text{ odd}} \dd(e_1^{i-1}) \wedge\cdots \wedge \dd(e_{r_{i-1}}) \otimes \tilde{f_1}\wedge \cdots\wedge \tilde{f_{s_i}}\otimes e^i_1\wedge \cdots \wedge e^i_{r_i}  }{ \prod\limits_{i \text{ even}} \dd(e_1^{i-1}) \wedge\cdots \wedge \dd(e_{r_{i-1}}) \otimes \tilde{f_1}\wedge \cdots\wedge \tilde{f_{s_i}}\otimes e^i_1\wedge \cdots \wedge e^i_{r_i}}.  
\end{align*}

\myparagraph{ Identifying $d\nu$ with torsion and symplectic volumes.}
For completeness, we recall here the proof of Witten following the arguments in ~\cite{LiuHeat1,Forman}. We consider the cell decomposition of our surface $\Sigma$ given
by the $4g$-sided polygons which is obtained by cutting the surface along the unstable curves $W^u(a)$, $a\in \mathsf{Crit}_1(f)$. Each edge is identified with an unstable curves, which gives the well-known description of the
fundamental group of $\Sigma$ as explained in Paragraph~\ref{sss:morsemoduli}. It has one $0$-cell, one $2$-cell and $2g$ $1$-cells. 
The complex $C_\rho^\bullet$ is precisely the Morse--Witten complex twisted by the representation $\rho$. The dual cell decomposition obtained by cutting our surface along stable curves $W^s(a)$, $a\in \mathsf{Crit}_1(f)$ gives
us a dual complex $C_\rho^{\bullet,\vee}$ to $C_\rho^\bullet$. Poincaré duality induces a natural skew-
symmetric pairing on the complex $\mathcal{D}_\rho \coloneq C_\rho^\bullet\oplus C_\rho^{\bullet,\vee} $ which is compatible
with both the differentials and the natural measures on each term in $\mathcal{D}_\rho$.
Therefore, the Atiyah--Bott symplectic form induces a 
symplectic pairing on $H^1\big(C^\bullet_\rho\big)\simeq T_\rho \mathcal{M}_g$. It induces a symplectic pairing $T_\rho \mathcal{M}_g\times T_\rho \mathcal{M}_g\rightarrow\mathbb{R} $, which is the natural symplectic form $4\pi^2\omega$ on $\mathcal{M}_g$~\cite[eq (4.16) to (4.28)]{Witten1}. 
By multiplicativity of the torsion,  we obtain \[\tau\big(\mathcal{D}_\rho\big)=\tau\big(C_\rho\oplus C^\vee_\rho\big)=\tau\big(C_\rho\big)\tau\big(C^\vee_\rho\big)=\tau\big(C_\rho\big)^2.\]
By \cite[eq (4.28)]{Witten1} we have $\sqrt{\tau(\mathcal{D}_\rho)}=(2\pi)^N \frac{\omega^N}{N!} $ for $N=(g-1)\dim(G)$--- half the dimension of the moduli space $\mathcal{M}_g$.  This implies the desired equality due to Witten~\cite[eq (4.19) to (4.28)]{Witten1}, \cite[Lemma 4 p.~753]{LiuHeat1}:
\[\label{eq:WittenequalsTorsion}
\dd\nu(\rho)=\frac{(2\pi\omega)^N}{N!}=\tau(C_\rho^\bullet),
\]
where $\omega$ is the natural symplectic form on $T_\rho\mathcal{M}_g$.

\begin{remark}
Another way to get this identification we could use the Cheeger--M\"uller Theorem in the version proved by Bismut--Zhang \cite{BZ}. We know the Milnor metric on $\det\big(H^\bullet(C_\rho^\bullet)\big)$ coincides with the  Ray--Singer or $L^2$ metric on   $\det\big(H^\bullet(\Omega^\bullet,\nabla_\rho)\big)$. However, in our case, the $L^2$ metric on the determinant line coincides with the symplectic volume.   
\end{remark}

We are able now to recall a central result of Forman~\cite{Forman} and Sengupta~\cite{SenguptaSemiclassic} that allows us to freely disregard the singular locus $\mathcal{M}_g^{\mathsf{sing}}$ when discussing symplectic volumes.

\begin{lemma}[The symplectic volume of singular connections vanishes]\label{NegSing}
The symplectic volume of $\mathcal{M}_g^{\mathsf{sing}}$ is zero.    
\end{lemma}

\begin{proof}    
This is proved in~\cite[Lemma 4 p.~407]{Forman}, where Forman gives an indirect proof using Witten's result. From the previous discussion, we have

\[
\frac{\mathsf{Vol}(\mathcal{M}^{\mathsf{reg}}_g)}{\sharp Z(G)}=\lim_{t \to 0^+} \frac{1}{\mathsf{Vol}(G)^2} \int_{G^{2g}} p_t\!\left( \mathsf{Hol}_{\partial\mathcal{S}}\!\left[ \sum_{a \in \mathsf{Crit}_1(f)} \log(g_a) \, U_a \right] \right) \prod_{a \in \mathsf{Crit}_1(f)} d g_a = \frac{\mathsf{Vol}(\mathcal{M}_g)}{\sharp Z(G)},
\]
where the second equality follows from Witten's computations, see \cite[Lemma 3 p.~752]{LiuHeat1}.
The careful reader will recognize that Forman does not actually analyze the volume of tubular neighborhoods of the singular locus, but rather relies on gluing formulas for the Reidemeister torsion, which allowed Witten to arrive at the above identity by combining~\cite[equation (4.72) p.~197 and equation (2.67) p.~71]{Witten1}.  
\end{proof}
\begin{remark}
    Another proof of the above result was given in~\cite[Prop. 8 p.~408]{SenguptaSemiclassic} in a more self-contained way, independent of Witten's calculation of symplectic volumes.
\end{remark}
\subsubsection{The case of genus $0,1$}

The purpose of this section is twofold. First, it can be seen as an illustrative and concrete example of the notions used previously. Second, it studies the case of genus $1$ which requires a special treatment. For simplicity, we will consider $G = \mathsf{SU}(2)$; the general case can be treated similarly.

We only discuss the case where $\Sigma$ is a torus, the case of the genus $0$ being trivial since the moduli space of flat connections on a sphere $\mathbb{S}^2$ is reduced to a point.

First let us discuss why in this case all the points of 
$\operatorname{Hom}(\pi_1(\Sigma), G)$ are singular in the sense of~\cite[Proposition on top of p.~204]{Goldman1}.

\myparagraph{Singularity of the case $g=1$.}
An element of $\operatorname{Hom}(\pi_1(\Sigma), G)$ can be identified with pairs of commuting matrices $\{AB = BA\} \subset G^2$. More precisely, $\pi_1(\Sigma) \simeq \mathbb{Z}^2$ and its image under the representation $\rho$ is the subgroup $\{A^n B^m : n, m \in \mathbb{Z}^2\}$ generated by the two commuting matrices $A, B$.

However, given any $\rho \in \operatorname{Hom}(\pi_1(\Sigma), G)$, we see that the centralizer $Z_\rho$ of the subgroup in $G$ generated by two commuting matrices is always of dimension at least $1$. This is straightforward since it is either a maximal torus of dimension $1$ if $(A, B) \neq (\pm 1_G, \pm 1_G)$, or the whole group $Z_\rho = G$ if $(A, B) = (\pm 1_G, \pm 1_G)$. The center $Z(G)$ of $G$ has dimension $0$, which implies that all points of $\operatorname{Hom}(\pi_1(\Sigma), G)$ are singular in the sense of \cite[Proposition on top of p.~204]{Goldman1} as discussed in paragraph~\ref{sss:recollectionsingular}.
Another more geometric way to view this is to identify $\operatorname{Hom}(\pi_1(\Sigma), G)$ with the preimage $\mathsf{Hol}_{\partial\mathcal{S}}^{-1}(1_G)$ of $1_G$ under the map
\[
\mathsf{Hol}_{\partial\mathcal{S}}: (A, B) \in G^2 \longmapsto ABA^{-1}B^{-1}.
\]
If we calculate the differential $\dd\mathsf{Hol}_{\partial\mathcal{S}}$ and restrict it to $T\mathsf{Hol}_{\partial\mathcal{S}}^{-1}(1_G)$, we see immediately that its image inside $T_{1_G}G \simeq \mathfrak{g}$ is not onto.
From the point of view of analysis, it also means that for any pair $(A_0,B_0)$  of commuting matrices, for any neighborhood $U_0\subset G^2$ of $(A_0,B_0)$, the pull--back $p_\varepsilon(ABA^{-1}B^{-1})$ does not converge in $\mathcal{D}^\prime(U_0)$ when $\varepsilon\rightarrow 0^+$.  This is due to
$\mathsf{Hol}_{\partial\mathcal{S}}$ failing to be a submersion
at $(A_0,B_0)$.

\myparagraph{Proof of Theorem~\ref{YMABG} for $g=1$.}
\par Up to simultaneous conjugation, any two commuting matrices $A, B \in \mathsf{SU}(2)$ can be chosen to lie in the maximal torus of $\mathsf{SU}(2)$, which is $\mathbb{S}^1$. Then, again up to conjugation, an element of the character variety can be identified with a pair of angles $(\theta_1, \theta_2) \in \mathbb{S}^1 \times \mathbb{S}^1$ modulo the symmetry $(\theta_1, \theta_2) \mapsto (-\theta_1, -\theta_2)$. 
\par Consequently, $\mathcal{M}_1$ can be identified with the quotient $\mathbb{T}^2 / \mathbb{Z}_2$, which is an orbifold sphere $\mathbb{S}^2$ with four singular points, sometimes called the \emph{pillowcase} (the torus $\mathbb{T}^2$ is a ramified covering of the pillowcase $\mathbb{S}^2$). A straightforward calculation also shows that the symplectic area of $\mathcal{M}_1$ is a multiple of $d\theta_1 \wedge d\theta_2$.
\par Now, let $\epsilon>0$, and define the measure $\nu_\epsilon$ as a measure on $G\times G$ by 
\[\nu_\epsilon(\dd g,\dd h)=\frac{p_{\epsilon}([g,h])}{\int_{G^2}p_{\epsilon}([x,y])\dd x \dd y}\,\dd g \dd h.\]
The first step is to show that $(\nu_\epsilon)_{\epsilon>0}$ converges. The second step is to relate it with the ABG symplectic form $\omega$ of Proposition~\ref{ABGSymp}.
\begin{proposition}[Convergence of $(\nu_\epsilon)_{\epsilon>0}$ in genus 1]\label{prop:genus1iden}
    There exists a measure $\nu$ on $G\times G$ invariant by the adjoint action of $G$ on $G\times G$ such that $\nu_\epsilon\xrightarrow[\epsilon\to 0]{}\nu$. Moreover, the support of $\nu$ is contained in $\mathcal{C}:= \{(g,h)\in G^2 : gh=hg\}$.
    \par Denote by $\pi:\mathcal{C}\mapsto \mathcal{M}_1\simeq \mathcal{C}/G$ where the quotient is by the adjoint action. The limiting measure $\nu$ is given by the formula
  \begin{align*}
  \int_{G\times G} F \nu=\int_{\mathcal{C}} F \pi^*\omega
  \end{align*}
  on all $G$--invariant functions $F$ and  where $\omega$ the symplectic probability measure on $\mathcal{M}_1$. 
\end{proposition}
\begin{proof}
This is the main result from the article of Sengupta~\cite{SenguptaTorus}.  
\end{proof}
This completes the proof of Theorem.

\begin{remark}
    Modulo some rather technical combinatorial complications, the above proof generalized to the case of $\mathsf{SU}(N)$ for all $N$. We chose to do the case $\mathsf{SU}(2)$ for simplicity.
\end{remark}

\section{Some Banach spaces of distributional connections} \label{SpacesBYM}
In this section, we introduce the spaces $\mathcal{B}^{\alpha,p,s,\ell}_\mathsf{YM}$. This space appears as a canonical space that contains the support of the Morse gauge fixed Yang--Mills measure.  We will prove a fundamental decomposition theorem for distributions in this space, and finally study compact injections with respect to the parameters $\alpha,p,s,\ell$.

\subsection{The spaces $\mathcal{B}^{\alpha,p,s,\ell}_\mathsf{YM}$}
Let $\Sigma$ be a closed compact surface, $G$ a compact connected Lie group (we only need $G=\mathsf{SU}(N)$ for the semiclassical limit), and $\g$ its Lie algebra endowed with a bi-invariant inner product. Let $f$ be a Morse function on $\Sigma$ verifying the Smale conditions. 
We start by recalling the definition of the space $\mathcal{W}^{\alpha,p,s,\ell}$ of distributional $1$-forms introduced in \cite[Definition 5.2]{DN}.
\par We would like to define anisotropic spaces of distributional connections that contain the support of the Yang--Mills measure. In \cite{BCDRT,DN}, the Yang--Mills measure is expressed as the sum of a noise term and a singular term. Therefore, our spaces  $\mathcal{B}^{\alpha,p,s,\ell}_\mathsf{YM}$ will be defined as direct sums of two types of spaces. We will start with the noise part.

\subsubsection{The weighted  noise space}
In \cite{BCDRT,DN}, the noise term is roughly the integral of a white noise in the direction of the flow of $f$. We therefore expect the local regularity of the Yang--Mills random connection to be $\frac{1}{2}^-$ in the direction of the flow of $f$ and $-\frac{1}{2}^-$ in the transverse direction. However, since Morse theory imposes analytical singularities on the area function near critical points, and since the Yang--Mills measure exhibits area law, we expect the behaviour of the Yang--Mills random field to be more singular as we approach critical points of $f$.
\par Therefore, the idea of our spaces is as follows. Define local anisotropic norms using Gagliardo-type norms. Then, define weighted scaled quantities of these norms as we approach the critical points of $f$. We start by covering the complements of the critical locus by flow boxes to define the local anisotropic norms.

\myparagraph{Covering complements of the critical locus by flow boxes.}
Near each critical point \( a \in \mathsf{Crit}(f) \), consider a small disc \( D_a \) containing \( a \). Remove these discs and consider the punctured surface
\[
\Sigma \setminus \cup_{a \in \mathsf{Crit}(f)} D_a.
\]
This set can be covered by finitely many flow boxes \( (\square_i)_{i \in I} \) associated with the gradient vector field \( V = \nabla f \). That is,
\[
\Sigma \setminus \cup_{a \in \mathsf{Crit}(f)} D_a = \bigcup_{i \in I} \square_i.
\]
On each flowbox $\square_i$, we will now define local anisotropic norms.
\myparagraph{The anisotropic space $\mathcal{W}^{\alpha,\alpha-1,p}(\square)$.}
In this paragraph, we introduce the anisotropic space $\mathcal{W}^{\alpha,\alpha-1,p}(\square)$ associated to the flow box $\square$. Each flow box is identified with $[0,1]_x\times [0,1]_y$ with local coordinates $(x,y)$ where $V=\partial_x$ in the local coordinates. This means that the direction of the flow of $\nabla f$ is parallel to $x$ and the transverse direction is parallel to $y$.
\begin{definition}[Flowbox anisotropic norms] \label{FlowboxSpace}
 On $\square \equiv[0,1]_x\times [0,1]_y$, we define local semi-norms on smooth connections as follows. For all $ \tilde{A}\in\Omega^1(\square,\g)$,
 \begin{align*}\label{eq:weightedSobo}
    &\Vert \tilde{A}\Vert_{\mathcal{W}^{\alpha,\alpha-1;p}_{x,y}(\square) } \coloneq    \Vert W\Vert_{\mathcal{W}^{\alpha,\alpha;p}_{x,y}([0,1]\times [0,1]) }\\
     &\hspace{2cm}= 
\left\{
\int_{ [0,1]\times [0,1]} \frac{\vert W(x_1,y_2)-W(x_2,y_1)-W(x_1,y_2)+W(x_2,y_2) \vert^p_{\g}}{\vert x_1-x_2\vert^{1+\alpha p}\vert y_1-y_2\vert^{1+\alpha p}} \,\dd x_1\dd x_2 \dd y_1\dd y_2\right\}^{\frac{1}{p}}
\nonumber\\
&\nonumber\hspace{4cm}+ \Vert W \Vert_{L^p([0,1]\times [0,1])},
 \end{align*}
where $W$ is such that
\[\tilde{A}(x,y)= \left(\partial_y W \right) (x,y)\dd y  \text{ and } W(.,0)=0~\footnote{This acts like some half--Dirichlet condition on $W$.}.\]
With the usual changes for $p=\infty$.
\end{definition}

On $\Sigma\setminus D_a$, we consider the norm 
\[\sum_{i\in I}  \Vert A \Vert_{\mathcal{W}^{\alpha,\alpha-1,p}(\square_i) } \]
where the sum runs over all flow boxes of the covering.
The next step is to define norms near the critical locus of $f$. We will do this by scaling flow box norms closer and closer to each critical locus.

\myparagraph{Scaling near critical points.}
The next step involves scaling near the saddle points and near the maximum $\max f$ \footnote{And, of course, $\min f$. However, the dynamics of the Morse function is similar near these two points, which makes these cases exactly similar. We treat therefore only one of them.}.
Scaling near $\max(f)$ is the same as scaling near the boundary $\partial\mathcal{S}$, since in polar coordinates $(r,\theta)$,  $\{r=\max f\}$ is the equation for the boundary $\partial\mathcal{S}$ which happens to be the blown--up maximum. 

\par Fix now a critical point $a\in \mathsf{Crit}(f) $, that can be either a saddle point, or an extremal point. We will study the scaling near $a$.
Consider a corona $C_a$ obtained by setting $C_a \coloneq D_a\setminus \tilde{D}_a$ where $\tilde{D}_a\subsetneq D_a$ is another disc around $a$ strictly smaller than $D_a$. We will probe the regularity of $A$ near each critical point by scaling. A weight $s\in \mathbb{R}$ is used to track the growth when we scale near saddle points, and a weight $\ell$ is used for the extremal points. Refer to Figure~\ref{fig:scalednorms} for a schematic representation of the scaling process.
\par For $\lambda\in (0,1]$, we denote by $\mathcal{S}^\lambda_a:D_a\rightarrow D_a$ the scaling operator around $a$. In Morse chart \footnote{ Recall by the Morse Lemma, near every critical point $a$ of $f$, there exists a coordinate system $(x,y)$ such that $a$ reads $(0,0)$ in this coordinate system and $f(x,y)=f(0,0)+ \varepsilon_1x^2+\varepsilon_2y^2 $ where $\varepsilon_1\in \pm 1$, $\varepsilon_2\in \pm 1$. The metric $h$ is chosen to be flat in the Morse chart near every critical points.}  $(x,y)$ near any critical point $a=(0,0)$, it is simply defined by
\begin{equation}
\mathcal{S}^\lambda_a(x,y)=(\lambda x,\lambda y).    
\end{equation}
This defines naturally a scaling $\mathcal{S}^{\lambda,\star}_a$ on forms $\Omega^k(D_a,\g), 0\leq k\leq 2$, by pull--back
\[  \left( \mathcal{S}_a^{\lambda,\star}\omega\right)(x,y;\dd x,\dd y)= \omega(\lambda x,\lambda y; \lambda \dd x, \lambda \dd y)   \]
where the notation $\omega(x,y;\dd x,\dd y)$ suggests that a differential form $\omega\in \Omega^k(D_a,\g)$ is viewed as a smooth function of $(x,y)$ and also as a polynomial in the Grassmann variables $(\dd x,\dd y)$.

\begin{figure}[t]
    \centering
    \includegraphics[width=\linewidth]{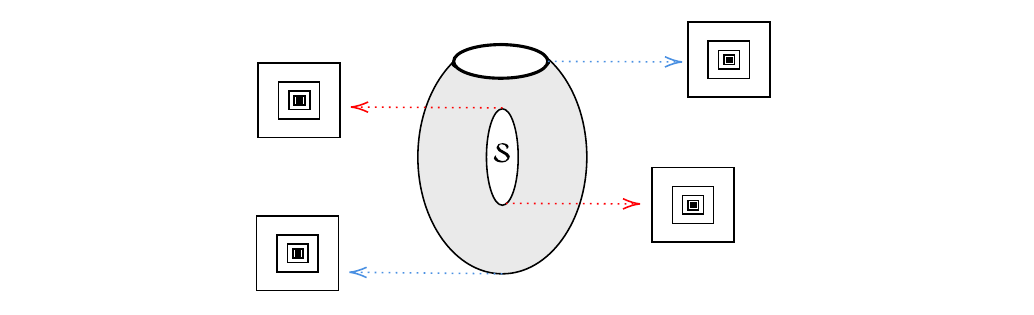}
    \caption{The corona $C_a$ and its scaling near critical points $a\in\mathsf{Crit}_1(f)$ in red, and near $a\in\{\min f,\max f \}$ in blue.}
    \label{fig:scalednorms}
\end{figure}

We introduce the weighted norms in the following definition.
\begin{definition}\label{WeightedSemiNorms}
Near $a\in  \mathsf{Crit}(f)$, define the local semi-norm 
\begin{eqnarray*}
 \Vert A\Vert_{\mathcal{W}^{p;\alpha,\alpha-1,s} (D_a) }^p \coloneq \sum_{\substack{i\in I \\\square_i\cap C_a \neq \emptyset}}\ \sum_{n=0}^\infty  \left(2^{n(s-\frac{2}{p})}\Vert \mathcal{S}^{2^{-n} ,\star}_{a} A\Vert_{\mathcal{W}^{\alpha,\alpha-1;p} (\square_i) }\right)^p   
\end{eqnarray*}
Define the global norm $\mathcal{W}^{p;\alpha,\alpha-1,s}\big(\Sigma\setminus \max(f)\big)$ by
\begin{align*}
    &\Vert A\Vert_{\mathcal{W}^{\alpha,p,s,\ell}(\Sigma\setminus \max(f)) }^p\\
    &\hspace{2cm}=\sum_{a\in \mathsf{Crit}_1(f)}\Vert A\Vert_{\mathcal{W}^{p;\alpha,\alpha-1,s} (D_a) }^p
+ \sum_{a\in \{\min(f),\max(f)\} } \Vert A\Vert_{\mathcal{W}^{p;\alpha,\alpha-1,\ell} (D_{a}) }^p
+\sum_{i\in I}  \Vert A \Vert_{\mathcal{W}^{\alpha,\alpha-1,p}(\square_i) }^p.
\end{align*}
We can now define the space $\mathcal{W}^{\alpha,p,s,\ell}$.
\end{definition}
\begin{definition}[The noise space]\label{NoiseSpace}
    The space $\mathcal{W}^{\alpha,p,s,\ell}$ is the completion of the space of smooth $1$-forms $\Omega^1\big(\Sigma\setminus \max(f)\big)$ with respect to the norm $\mathcal{W}^{\alpha,p,s,\ell}(\Sigma\setminus \max(f))$.
\end{definition}

We should allow different weights for $\min(f)$, $\max(f)$ and saddle points. In fact, the dynamics are around these ponts are different in nature, and we expect the behavior of the distribution to be different at these points. Let us fix our convention for the remainder of the paper:
the weight order $s$ is chosen to measure the blow--up at the saddle points and the weight order $\ell$ is chosen to measure the blow--up at the $\min$, $\max$ of $f$. In \cite{DN}, the support of the Yang--Mills measure was shown to be with weights $s<-1$, and we will verify $\ell<0$. The above norms also take into account the growth at the boundary $\partial\mathcal{S}$ which is obtained after blowing up the maximum $\max(f)$.

\subsubsection{ The noise-flat direct sum}
To define the space that supports the Yang--Mills measure, we need to add a space that supports singular connections of the form $x_aU_a$ where $x_a\in\g$ and $U_a$ is as before, the unstable integration current associated with the unstable manifold $W^u(a)$. We would like to define the space $\mathcal{B}_{\mathsf{YM}}^{\alpha,p,s,\ell}$ as
    \[\mathcal{B}_{\mathsf{YM}}^{\alpha,p,s,\ell}=\mathcal{W}^{\alpha,p,s,\ell}\oplus\Big(
    \g\otimes \mathsf{vect}\big\{U_a: a\in \mathsf{Crit}_1(f)\big\}\Big).\]
The direct sum is needed to have a well defined decomposition of elements in $\mathcal{B}_{\mathsf{YM}}^{\alpha,p,s,\ell}$ into a noise term and a flat, singular term. Therefore, we need the following lemma whose proof follows from Proposition~\ref{prop:notchargeflow}.
\begin{lemma}
  The following sum of vector spaces is direct.
  \[ \mathcal{W}^{\alpha,p,s,\ell}\oplus \Big(\g\otimes\mathsf{vect}\big\{U_a: a\in \mathsf{Crit}_1(f)\big\}\Big),\]
  when seen as subspaces of $\mathcal{D}'^1(S,\g)$.
\end{lemma}
Now we can give our definition of the space $\mathcal{B}_{\mathsf{YM}}^{\alpha,p,s,\ell}$.
\begin{definition}\label{YmSpace}
    The space $\mathcal{B}_{\mathsf{YM}}^{\alpha,p,s,\ell}$ is defined as 
    \[\mathcal{B}_{\mathsf{YM}}^{\alpha,p,s,\ell}=\mathcal{W}^{\alpha,p,s,\ell}\oplus\Big(
    \g\otimes \mathsf{vect}\{U_a: a\in \mathsf{Crit}_1(f)\}\Big).\]
    For $A\in\mathcal{B}_{\mathsf{YM}}^{\alpha,p,s,\ell}$ that can be written uniquely as 
\[A_\mathcal{N}+\sum_{a\in \mathsf{Crit}_1(f)}x_aU_a,\]
we define 
\[\|A\|_{\mathcal{B}_{\mathsf{YM}}^{\alpha,p,s,\ell}}=\|A_{\mathcal{N}}\|_{\mathcal{W}^{\alpha,p,s,\ell}}+\max_{a\in\mathsf{Crit}_1(f)}\|x_a\|_\g.\]
\end{definition}

\subsection{The fundamental decomposition theorem}
In this section we prove the following theorem. This is a key result because it tells us that distributions in $\mathcal{B}_{\mathsf{YM}}^{\alpha,p,s,\ell}$ for adequate parameters $\alpha,p,s,\ell$ admit a decomposition as a noise term and a singular term.
\begin{proposition}\label{prop:keyprop}
    Every $A\in \mathcal{B}^{\alpha,p,s,\ell}_{\mathsf{YM}}$ admits a unique decomposition of the form 
    \[A_\mathsf{\mathcal{N}}+\sum_{a\in \mathsf{Crit}_1(f)}x_aU_a,\]
    where $x_a\in \g$, $A_{\mathcal{N}}\in \mathcal{W}^{\alpha,p,s,\ell}$. 
    Moreover, the mappings 
    \[\pi_{\mathsf{noise}}: A\mapsto A_{\mathcal{N}}\ \text{ , } \ \mathsf{Ext}_a: A \mapsto x_a \ \text{ and } \ \pi_{\mathsf{flat}}: A\mapsto \sum_{a\in \mathsf{Crit}_1(f)}x_aU_a,\]
    are all well defined, and continuous.  The mapping $\mathsf{Ext}_a$ is local, meaning that for any $\psi\in C^\infty(\Sigma,\R)$ such that $\psi=1$ on a non-empty open set intersecting the unstable manifold $W^u(a)$, we have 
    \[\forall A\in\mathcal{B}_\mathsf{YM}^{\alpha,p,s,\ell}, \ \ \ \mathsf{Ext}_a\big(A\big)=\mathsf{Ext}_a\big(A\psi\big).\]
\end{proposition}

The name $\mathsf{Ext}$ is motivated by the fact that this map \emph{extracts} the singular part of $A$ which is supported on unstable curves.
To prove the Proposition, we will state and prove general results on currents and the set that they charge. 

\myparagraph{Currents that do not charge submanifolds.}
We start with the following notion, inspired by Meyer~\cite[Definition~2.5, p.~49]{Meyer98}.

\begin{definition}[Currents that do not charge a submanifold]\label{def:charge}
Let $M$ be a smooth manifold of dimension $d$, let $T\in\mathcal{D}^{\prime,0}(M)$ be a current of degree $0$, and let
$Y\subset M$ be a finite union of smooth embedded submanifolds of dimension $p<d$.
We say that $T$ \emph{does not charge} $Y$ if there exists a family of smooth cut-off functions
$(\chi_\varepsilon)_{\varepsilon>0}\subset C_c^\infty(M)$ such that, for every $\varepsilon>0$,
\[
0\leq \chi_\varepsilon \leq 1,\qquad
\chi_\varepsilon \equiv 1 \ \text{on a neighborhood of } Y,\qquad
\operatorname{supp}(\chi_\varepsilon)\subset U_{2\varepsilon}(Y),
\]
where $U_{2\varepsilon}(Y)$ denotes the $2\varepsilon$-tubular neighborhood of $Y$, and such that
\[
T(\chi_\varepsilon)\longrightarrow 0
\quad\text{as }\varepsilon\to 0^+.
\]
\end{definition}

We next define the notion of scaling near a submanifold $Y\subset M$, which generalizes the standard Euler vector field on $\mathbb{R}^n$.

\begin{definition}[Euler vector fields]\label{def:euler}
Let $M$ be a smooth manifold and let $Y\subset M$ be a smoothly embedded submanifold. Denote by
\[
\mathcal I_Y:=\{f\in C^\infty(M)\,:\, f|_Y=0\}
\]
the ideal of smooth functions vanishing on $Y$.
A vector field $\rho_Y$ on $M$ is called an \emph{Euler vector field of $Y$} if
\[
\rho_Y f - f \in \mathcal I_Y^2
\qquad\text{for every } f\in\mathcal I_Y,
\]
where $\mathcal{I}_Y^2$ denotes the ideal of smooth functions vanishing at order $2$ along $Y$.
Equivalently, $\rho_Y$ acts as the radial vector field in the normal directions to $Y$, modulo terms vanishing to second order along $Y$.
\end{definition}
Given some Euler vector field $\rho_Y$, the corresponding dynamics: $$e^{-t\rho_Y}:M\rightarrow M, t\geqslant 0$$ allows to scale near the submanifold $Y\subset M$ as made precise by the normal form Theorem proved in~\cite[Prop 2.5 p.~1252-1253]{DW}. We next give a sufficient condition, formulated in terms of scaling, ensuring that a given distribution does not charge a given sub-manifold $Y$.

\begin{lemma}\label{lem:notcharge1}
Under the assumptions and notation of Definition~\ref{def:charge}, let $\rho_Y$ be an Euler vector field of $Y$, let $T\in \mathcal{D}'(M)$, and assume that for some $s\in\mathbb{R}$ the family $
\left(e^{ts}\, e^{-t\rho_Y^*}T\right)_{t>0}
$
is bounded in $\mathcal{D}'(M)$. If $
s>\dim(Y)-\dim(M)$, $T$ does not charge $Y$.
\end{lemma}

The assumption that the family $
\left(e^{ts}\, e^{-t\rho_Y^*}T\right)_{t>0}
$
is bounded in $\mathcal{D}'(M)$ is a weak homogeneity assumption. It measures the growth of the distribution $T$ when we zoom in closer and closer to the submanifold $Y$. We use the scaling by $\rho_Y$ to approach $Y$.

\begin{proof}
By the results of \cite{DW}, there exists a neighborhood  $\mathcal U_Y$ of $Y$ which is stable under the forward flow $e^{-t\rho_Y}\mathcal U_Y\subset \mathcal U_Y$ for all $t\geqslant 0$.
The statement is local near $Y$, so we work in coordinates $(x,y)\in \mathbb{R}^{d-p}\times\mathbb{R}^p$ centered at a point of $Y$, with
\[
Y=\{x=0\},\qquad \rho_Y=\sum_{i=1}^{d-p} x_i\partial_{x_i}.
\]
The existence of such system of local coordinates is ensured by the normal form Theorem of~\cite[Prop 2.5 p.~1252-1253]{DW}. 
By the Banach--Steinhaus theorem, the boundedness assumption implies that there exists an integer $m\geq 0$ and a constant $C>0$ such that
\[
\sup_{t>0}\left|\left\langle e^{ts}e^{-t\rho_Y^*}T,\varphi\right\rangle\right|
\leq C\|\varphi\|_{C^m}
\]
for all test forms $\varphi\in \Omega_c^{\dim M}(\mathcal U_Y)$, where $\mathcal U_Y$ is the neighborhood of $Y$ which is stable under the forward flow.
Let $\chi\in C_c^\infty(\mathcal U_Y)$ satisfy $\chi\equiv 1$ near $Y$, and define
\[
\chi_\varepsilon:=e^{-(\log \varepsilon)\rho_Y^*}\chi,\qquad 0<\varepsilon<1.
\]
Then there exists $0<a<b$ such that 
$\chi_\varepsilon$ is a cutoff function supported in some $b\varepsilon$--tubular neighborhood of $Y$, and $\chi_\varepsilon=1$ on some $a\varepsilon$--neighborhood of $Y$ under the scaling induced by \(\rho_Y\). For any test form \(\varphi\in\Omega_c^{\dim M}(\mathcal U_Y)\), we want to prove that
$\langle T\chi_\varepsilon,\varphi\rangle$  vanishes when $\varepsilon\rightarrow 0^+$. The strategy is to use the dynamics to make a change of variables inside the duality pairing and conclude with the weak homogeneity assumption on $T$:
\[
\langle T\chi_\varepsilon,\varphi\rangle
=\langle e^{(\log\varepsilon)\rho_Y^*}\left( T\chi_\varepsilon \right),  e^{(\log\varepsilon)\rho_Y^*}\varphi\rangle
=
\left\langle e^{-(\log\varepsilon)s}e^{(\log\varepsilon)\rho_Y^*}T,\,
\chi e^{(\log\varepsilon)s}e^{(\log\varepsilon)\rho_Y^*}\varphi \right\rangle.
\]
By the Euler property of \(\rho_Y\), the rescaled test form is flattened by the flow hence it satisfies
\[
\big\| \varepsilon^{\dim(Y)-\dim(M)} \chi \, e^{(\log\varepsilon)\rho_Y^*}\varphi \big\|_{C^m}
\leq C \Vert \varphi \Vert_{C^m} ,
\]
uniformly for \(0<\varepsilon<1\). Hence
\[
|\langle T\chi_\varepsilon,\varphi\rangle|
\leq
C \Vert \varphi \Vert_{C^m}\,\varepsilon^{\,s-\dim(Y)+\dim(M)}.
\]
Since \(s-\dim(Y)+\dim(M)>0\), we conclude that $
\langle T\chi_\varepsilon,\varphi\rangle \to 0 \ \text{as } \varepsilon\to 0^+$.
Therefore \(T\) does not charge \(Y\).
\end{proof}

\myparagraph{Anisotropic spaces under scalings.}
We now explain how Lemma~\ref{lem:notcharge1} applies to our anisotropic spaces. The guiding principle is that the noise component has Hölder regularity $\alpha$ along the flow direction and regularity $\alpha-1$ transversally, while the weights $s$ and $\ell$ control the behavior near saddle points and extremal critical points, respectively.

Recall that we can choose $\alpha\in\left(\frac13,\frac12\right)$, so that the Yang--Mills random field is $\alpha$-regular along flow lines and $\alpha-1$-regular in the transverse direction. We also assume $s<-1$ and $\ell<0$. The singular behavior near saddle points is therefore stronger than near the extrema of the Morse function.

\begin{proposition}[Anisotropic currents do not charge flow lines]\label{prop:notchargeflow}
Let $T$ belong to an anisotropic space $\mathcal W^{\alpha,p,s,\ell}$ for $\alpha\in (\frac{1}{3},\frac{1}{2})$ and $(s,\ell)>-2$. Then $T$ does not charge any flow line of the gradient flow $(\varphi^t_f)_{t\in\mathbb R}$, in particular any stable or unstable curve.
\end{proposition}

\begin{proof}
We argue locally. First consider a point away from the critical set. In a flow box $[-1,1]_x\times[-1,1]_y$, we may assume that the vector field is $\partial_x$ and that the flow line under consideration is given by $\gamma=\{(t,0): t\in[-1,1]\}$. In these coordinates, the dyadic scaling in the transverse variable $y$ is
\[
(\mathcal S^{2^{-n}\star}\varphi)(x,y)=\varphi(x,2^{-n}y),
\]
so the corresponding Euler vector field is $\rho=y\partial_y$.

By definition of the anisotropic spaces, the distribution $T$ has transversal regularity $\alpha-1$ in the $y$ variable, uniformly in the $x$ variable (representing the flow direction). Therefore, the family
\[
\left(2^{n(\alpha-1)}\,\mathcal S^{2^{-n}\star}T\right)_{n\geq 0}
\]
is bounded in $\mathcal D'([-1,1]^2)$. Since $\alpha\in(\frac13,\frac12)$, we have $
\alpha-1> -\frac23>-1$,
and Lemma~\ref{lem:notcharge1} applies with the relevant scaling exponent. Hence $T$ does not charge $\gamma$.

It remains to treat the critical points. We start with the case where $a$ is either a minimum or maximum of $f$. Let $a\in\mathsf{Crit}(f)$ and let $\mathcal S_a^\lambda$ denote the scaling map centered at $a$. By definition of the weight $\ell$, the family
\[
\left(2^{n(\ell-\frac2p)}\,\mathcal S_a^{2^{-n}\star}T\right)_{n\geq 0}
\]
is bounded in $\mathcal D'(U_a)$. Since $\ell>-2$, and after choosing $p$ sufficiently large, the exponent satisfies the threshold required by Lemma~\ref{lem:notcharge1}. Therefore $T$ does not charge $a$. In the case $a$ is a saddle point, the argument is almost verbatim since the weight $s<-1$ can be chosen in such a way that $s>-2$.

Combining the regular-flow and critical-point cases, we conclude that $T$ does not charge any flow line of the gradient flow.
\end{proof}

\myparagraph{Proof of Proposition~\ref{prop:keyprop}.}
We now explain how to conclude the proof of Proposition~\ref{prop:keyprop}. Let $(\chi_\varepsilon)_{\varepsilon>0}$ be a family of smooth cut-off functions such that $\chi_\varepsilon\equiv 1$ on an $\varepsilon$-neighborhood of $\overline{\bigcup_{\ind(a)=1} W^u(a)}$ and $\chi_\varepsilon\equiv 0$ outside a $2\varepsilon$-neighborhood of this set as in the proof of Lemma \ref{lem:notcharge1}. 

For fixed $\varepsilon>0$, consider the map
$
A \in \mathcal{B}_{\mathsf{YM}}^{\alpha,p,s,\ell}
\longmapsto
\bigl(A\chi_\varepsilon,\; A(1-\chi_\varepsilon)\bigr).
$
Let
\[A=A^{\mathsf n}+\sum_{a\in\mathsf{Crit}_1(f)} x_a U_a
\] be the decomposition of $A$ given by Proposition~\ref{prop:keyprop}, then $1-\chi_\varepsilon$ vanishes near $\overline{\bigcup_{\ind(a)=1} W^u(a)}$, so
\[
A(1-\chi_\varepsilon)=A^{\mathsf n}(1-\chi_\varepsilon),
\qquad
A\chi_\varepsilon=A^{\mathsf n}\chi_\varepsilon+\sum_{a\in\mathsf{Crit}_1(f)}x_aU_a.
\]
By Proposition~\ref{prop:notchargeflow}, the distribution $A^{\mathsf n}$ does not charge $\overline{\bigcup_{\ind(a)=1} W^u(a)}$. Hence
\[
\lim_{\varepsilon\to 0^+} A^{\mathsf n}\chi_\varepsilon=0
\qquad\text{and}\qquad
\lim_{\varepsilon\to 0^+} A(1-\chi_\varepsilon)=A^{\mathsf n}.
\]
It follows that the map
\[
A \longmapsto \lim_{\varepsilon\to 0^+}\bigl(A\chi_\varepsilon,\; A(1-\chi_\varepsilon)\bigr)
=
\left(
\sum_{a\in\mathsf{Crit}_1(f)} x_a U_a,
\; A^{\mathsf n}
\right)
\]
is well defined and continuous as a map from $\mathcal{B}_{\mathsf{YM}}^{\alpha,p,s,\ell}$ to
\[
\Big(\mathfrak g\otimes \mathrm{Vect}\big\{U_a:a\in\mathsf{Crit}_1(f)\big\}\Big)
\times
\mathcal{W}^{\alpha,p,s,\ell}.
\]

It remains to prove uniqueness. Suppose that
\[
A=A_{\mathcal N,1}+\sum_{a\in\mathsf{Crit}_1(f)}x_{1,a}U_a
=
A_{\mathcal N,2}+\sum_{a\in\mathsf{Crit}_1(f)}x_{2,a}U_a,
\]
with $x_{i,a}\in\mathfrak g$ and $A_{\mathcal N,i}\in \mathcal{W}^{\alpha,p,s,\ell}$. Then
\[
A_{\mathcal N,1}-A_{\mathcal N,2}
=
\sum_{a\in\mathsf{Crit}_1(f)}(x_{2,a}-x_{1,a})U_a.
\]
The right-hand side is supported on 
$\overline{\bigcup_{a\in\mathsf{Crit}_1(f)} W^u(a)}$, while the left-hand side belongs to
$\mathcal{W}^{\alpha,p,s,\ell}$ and therefore does not charge that set. Hence
$A_{\mathcal N,1}-A_{\mathcal N,2}=0$.
It then follows immediately that $x_{1,a}=x_{2,a}$ for every $a\in\mathsf{Crit}_1(f)$, and the decomposition is unique.
This concludes the proof of Proposition~\ref{prop:keyprop}.

\subsection{Compact injections and weights}\label{CompInj}

We conclude this section by proving necessary continuous injections and compactness results on our spaces $\mathcal{B}_{\mathsf{YM
}}^{\alpha,p,s,\ell}$. This essentially reduces to studying the compact injections of
$\mathcal{W}^{\alpha,p,s,\ell}$. More precisely, we show the following two results, whose proof are similar to the compact injections in the appendix of~\cite{DN}. The same results with the same parameters hold then for $\mathcal{B}_{\mathsf{YM}}^{\alpha,p,s,\ell}$.
\begin{proposition}
    For $\alpha<\alpha', s'<s, \ell'<\ell$ the canonical injections $\mathcal{W}^{\alpha',p,s',\ell'}\hookrightarrow\mathcal{W}^{\alpha,p,s,\ell}$ and $\mathcal{B}^{\alpha',p,s',\ell'}\hookrightarrow\mathcal{B}^{\alpha,p,s,\ell}$ are compact.
\end{proposition}

\begin{proof}
  The proof proceeds in two steps. The first step is local: we show that for any pair of flowboxes \(\square_1 \subsetneq \square_2\), the inclusion of local spaces
\[
\mathcal{W}^{\alpha',\alpha',p}(\square_2)\hookrightarrow \mathcal{W}^{\alpha,\alpha,p}(\square_1)
\]
is compact. Without loss of generality, assume \(\alpha - \frac{1}{p} > 0\) and write \(\square_1 = I \times J_1\), \(\square_2 = I \times J_2\), where \(J_1\) is an interval compactly contained in the interior of \(J_2\). All intervals are compact otherwise the compact injections are wrong. 

Let \(\chi \in C_c^\infty(J_2)\) be a cut-off function depending only on the second variable \(y\), with \(\chi \equiv 1\) on \(J_1\). For any \(T\) defined on \(I \times J_2\), the product \(T\chi\) can be viewed as a function on \(I \times \mathbb{S}^1\) by extending it by zero outside \(J_2\).

It therefore suffices to prove that the map
\[
T \in \mathcal{W}^{\alpha+\varepsilon,\alpha+\varepsilon,p}_{x,y}(I \times J_2,\mathfrak{g}) \mapsto T\chi \in \mathcal{W}^{\alpha,\alpha,p}_{x,y}(I \times \mathbb{S}^1,\mathfrak{g})
\]
is compact, where \(\alpha' = \alpha + \varepsilon\). Taking derivatives in \(y\) and multiplying by an additional cut-off function in \(y\) then yields the desired compact embedding.
\par Since \(T\chi\) is now viewed as a function on \(I \times \mathbb{S}^1\), we may expand it in Fourier modes in the second variable \(y\). We write
\begin{align*}
 \|T\chi\|_{\mathcal{W}^{\alpha,\alpha,p}}
 := \int_{I^2} \frac{ \| T\chi(x_1,\cdot) - T\chi(x_2,\cdot) \|^p_{W^{\alpha,p}_y} }{|x_2 - x_1|^{1+\alpha p}} \, \mathrm{d}x_1 \mathrm{d}x_2.
\end{align*}
Recall that the Littlewood--Paley decomposition
relies on a dyadic partition of unity that writes 
$1=\sum_{j=0}^\infty \psi_j$ where each $\psi_j$ is supported on the dyadic interval $ \{ 2^{j-1} \leqslant \vert \tau\vert \leqslant 2^{j+1} \} $.
Moreover, we may write
\begin{align*}
 \| T\chi(x_1,\cdot) - T\chi(x_2,\cdot) \|^p_{W^{\alpha,p}_y}
 = \sum_{j=0}^\infty 2^{j p \alpha} \left\| \psi_j(\sqrt{-\Delta}) \big( T\chi(x_1,\cdot) - T\chi(x_2,\cdot) \big) \right\|_{L^p_y}^p,
\end{align*}
where we use the characterization of \(\mathcal{W}^{\alpha,p}(\mathbb{S}^1)\) via the Littlewood--Paley decomposition.

\par Assume that we have a bounded sequence $(T_n)_{n\geq 0}$ in the anisotropic space $\mathcal{W}^{\alpha+\varepsilon,\alpha+\varepsilon,p}(I\times J_2)$. The sequence of functions 
$\left( x\mapsto \widehat{T_n\chi}(x,k) \right)_{k\in \mathbb{Z}} $ is bounded in $\mathcal{W}^{\alpha+\varepsilon,p}(I)$ by the \textit{a priori} bound

\[ \int_{I^2} \frac{ \vert \widehat{T\chi}(x_1,k)-\widehat{T\chi}(x_2,k) \vert^p }{\vert x_2-x_1\vert^{1+\alpha p}}    \dd x_1\dd x_2 \lesssim  \Vert T\chi\Vert_{\mathcal{W}^{\alpha,\alpha,p}} \left\langle k \right\rangle^{-\alpha-\varepsilon}, \]
since
\[ \left \vert \left\langle T_n\chi(x_1,.)-T_n\chi(x_2,.), e^{ik.}\right\rangle\right \vert \leqslant C \Vert T_n\chi(x_1,.)-T_n\chi(x_2,.)\Vert_{W^{\alpha+\varepsilon,p}_y} \left\langle k \right\rangle^{-\alpha-\varepsilon } .\]

Assume $\alpha-\frac{1}{p}>0$.
By a diagonal extraction argument, and using the compact injection $\mathcal{W}^{\alpha+\varepsilon,p}(I)\hookrightarrow \mathcal{W}^{\alpha,p}(I)$ where $I$ is a compact interval~\cite[item 4) Thm 2.1]{Vasquez} \footnote{In fact, it is a bit indirect to get the form we need. So first given a function $u\in \mathcal{W}^{s,p}(I)$ extend it as $\tilde{u}\in \mathcal{W}^{s,p}(\mathbb{R})$, the extension is linear continuous by \cite[Thm 5.4 p.~548]{Nezzafractional} and then apply  \cite[item 4) Thm 2.1]{Vasquez} to get the compact injection } ,  
we can extract a subsequence 
so that for every $k$, $\widehat{T_n\chi}(.,k)$ converges in $\mathcal{W}^{\alpha+\varepsilon,p}(I)$ to a limit $\widehat{T_\infty}(.,k)$. We need to prove that up to extraction of the subsequence, the sequence $(T_n\chi)_n$ is a Cauchy sequence in $\mathcal{W}^{\alpha,\alpha,p}(I\times \mathbb{S}^1)$. 
\par By boundedness of the sequence $(T_n\chi)_{n\geq 0} $ in $\mathcal{W}^{\alpha+\varepsilon,\alpha+\varepsilon,p}(I\times \mathbb{S}^1)$
\[
\sup_{n\geq 0}\left\{\int_{I^2} \frac{\sum_{j=0}^\infty 2^{jp(\alpha+\varepsilon)} \Vert \psi_j(\sqrt{-\Delta})( T_n\chi(x_1,.)-T_n\chi(x_2,.))  \Vert_{L^p_y}^p }{\vert x_2-x_1\vert^{1+(\alpha+\varepsilon) p}}  \,  \dd x_1\dd x_2 \right\}\leqslant C.   
\]
We deduce that for every integer $N\geqslant 1$, for all $n$
\[
\int_{I^2 } \frac{\sum_{j=N}^\infty 2^{jp\alpha} \Vert \psi_j(\sqrt{-\Delta}) (T_n\chi(x_1,.)-T_n\chi(x_2,.) ) \Vert_{L^p}^p }{\vert x_2-x_1\vert^{1+\alpha p}}\,    \dd x_1\dd x_2 \leqslant C 2^{-Np\varepsilon}.\]
Now, using the above bound, we can control the following difference term
{\small
\begin{align*}
 &\Vert (T_{n_1}-T_{n_2})\chi\Vert_{\mathcal{W}^{\alpha,\alpha,p}(I\times \mathbb{S}^1)}\\
 &\hspace{1cm}= \int_{I^2} \frac{ \sum_{j=0}^\infty 2^{jp\alpha} \Vert \psi_j(\sqrt{-\Delta})( T_{n_1}\chi(x_1,.)-T_{n_1}\chi(x_2,.)-T_{n_2}\chi(x_1,.)+T_{n_2}\chi(x_2,.)  )\Vert_{L^p(\mathbb{S}^1)}^p }{\vert x_2-x_1\vert^{1+\alpha p}}  \,  \dd x_1\dd x_2 \\
 &\hspace{1cm}\leqslant \int_{I^2} \frac{ \sum_{j=0}^N 2^{jp\alpha} \Vert \psi_j(\sqrt{-\Delta})( T_{n_1}\chi(x_1,.)-T_{n_1}\chi(x_2,.)-T_{n_2}\chi(x_1,.)+T_{n_2}\chi(x_2,.) ) \Vert_{L^p(\mathbb{S}^1)}^p }{\vert x_2-x_1\vert^{1+\alpha p}} \,   \dd x_1\dd x_2  \\
 &\hspace{4cm}+2C  2^{-Np\varepsilon}. 
\end{align*}}
Now by choosing $N$ large enough, we can make the error term $2C  2^{-Np\varepsilon} $ as small as we want.
Since the partial sum \[\sum_{j=0}^N 2^{jp\alpha} \Vert \psi_j(\sqrt{-\Delta})( T_{n_1}\chi(x_1,.)-T_{n_1}\chi(x_2,.)-T_{n_2}\chi(x_1,.)+T_{n_2}\chi(x_2,.) ) \Vert_{L^p}^p\]  \emph{depends only on a finite number of Fourier modes}, letting $n_1,n_2\rightarrow +\infty$, we have for each $j\in \{0,\dots,N\} $, the term \[x\in I\mapsto \Vert \psi_j(\sqrt{-\Delta})\left(T_{n_1}\chi(x,.)-T_{n_2}\chi(x,.) \right)\Vert_{L^p_y}^p \]
goes to $0$ in $\mathcal{W}_x^{\alpha,p}(I)$. Finally, since the term
\[\int_{I^2} \frac{ \sum_{j=0}^N 2^{jp\alpha} \Vert \psi_j(\sqrt{-\Delta})( T_{n_1}\chi(x_1,.)-T_{n_1}\chi(x_2,.)-T_{n_2}\chi(x_1,.)+T_{n_2}\chi(x_2,.)  )\Vert_{L^p}^p }{\vert x_2-x_1\vert^{1+\alpha p}} \,   \dd x_1\dd x_2 \] goes to zero, we can conclude. 
The second step is to globalize taking into account the 
weights. Fix $s'>s,\alpha'>\alpha$. Start from a bounded sequence $(A_k)_k$ in $\mathcal{W}^{\alpha'-1,\alpha',p,s'}$ and just forget about the index $\ell'$ which can be treated exactly in the same way as $s'$. Then by definition, we have
\[\sup_{k\geq 0}\left\{\sum_{\substack{i\in I \\\square_i\cap C_a \neq \emptyset}}\ \sum_{n=0}^\infty  \left(2^{n(s'-\frac{2}{p})}\Vert \mathcal{S}^{2^{-n} \star}_{a} A_k
\Vert_{\mathcal{W}^{\alpha',\alpha'-1;p} (\square_i) }\right)^p \right\}\leqslant C <+\infty . \]
For every $n$, $\mathcal{S}^{2^{-n} \star}_{a} A_k, k\in \mathbb{N}$ is a bounded sequence in $\mathcal{W}^{\alpha',\alpha'-1;p} (\square_i) $ and therefore by a diagonal extraction argument using the compact injection $\mathcal{W}^{\alpha',\alpha'-1,p}(\square_i) \hookrightarrow \mathcal{W}^{\alpha,\alpha-1,p}(\tilde{\square}_i) $ from the previous step up to choosing some strictly smaller rectangles $\tilde{\square}_i$, we may assume that for every $n$,
$ \lim_{k\rightarrow +\infty} \mathcal{S}^{2^{-n} \star}_{a} A_k = \mathcal{S}^{2^{-n} \star}_{a} A_\infty  $ in $\mathcal{W}^{\alpha,\alpha-1,p}(\tilde{\square}_i)$.
For every $N$, we have a bound of the form
\[\sup_{k\geq 0}\left\{\sum_{\substack{i\in I \\\square_i\cap C_a \neq \emptyset}}\ \sum_{n=N}^\infty  \left(2^{n(s-\frac{2}{p})}\Vert \mathcal{S}^{2^{-n} \star}_{a}\left( A_k -A_\infty\right)
\Vert_{\mathcal{W}^{\alpha,\alpha-1;p} (\tilde{\square}_i) }\right)^p \right\}\leqslant 2C 2^{-Np(s'-s)}. \]
By choosing $N$ large enough we can make it smaller than $\frac{\varepsilon}{2}$. Then the finite sum
\[\sum_{\substack{i\in I \\\square_i\cap C_a \neq \emptyset}}\ \sum_{n=0}^N  \left(2^{n(s-\frac{2}{p})}\Vert \mathcal{S}^{2^{-n}\star}_{a} \left( A_k-A_\infty\right)
\Vert_{\mathcal{W}^{\alpha,\alpha-1;p} (\tilde{\square}_i) }\right)^p \]
goes to $0$ when $k\rightarrow +\infty$ 
since $ \lim_{k\rightarrow +\infty} \mathcal{S}^{2^{-n} \star}_{a} A_k = \mathcal{S}^{2^{-n} \star}_{a} A_\infty  $ in $\mathcal{W}^{\alpha,\alpha-1,p}(\tilde{\square}_i)$. The weight $\ell$ is treated exactly in the same way as the weight $s$ which concludes the proofs of the compact injections.
\end{proof}

\begin{proposition}
    For $ s'<s, \ell'<\ell$, $\alpha'>\alpha-\frac{1}{p}  $ the canonical injection $\mathcal{W}^{\alpha',p,s',\ell'}\hookrightarrow \mathcal{W}^{\alpha-\frac1p,\infty,s,\ell}$ is compact. 
\end{proposition}
\begin{proof}
The proof is very similar to the previous one. For step one, given any pair of flowboxes $\square_1\subsetneq \square_2$, we need to establish the compact embedding of ``local spaces''
$$\mathcal{W}^{\alpha',\alpha',p}(\square_2)\hookrightarrow  \mathcal{W}^{\alpha-\frac{1}{p},\alpha-\frac{1}{p},\infty}(\square_1).$$
 Only two things must be changed. We use the compact injection $W^{\alpha',p}(I) \hookrightarrow W^{\alpha-\frac{1}{p},\infty}(I) $ to deduce that $\widehat{T_n\chi}(.,k)$ converges in $W^{\alpha-\frac{1}{p},\infty}(I)$ to $\widehat{T_\infty\chi}(.,k)$. By the boundedness of $(T_n\chi)_n$ in $\mathcal{W}^{\alpha',\alpha',p}$,  we need to control:
\[
\sup_{x_1\neq x_2\in I^2 } \sup_{j\geq 0}  \left\{  \frac{2^{j(\alpha-\frac{1}{p})} \Vert \psi_j(\sqrt{-\Delta})( (T_n-T_\infty)\chi(x_1,.)-(T_n-T_\infty)\chi(x_2,.))  \Vert_{L^\infty_y} }{\vert x_2-x_1\vert^{(\alpha-\frac{1}{p})}} \right\}
\]
in terms of
\[
\sup_{n\geq 0}\int_{I^2} \frac{\sum_{j=0}^\infty 2^{jp\alpha'} \Vert \psi_j(\sqrt{-\Delta})( T_n\chi(x_1,.)-T_n\chi(x_2,.))  \Vert_{L^p_y}^p }{\vert x_2-x_1\vert^{1+\alpha' p}}  \,  \dd x_1\dd x_2 \leqslant C.     
\]

But we find from the existence of compact injections that
\begin{align*}
    &\sup_{j\geqslant N}\left\{2^{j(\alpha-\frac{1}{p})} \Vert \psi_j(\sqrt{-\Delta})( T\chi(x_1,.)-T\chi(x_2,.))  \Vert_{L^\infty_y} \right\} \\
    &\hspace{2cm}\leqslant C_N  \left( \sum_{j=0}^\infty 2^{jp(\alpha')} \Vert \psi_j(\sqrt{-\Delta})( T\chi(x_1,.)-T\chi(x_2,.))  \Vert_{L^p_y}^p \right)^{\frac{1}{p}} 
\end{align*}
where $C_N\rightarrow 0$ when $N\rightarrow +\infty$ and $C_N$ does not depend on $T$.

Therefore 
$$ 
\sup_{x_1\neq x_2\in I^2 } \sup_{j\geqslant N}    \frac{2^{j(\alpha-\frac{1}{p})} \Vert \psi_j(\sqrt{-\Delta})( (T_n-T_\infty)\chi(x_1,.)-(T_n-T_\infty)\chi(x_2,.))  \Vert_{L^\infty_y} }{\vert x_2-x_1\vert^{(\alpha-\frac{1}{p})}} 
$$
can be made arbitrarily small by choosing $N$ large enough and 
$$ 
\sup_{x_1\neq x_2\in I^2 } \sup_{j\leqslant N}    \frac{2^{j(\alpha-\frac{1}{p})} \Vert \psi_j(\sqrt{-\Delta})( (T_n-T_\infty)\chi(x_1,.)-(T_n-T_\infty)\chi(x_2,.))  \Vert_{L^\infty_y} }{\vert x_2-x_1\vert^{(\alpha-\frac{1}{p})}} 
$$
goes to $0$ when $n\rightarrow +\infty$ since the term depends on finite number of Fourier modes and $\widehat{T_n\chi}(.,k) \rightarrow \widehat{T_\infty\chi}(.,k) $ in $W^{\alpha-\frac{1}{p},\infty}(I)$.

The second step dealing with the weight is almost verbatim the same as for the previous proposition.
\end{proof}

\section{The Yang--Mills random field} \label{s3}
We review the construction from~\cite{BCDRT,DN} of the Yang--Mills measure (in the Morse gauge) on $\Sigma$. 
Starting with the free boundary 
Yang--Mills measure on the surface $\mathcal{S}$---obtained by blowing up the maximum of $f$ into a boundary circle---one then conditions the boundary holonomy to be trivial in order to close the surface. 
This paragraph justifies the existence of this random connection in the spaces $\mathcal{B}_{\mathsf{YM}}^{\alpha,p,s,\ell}$. We start with an informal exposition of the construction of the free measure. 

\subsection{The free boundary Yang--Mills measure}
In this paragraph, we give an informal account of the construction of the YM measure using the Morse gauge introduced in~\cite{BCDRT,DN}. This is a novel dynamical gauge fixing which allows to abelianize the Yang--Mills theory on surfaces of any genus at the price of introducing singularities along curves which act as \emph{defect lines}.
Conceptually, it yields a slicing of the space of connections, and on each slice, an infinite dimensional normal form for the Yang--Mills action functional. This introduces new coordinates $\big(B,(g_a)_{a\in\mathsf{Crit}_1(f)})$ on the moduli space $ \mathcal{A} /\mathcal{G} $ of connections. In this set of coordinates, $B$ is a $1$-form that only has a component in the orthogonal direction to the gradient of $f$, and $(g_a)_{a\in \mathsf{Crit}_1(f)}$ is a family of group elements.  In these coordinates, the Yang--Mills functional becomes quadratic in $B$, and the Morse gauge acts as an infinite dimensional Morse chart for the Yang--Mills functional where it becomes \emph{quadratic in one of the variables}. Let us explain how these coordinates arise.
\par We work on the blow--up surface $\mathcal{S}$ and lift the gradient field $V=\nabla f$ that vanishes at $\max(f)$ as a $b$--vector field still denoted by $V$~\footnote{informally, a $b$--vector field is a $C^\infty$ vector field tangent to the boundary of $\mathcal{S}$. In the present case, the boundary $\partial\mathcal{S}$ is given by $r=0$ for the boundary defining function $r=\sqrt{x^2+y^2}$ where $V=-x\partial_x-y\partial_y$, then the  $b$--vector field $V$ reads $-r\partial_r$ in this coordinate} on $\mathcal{S}$ that vanishes along the boundary $\partial\mathcal{S}$. Given the gradient vector field $V= \nabla f$ with $f$ a Morse function, we can define the notion of Morse gauge: for any smooth connection $\nabla=\dd+A$, there exists a family of gauge transformations $\left(g_T\right)_{T\geqslant 0}$ on $\mathcal{S}$ such that \[\lim_{T\rightarrow +\infty}g^{-1}_T(\dd+A) g_T\] yields trivial parallel transport along the integral curves of $V$:
  \[ \iota_V \big( \lim_{T\rightarrow +\infty} g_T^{-1}\dd g_T+g_T^{-1}Ag_T \big)=0 .\]
 The fact that the Morse gauge is a genuine gauge transformation is discussed in~\cite{BCDRT}.
\par The limiting connection that we obtain after the dynamical procedure \[A_\infty=\lim_{T\rightarrow +\infty} g_T^{-1}\dd g_T+g_T^{-1}Ag_T \] in the space of $\mathfrak{g}$ valued $1$--forms of suitable negative Sobolev regularity decomposes as
\begin{equation}\label{eq:MorseHodge}
 A_\infty=\underset{\in \mathrm{Im}(\iota_V)}{ \underbrace{ \iota_V \beta }} + \underset{\text{zero modes}}{\underbrace{\sum_{a\in \mathsf{Crit}_1(f)} {\rm m}_a U_a  }},
\end{equation} 
where $\beta$ is a $2$--form, smooth outside $\overline{\cup_{a\in \mathsf{Crit}_1(f)}W^u(a)}$, the $m_a$ are elements in $\g$, and each $U_a, a\in\mathsf{Crit}_1(f)$ is the current of integration on the unstable curve $W^u(a)$. The right hand side of the above decomposition involves a direct sum and the space $\left(\ker(\mathcal{L}_V)\cap \Omega^1\right)$ represents the zero modes of the Lie derivative $\mathcal{L}_V$ acting on the space of $1$--forms. This is a Morse theoretic decomposition reminiscent of the Hodge decomposition in elliptic theory where usual Sobolev spaces are replaced by anisotropic spaces. 
\par Endow the closed  surface $\Sigma$  with an area form $\sigma$ which induces an area form on $\mathcal{S}$ and that we denote by $\sigma$ as well. This area form induces a Hodge $\star$ operator acting on  $\Omega^0(\mathcal{S})\rightarrow \Omega^2 (\mathcal{S})$.  Under the Morse gauge, $\iota_VA=0$, and we may rewrite the Yang--Mills action functional as
\[S_{\mathsf{YM}}(A) = - \lim_{\varepsilon\rightarrow 0^+} \int_{\mathcal{S}\setminus C_\varepsilon} \Tr\left(\dd A\wedge \star \dd A \right) \]
where $C_\varepsilon$ is a small $\varepsilon$--neighborhood collar around
the boundary $\partial\mathcal{S}$.
\par The goal of the next paragraph is to properly motivate our definition of the physicist's path integral $
   \exp\big( -S_{\mathsf{YM}}(A) \big) \mathcal{D} A $ .
\myparagraph{Physicist's functional integral.}
Set $B=\dd A\in \Omega^2(\mathcal{S})$, with the constraint $A\in \ker(\iota_V)$. Then, in some sense, we can establish that the de Rham boundary operator $\dd: A\in \ker(\iota_V) \mapsto B=\dd A $ 
can be inverted by the chain homotopy operator $\iota_V\mathcal{L}_V^{-1}$ up to elements in $\ker\left( \mathcal{L}_V\right)\cap \Omega^1$. This suggests that any connection in the Morse gauge admits a decomposition of the form~\cite{BCDRT, DN}:
\begin{align}
    \label{eq:A_specific_form}
    A = & \iota_V\mathcal{L}_V^{-1}B+\sum_{a\in \mathsf{Crit}_1(f)} \log(g_a) U_a,
\end{align}
where $g_a \in G$ and $U_a$ are the unstable currents~\cite[Prop 6.6 p.~1437]{DR19}. Here we have made the choice of a logarithm $\log: G \rightarrow \mathfrak{D}\subset \mathfrak{g}$, taking values in a fundamental domain $\mathfrak{D}$. For convenience, we write $A_{\mathcal{N}}  \coloneq  \iota_V\mathcal{L}_V^{-1}B$. 
\par Since unstable currents are de Rham closed --- which follows from unstable curves being closed curves or from the instanton formula~\cite[Thm 2.5 p.~1803]{DR21}--- we indeed have $\dd A = B$. Thus the Yang Mills functional reads
\[ S_{\mathsf{YM}}(A) = -\int_{\Sigma} \Tr\left(B\wedge \star B \right)\]
in terms of $B$. In particular, given that $A$ is defined from the data of $\left(B, (g_a)_{a \in \mathrm{Crit}(f)_1} \right)$, the $g_a$'s contribution vanishes through the action functional. In physics, such variables are called \textit{zero modes}, and they need to be attributed a natural finite-dimensional measure on $G^{2g}$. Unsurprisingly, in order to agree with the theoretical physics literature, the appropriate choice will prove to be the Haar measure $\prod_a \dd g_a$.
\par Expressing the $2$-form $B$ in terms of the area form $\sigma$ yields $B=\xi\sigma$ where $\xi$ is a Lie algebra valued distribution and the Yang--Mills functional becomes
$  \int_{\mathcal{S}} \vert \xi\vert_{\mathfrak{g}}^2\sigma $. 
The above discussion is summarized by saying that
\begin{align*}
S = - \lim_{\varepsilon\rightarrow 0^+} \int_{\mathcal{S}\setminus C_\varepsilon } \Tr\left(\dd A\wedge \star \dd A \right), \text{ for } A\in \Omega^1(\mathcal{S},\mathfrak{g}) \text{ s.t. } \iota_VA=0,
\end{align*}
corresponds to
\begin{align*}
S = \int_{\mathcal{S}} \vert \xi\vert_{\mathfrak{g}}^2\sigma \text{ for }A=\iota_V\mathcal{L}_V^{-1}\left(\xi\sigma \right)+\sum_{a\in\mathsf{Crit}_1(f)} \log g_a U_a , \ \xi\in C^\infty(\mathcal{S},\mathfrak{g}), g_a \in G. 
\end{align*}

The justification of the existence of the term $\iota_V\mathcal{L}_V^{-1}\left(\xi\sigma \right)$ comes from solving stochastic cohomological equations in \cite{BCDRT}, or using pseudo-coordinates in~\cite{DN}. The former relies on some spectral gap estimates for the transfer operator associated with the gradient flow proved in \cite{BCDRT}. In the functional integral, for any gauge invariant observable $F$, we write
\begin{align*}
   & \int_{A\in \Omega^1(\mathcal{S},\mathfrak{g}),\iota_VA=0} F(A) e^{-S_{\mathsf{YM}}(A)} \dd A \\
 &\hspace{2cm}= \int_{G^{2g}} \prod_{a\in \mathsf{Crit}_1(f)} \dd g_a 
 \\&\hspace{4cm}\cdot
\int_{\xi\in \Omega^0(\mathcal{S},\mathfrak{g})}
 \dd\xi\, 
e^{ -\int_{ \mathcal{S}} \vert \xi\vert_{\mathfrak{g}}^2\sigma }  \left| \det( \iota_V\mathcal{L}_V^{-1}) \right|
F\Big(\iota_V\mathcal{L}_V^{-1}(\xi\sigma)+\sum_{a\in\mathsf{Crit}_1(f)} \log g_aU_a \Big),
\end{align*}
under the formal change of variable 
$A=\iota_V\mathcal{L}_V^{-1}\left(\xi\sigma \right)+\sum \log(g_a)  U_a$ inside the functional integral. Furthermore, when normalized into a probability measure, the formal Jacobian $\left| \det( \iota_V\mathcal{L}_V^{-1}) \right|$ is a constant that simplifies.
At this stage, the kinetic part of the path integral $ \dd \xi
e^{ -\int_{\mathcal{S}} \vert \xi\vert_{\mathfrak{g}}^2\sigma }$ should be interpreted as white noise with values in $\gfrak$. This leads to the following definition of the (free boundary) Yang Mills measure on the surface $\mathcal{S}$. Let $\xi$ be a $\gfrak$-valued Gaussian process on $\mathcal{S}$ with induced area form that we still denote by $\sigma$, with zero mean and covariance structure given by 
\[
\E\left[ \xi(f_1) \xi(f_2) \right]
= 
\int_{\mathcal{S}} \langle f_1, f_2 \rangle_\gfrak \sigma
\ ,
\]
for any smooth functions $f_1, f_2: S \rightarrow \gfrak$, and where $\xi(f)$ is a shorthand notation for $\int_{\mathcal{S}} \left\langle\xi,f\right\rangle_{\mathfrak{g}}\sigma$.

\begin{definition}[Free boundary Yang--Mills measure] \label{FreeYMM}
\label{definition:YM_free_boundary}
Define the random connection $A^\mathsf{Free}_{\mathcal{S},\sigma}$ on $\mathcal{S}$ via its expectation for any functional $F$  by
\[
\E \Big[ F\big(A^\mathsf{Free}_{\mathcal{S},\sigma} \big)\Big]
=
\int_{G^{2g}} \prod_{a\in \mathsf{Crit}_1(f)} \dd g_a\, 
\E\Bigg[ F\Big( \iota_V\mathcal{L}_V^{-1}\left(\xi\sigma \right)+\sum_{a'\in\mathsf{Crit}_1(f)} \log g_{a'}[U_{a'}] \Big)  \Bigg].
\]
The law of $A^\mathsf{Free}_{\mathcal{S},\sigma}$ seen as a measure on $\mathcal{D}'^1(\mathcal{S},\g)$ is called the free boundary  Yang--Mills measure on $\mathcal{S}$. 
\end{definition}
It is, of course, a task to show that the above definition makes sense. This work was carried out in \cite{DN,BCDRT}. Here we will only state the precise existence theorem for this measure.

The following theorem is a rigorous result on the existence of the Yang--Mills measure. It has been proved in both \cite{DN} and \cite{BCDRT}. In \cite{DN}, it is proved using a scaling limit argument, meaning that there exists a sequence of lattice measures that converge, in the spaces $\mathcal{B}_{\mathsf{YM}}^{\alpha,p,s,\ell}$, to the continuous measure. In \cite{BCDRT}, it was shown using mollifications that the Yang--Mills measure can be obtained as a limit in some weighted Sobolev space of Sobolev regularity $<-1$. In the appendix, we recall how to adapt the arguments of \cite{BCDRT} to obtain a convergence in  $\mathcal{B}_{\mathsf{YM}}^{\alpha,p,s,\ell}$ instead of only in some less sharp weighted Sobolev space.

\begin{thm}[The free boundary Yang--Mills measure \cite{DN, BCDRT}, Appendix \ref{s:justifyformulaDAS}] \label{MeasureTheorem}
    For $\alpha<\frac{1}{2}$, $p\geq 1$, $s<0$, and $\ell$ small enough, there exists a $\mathcal{B}_{\mathsf{YM}}^{\alpha,p,s,\ell}$-valued random variable $A^\mathsf{Free}_{\mathcal{S},\sigma}$ such that
    \begin{itemize}
        \item The solution $(U_\varepsilon(1))_{\varepsilon>0}$ at time $1$  of the $\epsilon$-mollification of the SDE defining the holonomy along level curves:
        \[\dd U_\varepsilon=U_\varepsilon \dd W_\varepsilon, U_\varepsilon(0)=1_G, \,\ W_\varepsilon(t):= \int_{\mathbb{R}} \chi_\varepsilon(t-r)\left(\int_0^r
        \gamma^*A^\mathsf{Free}_{\mathcal{S},\sigma} \right) \dd r \]
        converges a.s., when $\varepsilon\rightarrow 0^+$, to a random variable $\mathsf{Hol}_{\gamma}\big (A^\mathsf{Free}_{\mathcal{S},\sigma}\big )$  measurable w.r.t. $A^\mathsf{Free}_{\mathcal{S},\sigma}$.
        \item The family of random variables $\big(\mathsf{Hol}_{\{f=r\}}(A^\mathsf{Free}_{\mathcal{S},\sigma})\big)_{a_{2g}\leq r\leq \max f}$ is a Markov process.
        \item The law of $\mathsf{Hol}_{\{f=r\}}$ is given by 
        \begin{equation*}
\mathbb{P}\big[\mathsf{Hol}_{\{f=r\}}\big(A^{\mathsf{Free}}_{S,\sigma}\big)\in \dd g \big]= \sum_{\rho\in \widehat{G}}  \exp\Big(-\sigma\big(\Sigma_{\leq r}\big) c_2(\rho)\Big)  \dim(V_\rho)^{1-2g} \chi_\rho(g) \,\dd g=:Z_{\Sigma_{\leq r}, \sigma},
\end{equation*}
where $Z_{\Sigma,\sigma}$ is the Yang--Mills partition function of the surface $\Sigma$ equipped with area form $\sigma$.
    \end{itemize}
\end{thm}
For the precise reason of why this random variable is indeed the free boundary Yang--Mills measure, we refer the reader to \cite{DN,BCDRT}.
\subsection{The Yang--Mills pre-measure on closed surfaces}
\label{ss:YMmeasureconditioning}
Definition \ref{definition:YM_free_boundary} and the associated results describe a measure with two independent components: the gaussian component $A_{\Nc}$ and the component associated to the zero modes $\sum_{a\in\mathsf{Crit}_1(f)}U_a\log g_a$. In order to define the Yang--Mills measure on the closed surface $\Sigma$, those two random variables will be coupled through a (non-linear) conditioning. The intuition behind this procedure is two-fold. 
\begin{itemize}
    \item Geometrically, this amounts to ``plugging the hole'' at $\max f$ by prescribing the appropriate condition $\mathsf{Hol}_{\partial \mathcal{S}}\big(A^{\mathsf{Free}}_{S,\sigma}\big)=1_{G}$. 
    \item Probabilistically it is reminiscent of the construction of the Brownian bridge from the  Brownian motion by conditioning its end value to be $0$.
\end{itemize}
A similar idea played a central role in the work of Sengupta~\cite{Sengupta97}. We will start by identifying the expected law of the closed measure. This does not give for the moment one proper measure, but rather, a family of projective measures. It will be shown in later sections that one could define one single measure out of this family by tightness.
\par For $\epsilon>0$ let $\psi_\epsilon:\Sigma \rightarrow [0,1]$ be a smooth function such that 
\[\psi_\epsilon(x)=\begin{cases}
    1 \text{ if } d(x,\max f)\geq 3\epsilon, \\
    0 \text{ if } d(x,\max f)\leq \epsilon.
\end{cases}\]
The notion $\Sigma_{\geq r}$ refers to the part of the surface such that $f\geq r$. Refer to Appendix \ref{symb} for notations.

\begin{definition}[The Yang--Mills pre-measure] \label{ClosedYMM}
For $\epsilon>0$, define the $\mathcal{B}_{\mathsf{YM}}^{\alpha,p,s,\ell}$-valued random variable $A_{\Sigma,\sigma}^\epsilon$ as follows. For any $F:\mathcal{B}_{\mathsf{YM}}^{\alpha,p,s,\ell}\rightarrow \R_+$ measurable, set
 \begin{equation}  \label{FiniteDimClosed}
 \mathbb{E}\Big[F\big(A^\epsilon_{\Sigma,\sigma}\big) \Big]=\frac{\mathbb{E} \Bigg[ 
 F
 \big(A^{\mathsf{Free}}_{S,\sigma}\psi_\epsilon\big) 
 \cdot 
 p_{
 \sigma\big(\Sigma_{\geq \max f- \epsilon}\big)
 }
 \Big(\mathsf{Hol}_{\max f-\epsilon}\big(A^{\mathsf{Free}}_{S,\sigma}\big)\Big)
 \Bigg]}
 {\mathbb{E}\Big[ p_{
 \sigma\big(\Sigma_{\geq \max f -\epsilon}\big)
 }
 \Big(\mathsf{Hol}_{\max f-\epsilon }\big(A^{\mathsf{Free}}_{S,\sigma}\big)\Big)\Big] },
 \end{equation}   
 where $A^{\mathsf{Free}}_{S,\sigma}\psi_\epsilon$ is short for 
 \[\left(\pi_\mathsf{noise}\cdot A^{\mathsf{Free}}_{S,\sigma}\right)\psi_\epsilon+\pi_{\mathsf{flat}}\cdot A^{\mathsf{Free}}_{S,\sigma}. \]
\end{definition}
We refer to Appendix~\ref{s:justifyformulaDAS} 
where we justify formula~(\ref{FiniteDimClosed}), precisely for the problem of defining the holonomy map for $A^{\mathsf{Free}}_{S,\sigma}$.
\par The above definition defines a family of random variables $\big(A^\epsilon_{\Sigma,\sigma}\big)_{\epsilon>0}$. It does not however define one single random variable that can be called the Yang--Mills measure on the closed surface. However, by compatibility provided by the Markov property of the family of random variables $\big(\mathsf{Hol}_{\{f=r\}}(A^\mathsf{Free}_{\mathcal{S},\sigma})\big)_{a_{2g}\leq r\leq \max f}$, it allows to compute without ambiguity the expectation of any observable measurable with respect to $\Sigma_{\leq \max f-\epsilon}$. In the next section, we will show that $\big(A^\epsilon_{\Sigma,\sigma}\big)_{\epsilon>0}$ is tight. Then, the convergence is straightforward and its limit will be called the Yang--Mills measure on the closed surface $\Sigma$.

\subsection{The Yang--Mills measure as a Borel probability measure on $\mathcal{B}_{\mathsf{YM}}^{\alpha,p,s,\ell}$} \label{s6}
In this section, we want to show that the family of measures that we defined in Definition \ref{ClosedYMM} induce a Borel probability measure on the space $\mathcal{B}_{\mathsf{YM}}^{\alpha,p,s,\ell}$.  The main theorem of this section is as follows.
\begin{thm} [Yang--Mills measure] \label{ClosedYMMeasure}
    For $\alpha<\frac 1 2$, $p\geq 1$, $\ell<1$ , and $s<0$,  as $\epsilon\to 0$, $\big(A^\epsilon_{\Sigma,\sigma}\big)_{\epsilon>0}$ converges, in the sense of probability measures on $\mathcal{B}_{\mathsf{YM}}^{\alpha,p,s,\ell}$, to a random variable $A_{\Sigma,\sigma}$. This random variable is called the Yang--Mills random field in Morse gauge.
\end{thm}
\begin{proof}
    First note that it is enough to show tightness to conclude the convergence. By the compact injections of Section \ref{CompInj},  to show tightness of the family $\big(A^\epsilon_{\Sigma,\sigma}\big)_{\epsilon>0}$, it is enough to show 
    \[\sup_{\epsilon>0} \mathbb{E}\Big[\big \| A^\epsilon_{\Sigma,\sigma}\big\|_{\mathcal{B}_{\mathsf{YM}}^{\alpha,p,s,\ell}}^p\Big] <\infty.\]
    This will be showed throughout this section. 
\end{proof}
Recall that
\[\big\|A^\epsilon_{\Sigma,\sigma}\big\|_{\mathcal{B}_{\mathsf{YM}}^{\alpha,p,s,\ell}}=\big \| \pi_{\mathsf{noise}}\cdot A^\epsilon_{\Sigma,\sigma}\big\|_{\mathcal{W}^{p;\alpha,\alpha-1,s,\ell}}+\max_{a\in\mathsf{Crit}_1(f)}\big\|\mathsf{Ext}_a\big(A^\epsilon_{\Sigma,\sigma}\big)\big\|_\g,\]
and we have to deal with these two terms. Let us start with the easiest of both, the second term. 
\myparagraph{Tightness of the flat term $ \mathsf{Ext}_a\big(A^\epsilon_{\Sigma,\sigma}\big)$.}
In the case of the free boundary Yang--Mills measure, we have a nice decomposition as the independent sum of a noise term, the term $\mathcal{L}_V^{-1}\left(\iota_V\left(\xi \sigma\right)\right)$, and a random flat connection term given by 
\[\sum_{a\in \mathsf{Crit}_1(f)} \log g_a U_a,\]
where the $g_a$ are independent, and Haar distributed.
After closing the surface, \textit{i.e.} after conditioning the holonomy on $\partial \mathcal{S} $ to be $1_G$, we loose the independence. However, using the fundamental decomposition lemma in the space $\mathcal{B}^{\alpha,p,s,\ell}_{\mathsf{YM}}$, we can still extract a structural formula of this type for the closed Yang--Mills measure. More precisely, we have the following Lemma: 
\begin{lemma}\label{structureLemma}
    For every $\epsilon>0$, $A^\epsilon_{\Sigma,\sigma}$ can be written as  \[A^\epsilon_{\Sigma,\sigma}= A^\epsilon_{\mathcal{N};\Sigma;\sigma} +\sum_{a\in \mathsf{Critf}_1(f)}x^\epsilon_a U_a,\]
    where 
    \[\mathbb{P}\left[\bigcap_{a\in  \mathsf{Critf}_1(f) } \big\{x^\epsilon_a\in B_\g(0,\dim G)\big\}\right]=1.\]
    Moreover, the joint law of the $x_a^\epsilon$ does not depend on $\epsilon$.
\end{lemma} 
\begin{proof}
   We drop for simplicity the subscript $\epsilon$. From Proposition \ref{prop:keyprop}, we can write this random element uniquely as 
    \[A_{\Sigma,\sigma}=\pi_{\mathsf{noise}}\left( A_{\Sigma,\sigma} \right) +\sum_{a\in \mathsf{Critf}_1(f)}\mathsf{Ext}_a\big(A_{\Sigma,\sigma}\big) U_a.\]
    It remains to show that almost surely, $\mathsf{Ext}_a\big(A_{\Sigma,\sigma}\big)\in B_\g(0,\mathsf{diam}\, G)$. To do this, we note that the extraction map verifies $\mathsf{Ext}_a\big( A_{\Sigma,\sigma} \big)=\mathsf{Ext}_a\big( A_{\Sigma,\sigma}\psi \big),$
    for any $\psi$ that is equal to $1$ for $r>f(a_{2g})$, where $a_{2g}$ is the last saddle point of $f$. This lets us write, for any $a\in\mathsf{Crit}_1(f)$, 
        \[\mathbb{E}\Big[1_{d(\mathsf{Ext}_a,1)\leq \mathsf{diam}\,G}\big (A_{\Sigma,\sigma})\Big]=\frac{\mathbb{E}\Big[1_{d(\mathsf{Ext}_a,1)\leq \mathsf{diam}\, G}(A^\mathsf{Free}_{\Sigma,\sigma})\cdot p_{\sigma\big(\Sigma_{\geq a_{2g}}\big)}\left(\mathsf{Hol}_{\max f-a_{2g}}\big(A^\mathsf{Free}_{\Sigma,\sigma}\big)\right)\Big]}{\mathbb{E}\Big[
        p_{\sigma\big(\Sigma_{\geq a_{2g}}\big)}\left(\mathsf{Hol}_{\max f-a_{2g}}\big(A^\mathsf{Free}_{\Sigma,\sigma}\big)\right)\Big]},\]
        which is equal to $1$ since almost surely, $1_{d(\mathsf{Ext}_a,1)\leq \mathsf{diam}\, G}(A^\mathsf{Free}_{\Sigma,\sigma})=1$.
\end{proof}
As a direct corollary of the previous lemma, we get 
\[\sup_{\epsilon>0}\mathbb{E}\Big[\max_{a\in\mathsf{Crit}_1(f)}\big\|\mathsf{Ext}_a\big(A^\epsilon_{\Sigma,\sigma}\big)\big\|_\g^p\Big]\leq \mathsf{diam}(G)^p.\]
\myparagraph{Tightness of the noise term $\pi_{\mathsf{noise}}\cdot A^\epsilon_{\Sigma,\sigma}$.} \label{BoundNoise}
In this section, we would like to show that under suitable paramters $\alpha,p,s,\ell$, 
\[\sup_{\epsilon>0}\mathbb{E}\Big[\big \| \pi_{\mathsf{noise}}\cdot A^\epsilon_{\Sigma,\sigma}\big\|_{\mathcal{W}^{p;\alpha,\alpha-1,s,\ell}}^p\Big] < \infty.\]
Recall from Definition \ref{WeightedSemiNorms} that 
\begin{align*}
    &\Vert A\Vert_{\mathcal{W}^{p;\alpha,\alpha-1,s,\ell}}^p\\
    &\hspace{2cm}=\sum_{a\in \mathsf{Crit}_1(f)}\Vert A\Vert_{\mathcal{W}^{p;\alpha,\alpha-1,s} (D_a) }^p
+ \sum_{a\in \{\min(f),\max(f)\} } \Vert A\Vert_{\mathcal{W}^{p;\alpha,\alpha-1,\ell} (D_{a}) }^p
+\sum_{i\in I}  \Vert A \Vert_{\mathcal{W}^{\alpha,\alpha-1,p}(\square_i) }^p,
\end{align*}
and therefore there are many different terms to bound. 
\begin{remark}\label{remcond}
Observe that when we rewrite this norm as 
\begin{align*}
    &\Vert A\Vert_{\mathcal{W}^{p;\alpha,\alpha-1,s,\ell}}^p\\
    &\hspace{2cm} = \Vert A\Vert_{\mathcal{W}^{p;\alpha,\alpha-1,\ell} (D_{\max f}) }^p\\
    &\hspace{4cm}+\underbrace{\sum_{a\in \mathsf{Crit}_1(f)}\Vert A\Vert_{\mathcal{W}^{p;\alpha,\alpha-1,s} (D_a) }^p
+  \Vert A\Vert_{\mathcal{W}^{p;\alpha,\alpha-1,\ell} (D_{\min f}) }^p
+\sum_{i\in I}  \Vert A \Vert_{\mathcal{W}^{\alpha,\alpha-1,p}(\square_i) }^p}_{\|A\|^p_{\leq a_{2g}}},
\end{align*}
we see that the term under bracket is measurable with respect to the information located at a  distance of $d(a_{2g},\max f)$ from the maximum. Therefore, we have directly that 
\[\mathbb{E}\Big[\big\|  \pi_{\mathsf{noise}}\cdot A^\epsilon_{\Sigma,\sigma} \big\|^p_{\leq a_{2g}}\Big] \lesssim \mathbb{E}\Big[\big\|  \pi_{\mathsf{noise}}\cdot A^\mathsf{Free}_{\Sigma,\sigma} \big\|^p_{\leq a_{2g}}\Big].\]
These norms in the free case were treated in details in \cite[Section 5]{DN}. For the sake of completeness, we will expose again the treatment of one of the terms, for example near $\min f$.
\par It is important to note that this is not anymore true for the first term which gets closer and closer to the maximum, and conditioning needs to be taken care of. 
\end{remark}

\paragraph{Near $\min(f)$.}
Near the minimum, there are local polar coordinates, $(r,\theta)$ where $r=0$ is the minimum. In these coordinates, we have 
\[V=r\partial_r,\ \ \text{ and }\ \  \sigma=r\rho(r,\theta)\,\dd r\dd\theta,\]
where $\rho\in C^\infty([0,1]\times \mathbb{S}^1)$ and $\rho(0,.)\equiv C>0$ that we may choose equal to $1$ for simplicity.
Let us start with a proposition. We will introduce a scaling parameter of the total area by $t$ as it will be of interest for later.
\begin{proposition}\label{ScaleFreeMin}
    We have, for  $\alpha<\frac 12$, $p\geq 1$ and all $0<t<1$ 
    \[\mathbb{E}\Big[\big\Vert \mathcal{S}^{2^{-n}\star}\left( \pi_{\mathsf{noise}}\cdot A^\epsilon_{\Sigma,t\sigma} \right)\big  \Vert^{2p}_{\mathcal{W}^{\alpha,\alpha-1,2p}(I_1\times J_1)} \Big] \lesssim t^{p}2^{-2np},\]
    where the bound is uniform in $t,\epsilon$ and $n$.
\end{proposition}
\begin{proof}
Recall from Remark \ref{remcond} that, near the minimum, it is enough to treat the free case. 
We denote by $\xi_\sigma$ the white noise 
associated with the area form $\sigma$, we use the identity
$\xi_{t\sigma}= t^{-\frac{1}{2}}\xi_\sigma $ for all $t>0$.
We have, in law, 
\[\pi_{\mathsf{noise}}\cdot A^\mathsf{Free}_{\Sigma,t\sigma}=A_{\mathcal{N};S;t\sigma}=\iota_V\mathcal{L}_V^{-1}\left(\xi_{t\sigma} \right)=\sqrt{t}\mathcal{L}_V^{-1} \left(\iota_V \xi_\sigma \sigma \right) ,\]
and the parameter $\sqrt{t}$ factors out.
We would like to control the anisotropic norm of 
$W$ on some product of intervals $I_1\times I_2=[\frac{1}{2},1]_r\times [a,b]_\theta$, where $W$ is such that
\[A_{\mathcal{N};S;t\sigma}(r,\theta)= \left(\partial_\theta W \right) (r,\theta)  \text{ and } W(.,0)=0.\] 
The fact that we are allowed to primitive $A_{\mathcal{N};S;t\sigma}$ 
is recalled in the appendix using Cauchy sequences arguments.
We have
\begin{align*}
\Vert W \Vert^{2p}_{\mathcal{W}^{\alpha,\alpha;2p}_{r,\theta}(I_1\times J_1)} \coloneq  \int_{I_1^2\times I_2^2} \frac{\vert W(r_1,\theta_1)-W(r_2,\theta_1)- W(r_1,\theta_2)+W(r_2,\theta_2)  \vert_{\mathfrak{g}}^{2p} }{ \vert r_1-r_2 \vert^{1+2\alpha p}\vert \theta_1-\theta_2 \vert^{1+2\alpha p}  }  \,\dd r_1\dd r_2\dd\theta_1\dd\theta_2. 
\end{align*}

Then we get an estimate of the form
\begin{align*}
&\mathbb{E}\Big[\big\Vert W \big \Vert^{2p}_{\mathcal{W}^{\alpha,\alpha;2p}_{r,\theta}(I_1\times I_2)}\Big]\\ 
&\hspace{2cm} =\int_{I_1^2\times I_2^2} \frac{\mathbb{E}\Big[ \big\vert W(r_1,\theta_1)-W(r_2,\theta_1)- W(r_1,\theta_2)+W(r_2,\theta_2)  \big \vert_{\mathfrak{g}}^{2p} \Big] }{ \vert r_1-r_2 \vert^{1+2\alpha p}\vert \theta_1-\theta_2 \vert^{1+2\alpha p}  }  \,\dd r_1\dd r_2\dd\theta_1\dd\theta_2 \\
&\hspace{2cm}\lesssim \int_{I_1^2\times I_2^2} \frac{\mathbb{E}\Big[ \big\vert W(r_1,\theta_1)-W(r_2,\theta_1)- W(r_1,\theta_2)+W(r_2,\theta_2)  \big \vert_{\mathfrak{g}}^{2} \Big]^p }{ \vert r_1-r_2 \vert^{1+2\alpha p}\vert \theta_1-\theta_2 \vert^{1+2\alpha p}  }  \,\dd r_1\dd r_2\dd\theta_1\dd\theta_2\\
&\hspace{2cm}\lesssim C_p\int_{I_1^2\times I_2^2} \frac{ \sigma\big(\square(r_1,r_2,\theta_1,\theta_2)\big) ^p }{ \vert r_1-r_2 \vert^{1+2\alpha p}\vert \theta_1-\theta_2 \vert^{1+2\alpha p}  }  \,\dd  r_1\dd r_2\dd\theta_1\dd\theta_2,
\end{align*}
where we used respectively Fubini's theorem, then hyper-contractivity of $W$ (since $(r,\theta)\mapsto W(r,\theta)$ is a Brownian sheet, check \cite[Proposition 2.4]{DN}), and where  we have used the notation $\square(r_1,r_2,\theta_1,\theta_2) $ for the flow box bordered by $(r_1,\theta_1),(r_2,\theta_1),(r_1,\theta_2),(r_2,\theta_2)$. Since far from the minimum (and other critical points)  the area function behaves like 
\[\square(r_1,r_2,\theta_1,\theta_2)\sim\vert r_2-r_2\vert \vert \theta_2-\theta_2\vert,\]
we see immediately that one needs to impose the condition $p-1-2\alpha p>-1$ which is equivalent to $\alpha<\frac 12$. Beware that the above estimate on the area 
is not uniform when we get closer to $\min(f)$, which is why
we need to scale the previous estimate taking into account that the area functional has the same growth of $r$ near $0$. We have 
\begin{align*}
    \mathbb{E}\Big[\big\Vert  \mathcal{S}^{\lambda \star} W \big\Vert^{2p}_{\mathcal{W}^{\alpha,\alpha;2p}_{r,\theta}(I_1\times I_2)}  \Big] &\lesssim C_p\int_{I_1^2\times I_2^2} \frac{ \sigma\big(\square(\lambda r_1, \lambda r_2, \lambda \theta_1, \lambda \theta_2)\big)^p }{ \vert r_1-r_2 \vert^{1+2\alpha p}\vert \theta_1-\theta_2 \vert^{1+2\alpha p}  }  \,\dd  r_1\dd r_2\dd\theta_1\dd\theta_2\\
    &= C_p\int_{I_1^2\times I_2^2} \frac{ \left(\int_{\lambda r_1}^{\lambda r_2} \int_{\lambda \theta_1}^{\lambda \theta_2} \sigma(r,\theta)\,\dd  r\dd\theta \right) ^p }{ \vert r_1-r_2 \vert^{1+2\alpha p}\vert \theta_1-\theta_2 \vert^{1+2\alpha p}  }  \,\dd  r_1\dd r_2\dd\theta_1\dd\theta_2\\
&= C_p\int_{I_1^2\times I_2^2} \frac{ \left(\int_{ r_1}^{ r_2} \int_{\theta_1}^{ \theta_2}   \lambda^3    r\rho(\lambda r,\lambda\theta)  \,\dd  r\dd\theta  \right) ^p }{ \vert r_1-r_2 \vert^{1+2\alpha p}\vert \theta_1-\theta_2 \vert^{1+2\alpha p}  }  \,\dd  r_1\dd r_2\dd\theta_1\dd\theta_2\\
&\lesssim C_p\lambda^{3p}\int_{I_1^2\times I_2^2} \frac{ \left(\int_{\square(r_1,r_2,\theta_1,\theta_2)} \,\dd  rd\theta \right) ^p }{ \vert r_1-r_2 \vert^{1+2\alpha p}\vert \theta_1-\theta_2 \vert^{1+2\alpha p}  }  \,\dd r_1\dd  r_2\dd\theta_1\dd\theta_2.
\end{align*}
Note that this is finite for $\alpha <\frac 12$,
which gives the desired result.
\end{proof}

\begin{proposition}\label{normep}
     We have, for  $\alpha<\frac 12$, $p\geq 1$, $\ell $ such that $\ell-\frac{1}{p}<0$, and all $0<t<1$ 
\[ \mathbb{E}\Big[\big\Vert \pi_{\mathsf{noise}}\cdot A^\epsilon_{\Sigma,t\sigma}\big\Vert_{\mathcal{W}^{2p;\alpha,\alpha-1,\ell} (D_{\min f}) }^{2p}\Big] \leqslant C t^{p}, \]
    where again the constant $C>0$ does not depend on $\epsilon,t\in (0,1)^2$.
\end{proposition}
\begin{proof}
Using Proposition \ref{ScaleFreeMin} 
\begin{align*}
    \mathbb{E} \Big[  \sum_{n=0}^\infty 2^{n2p (\ell-\frac{2}{2p} ) } \big \Vert \mathcal{S}^{2^{-n}\star} W\big \Vert^{2p}_{\mathcal{W}^{\alpha,\alpha,2p}} \Big] &=\sum_{n=0}^\infty 2^{n2p (\ell-\frac{2}{2p} ) }  \mathbb{E}\Big[\big \Vert \mathcal{S}^{2^{-n}\star} W\big \Vert^{2p}_{\mathcal{W}^{\alpha,\alpha,2p}}\Big]\\ 
    & \lesssim t^{p}\sum_{n=0}^\infty 2^{2np (\ell-\frac{2}{2p}-\frac{3}{2} ) }  
\end{align*}
which is finite provided $\ell-\frac{1}{p}<\frac{3}{2}$.
 We still need to mention one more detail. Since we need the norm in term of $A$ rather than in term of $W$, there is a slight change in the weights.
The commutation relation
$\partial_\theta\mathcal{S}^{2^{-n}\star}=2^{-n}\mathcal{S}^{2^{-n}\star}\partial_\theta$, this implies that 
\[\mathcal{S}^{2^{-n}\star}A_{\mathcal{N}}=\mathcal{S}^{2^{-n}\star}\left(\partial_\theta W\right)=2^n\partial_\theta\mathcal{S}^{2^{-n}\star} W \]
and hence by definition,
\[\Vert \mathcal{S}^{2^{-n}\star} A_{\mathcal{N}} \Vert^{2p}_{\mathcal{W}^{\alpha,\alpha-1,2p}}=2^{2pn}\Vert \partial_\theta\mathcal{S}^{2^{-n}\star} W \Vert^{2p}_{\mathcal{W}^{\alpha,\alpha-1,2p}}=2^{2pn}\Vert\mathcal{S}^{2^{-n}\star} W \Vert^{2p}_{\mathcal{W}^{\alpha,\alpha,2p}}. \] 
From the equality
\begin{align*}
\Vert \mathcal{S}^{2^{-n}\star} A_{\mathcal{N}} \Vert^{2p}_{\mathcal{W}^{\alpha,\alpha-1,2p}}=2^{2pn}\Vert\mathcal{S}^{2^{-n}\star} W \Vert^{2p}_{\mathcal{W}^{\alpha,\alpha,2p}},    
\end{align*}
we deduce that
\begin{align*}
 \mathbb{E} \Big[ \sum_{n=0}^\infty 2^{n2p (\ell-\frac{2}{2p} ) } \big \Vert \mathcal{S}^{2^{-n}\star} A_{\mathcal{N},S,t\sigma} \big\Vert^{2p}_{\mathcal{W}^{\alpha,\alpha-1,2p}} \Big] &=\sum_{n=0}^\infty 2^{n2p (\ell-\frac{2}{2p} ) } 2^{2pn} \underset{2^{-3pn}}{\underbrace{  \mathbb{E}\Big[\big \Vert\mathcal{S}^{2^{-n}\star} W \big \Vert^{2p}_{\mathcal{W}^{\alpha,\alpha,2p}}\Big]}}\\
  &\lesssim C_p \sum_{n=0}^\infty 2^{n2p (\ell-\frac{2}{2p} ) } 2^{2pn} 2^{-3np}
\end{align*}
therefore resumming the previous bound yields
\[
\mathbb{E} \Big[ \sum_{n=0}^\infty 2^{n2p (\ell-\frac{2}{2p} ) } \big\Vert \mathcal{S}^{2^{-n}\star} A_{\mathcal{N},S,t\sigma} \big\Vert^{2p}_{\mathcal{W}^{\alpha,\alpha-1,2p}} \Big]  <+\infty   
\]
provided that $\ell-\frac{2}{2p}<\frac{1}{2}$.
\end{proof}
Note that we have also showed that near the minimum 
$A_{\mathcal{N}}$ belongs
to the space 
$\mathcal{W}^{2p;\alpha,\alpha-1;\ell}$
for all $\ell<\frac{1}{2}$ and $p$ large enough.

\paragraph{Near $\max(f)$.}
Near the maximum, the difficulty is twofold: we must combine a scaling argument with a conditioning argument. After conditioning, we can only access observables that depend on the region strictly below the maximum. The scaling norms are essential in this setting, since they probe regularity near the maximum by testing on shrinking coronas that approach it arbitrarily closely without ever touching it. We repeat the operation of the previous section for the weighted Bony norm as observable.
At each dyadic scale $\lambda=2^{-n}$, the key idea is that each dyadic block 
$\Vert  \mathcal{S}^{2^{-n}\star} A \Vert^{2p}_{\mathcal{W}^{\alpha,\alpha;2p}_{r,\theta}([\frac{1}{2},1]\times [a,b]_\theta)} $
of the weighted norm can be viewed as an observable which is measurable with respect to what happens in $f\leq \max f-2^{-n}$. Therefore by Definition \ref{ClosedYMM}, and for every $\epsilon>0$,
the expectation of the weighted norm w.r.t. to the closed surface Yang--Mills pre-measure can be decomposed as
\begin{align} \label{WeightSum}
&\mathbb{E}\Big[\big\| \pi_{\mathsf{noise}}\cdot A^\epsilon_{\Sigma,\sigma}\big\|^{2p}_{\mathcal{B}_{\mathsf{YM}}^{\alpha,2p,s,\ell}}\Big] \\
&\nonumber\hspace{2cm}\leq  \sum_{n=0}^\infty 2^{n2p(\ell-\frac{2}{2p})} \frac{ \mathbb{E}\Bigg[ \big \Vert \pi_{\mathsf{noise}}\left( \mathcal{S}^{2^{-n}\star}A^\mathsf{Free}_{\Sigma,\sigma} \right) \Vert^{2p}_{\mathcal{W}^{\alpha,\alpha-1;2p}}\cdot p_{\sigma\big(\Sigma_{\geq \max f -2^{-n}}\big)}\Big(\mathsf{Hol}_{\max f- 2^{-n} }\big(A^{\mathsf{Free}}_{\Sigma,\sigma}\big ) \Big)\Bigg] }{ \mathbb{E}\Bigg[p_{\sigma\big(\Sigma_{\geq \max f -2^{-n}}\big)}\Big(\mathsf{Hol}_{\max f- 2^{-n} }\big(A^{\mathsf{Free}}_{\Sigma,\sigma}\big ) \Big)\Bigg]}.
\end{align}
The denominator is equal to the partition function of the closed Yang--Mills measure, and does not depend on $n\geqslant 0$. We will denote it by $Z_{\Sigma,t\sigma}(1)$, where $1$ stands for the boundary condition which is $\mathsf{Hol}_{\partial \mathcal{S}}=1$ in the case of the closed measure.
\par This leads us into dividing the analysis into two cases. The first case is the case of surfaces of genera greater or equal to $2$, in which $Z_{\Sigma,t\sigma}(1)$ is uniformly bounded in $t$. In this case, we have 
\begin{align} \label{Boundggeq2}
&\mathbb{E}\Big[\big\| \pi_{\mathsf{noise}}\cdot A^\epsilon_{\Sigma,\sigma}\big\|^{2p}_{\mathcal{B}_{\mathsf{YM}}^{\alpha,2p,s,\ell}}\Big] \\
&\nonumber\hspace{2cm}\lesssim \sum_{n=0}^\infty 2^{n2p(\ell-\frac{2}{2p})}  \mathbb{E}\Bigg[ \big \Vert \pi_{\mathsf{noise}}\left( \mathcal{S}^{2^{-n}\star}A^\mathsf{Free}_{\Sigma,\sigma} \right) \Vert^{2p}_{\mathcal{W}^{\alpha,\alpha-1;2p}}\cdot p_{\sigma\big(\Sigma_{\geq \max f -2^{-n}}\big)}\Big(\mathsf{Hol}_{\max f- 2^{-n} }\big(A^{\mathsf{Free}}_{\Sigma,\sigma}\big ) \Big)\Bigg] .
\end{align}
\par The second case concerns the case of genus $1$. In this case $Z_{\Sigma,t\sigma}(1)\xrightarrow[t\to\infty]{}\infty$. Even though this does not yield the optimal choice for the parameter $p$, we can still bound as in Equation \ref{Boundggeq2}.

Beware that in inequality (\ref{WeightSum}), the expectation is taken w.r.t. the conditionned ($\varepsilon$--regularized) Yang--Mills measure of the closed surface on the left hand side and the expectation on the r.h.s. is taken w.r.t. the free boundary Yang--Mills measure.

\begin{proposition}\label{qbigenough}
   Let $0<t<1$, $\alpha<\frac 1 2$, $p\geq 1$, $\ell<0$, and $q$ such that $2p\ell-2+\frac{\dim G}{q'}-p<0$, where $\frac{1}{q}+\frac{1}{q'}=1$.  We have the following.
    \begin{itemize}
        \item For the case $g\geq 2$: 
            \[\sup_{\epsilon>0} \mathbb{E}\Big[\big\| \pi_{\mathsf{noise}}\cdot A^\epsilon_{\Sigma,t\sigma}\big\|^{2p}_{\mathcal{B}_{\mathsf{YM}}^{\alpha,2p,s,\ell}}\Big]\lesssim t^{p-\frac{\dim G}{2q'}}.\]
        \item For the case $g=1$:
                   \[\sup_{\epsilon>0} \mathbb{E}\Big[\big\| \pi_{\mathsf{noise}}\cdot A^\epsilon_{\Sigma,t\sigma}\big\|^{2p}_{\mathcal{B}_{\mathsf{YM}}^{\alpha,2p,s,\ell}}\Big]\lesssim t^{p-\frac{\dim G}{2q'}-\frac{\dim G}{2}}.\]
    \end{itemize}

\end{proposition}

\begin{proof}
    We have 
    {\small
        \begin{align*}
        2^{-2np}&\mathbb{E}\Bigg[ \big \Vert \mathcal{S}^{2^{-n}\star}\left(\pi_{\mathsf{noise}} A^\mathsf{Free}_{\Sigma,t\sigma} \right)  \Vert^{2p}_{\mathcal{W}^{\alpha,\alpha-1;2p}}\cdot p_{t\sigma\big(\Sigma_{\geq \max f -2^{-n}}\big)}\Big(\mathsf{Hol}_{\max f- 2^{-n} }\big(A^{\mathsf{Free}}_{\Sigma,t\sigma}\big ) \Big)\Bigg] \\
        =&\int_{I_1^2\times I_2^2}  \frac{\,\dd r_1\dd r_2\dd\theta_1\d d\theta_2 }{\vert r_1-r_2\vert^{1+2\alpha p} \vert \theta_1-\theta_2 \vert^{1+2\alpha p}}\\
        &\hspace{2cm}\cdot\mathbb{E}\Big[\vert W_t(2^{-n}r_1,2^{-n}\theta_1)-W_t(2^{-n}r_2,2^{-n}\theta_1)- W_t(2^{-n}r_1,2^{-n}\theta_2)+W_t(2^{-n}r_2,2^{-n}\theta_2)    \vert^{2p}  \\
        &\hspace{6cm}\cdot p_{t\sigma\big(\Sigma_{\geq \max f -2^{-n}}\big)}\Big(\mathsf{Hol}_{\max f- 2^{-n} }\big(A^{\mathsf{Free}}_{\Sigma,t\sigma}\big ) \Big)\Big],
    \end{align*}}
   where $W_t$ is such that
\[A^\mathsf{Free}_{\Sigma,t\sigma}(x,y)= \left(\partial_y W_t\right) (x,y)  \text{ and } W_t(.,0)=0.\] 
    Now using H\"older inequality with conjugate exponents $(q,q')\in (1,+\infty)^2$ such that $\frac 1 q +\frac{1}{q'}=1,$ we get  
    {\small
    \begin{align*}
        2^{-2np}&\mathbb{E}\Bigg[ \big \Vert \mathcal{S}^{2^{-n}\star}\left(\pi_{\mathsf{noise}} A^\mathsf{Free}_{\Sigma,t\sigma} \right) \Vert^{2p}_{\mathcal{W}^{\alpha,\alpha-1;2p}}\cdot p_{t\sigma\big(\Sigma_{\geq \max f -2^{-n}}\big)}\Big(\mathsf{Hol}_{\max f- 2^{-n} }\big(A^{\mathsf{Free}}_{\Sigma,t\sigma}\big ) \Big)\Bigg] \\
        \leq&\int_{I_1^2\times I_2^2}  \frac{\,\dd r_1\dd r_2\dd\theta_1\dd\theta_2 }{\vert r_1-r_2\vert^{1+2\alpha p} \vert \theta_1-\theta_2 \vert^{1+2\alpha p}}\\
        &\hspace{2cm}\cdot\mathbb{E}\Big[\vert W_t(2^{-n}r_1,2^{-n}\theta_1)-W_t(2^{-n}r_2,2^{-n}\theta_1)- W_t(2^{-n}r_1,2^{-n}\theta_2)+W_t(2^{-n}r_2,2^{-n}\theta_2)    \vert^{2pq'}\Big]^{\frac{ 1}{q'}}  \\
        &\hspace{6cm}\cdot \mathbb{E}\Big[p_{t\sigma\big(\Sigma_{\geq \max f -2^{-n}}\big)}\Big(\mathsf{Hol}_{\max f- 2^{-n} }\big(A^{\mathsf{Free}}_{\Sigma,t\sigma}\big ) \Big)^{q}\Big]^\frac{1}{q}.
    \end{align*}}
    
    In law, $W_t=\sqrt{t}W$ so that, combining with hypercontractivity, and using a procedure similar to the one in Proposition \ref{normep}, we get 
    {\small
    \begin{align*}
&\int_{I_1^2\times I_2^2}  \frac{\,\dd r_1\dd r_2\dd\theta_1\dd\theta_2 }{\vert r_1-r_2\vert^{1+2\alpha p} \vert \theta_1-\theta_2 \vert^{1+2\alpha p}}\\
        &\hspace{2cm}\cdot\mathbb{E}\Big[\vert W_t(2^{-n}r_1,2^{-n}\theta_1)-W_t(2^{-n}r_2,2^{-n}\theta_1)- W_t(2^{-n}r_1,2^{-n}\theta_2)+W_t(2^{-n}r_2,2^{-n}\theta_2)    \vert^{2pq'}\Big]^{\frac{ 1}{q'}}  \\
\lesssim&\ t^{p}\int_{I_1^2\times I_2^2} \frac{\dd r_1\dd r_2\dd\theta_1\dd\theta_2 }{\vert r_1-r_2\vert^{1+2\alpha p} \vert \theta_1-\theta_2 \vert^{1+2\alpha p}}\\
        &\hspace{2cm}\cdot\mathbb{E}\Big[\vert W(2^{-n}r_1,2^{-n}\theta_1)-W(2^{-n}r_2,2^{-n}\theta_1)- W(2^{-n}r_1,2^{-n}\theta_2)+W(2^{-n}r_2,2^{-n}\theta_2)    \vert^{2}\Big]^{p}\\
        \lesssim & t^{p}2^{-3np}
    \end{align*}}
since the proof proceed as in Proposition \ref{normep}.  
For the moment, our estimate writes
\begin{align*}
       &\mathbb{E}\Bigg[ \big \Vert \mathcal{S}^{2^{-n}\star}\pi_{\mathsf{noise}}\cdot A^\mathsf{Free}_{\Sigma,t\sigma} \Vert^{2p}_{\mathcal{W}^{\alpha,\alpha-1;2p}}\cdot p_{t\sigma\big(\Sigma_{\geq \max f -2^{-n}}\big)}\Big(\mathsf{Hol}_{\max f- 2^{-n} }\big(A^{\mathsf{Free}}_{\Sigma,t\sigma}\big ) \Big)\Bigg] \\
&\hspace{2cm}\lesssim t^{p}2^{-np}\underbrace{ \mathbb{E}\Big[p_{t\sigma\big(\Sigma_{\geq \max f -2^{-n}}\big)}\Big(\mathsf{Hol}_{\max f- 2^{-n} }\big(A^{\mathsf{Free}}_{\Sigma,t\sigma}\big ) \Big)^{q}\Big]^\frac{1}{q}.}
    \end{align*}    
    
    We need now an intermediate lemma to control the term underbraced.
\begin{lemma}\label{LemmaNoyau}
We have 
\begin{align*}
    &\mathbb{E}\Big[ p_{t\sigma\big(\Sigma_{\geq \max f -2^{-n}}\big)}\big (\mathsf{Hol}_{{\max f-2^{-n}} }(A_{\Sigma,t\sigma}^{\mathsf{Free}}) \big)^{q}\Big]^{\frac{1}{q}} \lesssim \Big (t2^{-2n}\Big)^{-\frac{\dim G}{2}\left(1-\frac 1q\right)}\left\|Z_{\Sigma_{\leq\max f- 2^{-n}}}\right\|_{\infty}.
\end{align*}
\end{lemma}
\begin{proof}
We have
    \begin{align*}
    \mathbb{E}\Big[ p_{t\sigma\big(\Sigma_{\geq \max f -2^{-n}}\big)}\big (\mathsf{Hol}_{{\max f-2^{-n}} }(A_{\Sigma,t\sigma}^{\mathsf{Free}}) \big)^{q}\Big]^{\frac{1}{q}}  &=\left(\int_{G}p_{t\sigma\big(\Sigma_{\geq \max f -2^{-n}}\big)}(x)^qZ_{\Sigma_{\leq\max f- 2^{-n}},t\sigma}(x)\,\dd x\right)^{\frac{1}{q}} \\
    &\lesssim \left\|p_{t\sigma\big(\Sigma_{\geq \max f -2^{-n}}\big)}\right\|_{L^{q}} \left\|Z_{\Sigma_{\leq\max f- 2^{-n}},t\sigma}\right\|_{\infty}
\end{align*}
From the $L^q$ bounds of heat kernel, we get 
\[\mathbb{E}\Big[ p_{t\sigma\big(\Sigma_{\geq \max f -2^{-n}}\big)}\big (\mathsf{Hol}_{{\max f-2^{-n}} }(A) \big)^{q}\Big]^{\frac{1}{q}} \lesssim \Big (t\sigma\big(\Sigma_{\geq \max f -2^{-n}}\big)\Big)^{-\frac{\dim G}{2}\left(1-\frac 1q\right)}\left\|Z_{\Sigma_{\leq\max f- 2^{-n}},t\sigma}\right\|_{\infty}. \]
Since $\Sigma_{\geq \max f -2^{-n}}$ is diffeomorphic to a disc of radius $2^{-n}$, we get 
\[\mathbb{E}\Big[ p_{t\sigma\big(\Sigma_{\geq \max f -2^{-n}}\big)}\big (\mathsf{Hol}_{{\max f-2^{-n}} }(A) \big)^{q}\Big]^{\frac{1}{q}} \lesssim \Big (t2^{-2n}\Big)^{-\frac{\dim G}{2}\left(1-\frac 1q\right)}\left\|Z_{\Sigma_{\leq\max f- 2^{-n}},t\sigma}\right\|_{\infty}. \]
\end{proof}

Now we need to separate into two cases. The case where the genus $g\geq 2$ and the case where the genus $g=1$.
\paragraph{The case $g\geq 2$.} From Lemma \ref{LemmaNoyau},  since for $g\geq 2$, $\big\|Z_{\Sigma_{\leq\max f- 2^{-n}}}\big\|_{\infty}$ is uniformly bounded in $n,t$ and $x$, we have 
   \begin{align*}
        &\mathbb{E}\Bigg[ \big \Vert \mathcal{S}^{2^{-n}\star}\pi_{\mathsf{noise}}\cdot A^\mathsf{Free}_{\Sigma,t\sigma} \Vert^{2p}_{\mathcal{W}^{\alpha,\alpha-1;2p}}\cdot p_{t\sigma\big(\Sigma_{\geq \max f -2^{-n}}\big)}\Big(\mathsf{Hol}_{\max f- 2^{-n} }\big(A^{\mathsf{Free}}_{\Sigma,t\sigma}\big ) \Big)\Bigg] 
       \\ 
      & \lesssim  t^{-\frac{\dim(G)}{2}(1-\frac{1}{q})} 2^{n\dim(G)(1-\frac{1}{q})}t^{p}2^{-np} \\
        &\lesssim  t^{p-\frac{\dim G}{2q'}}2^{n\left(\frac{\dim G }{q'} -p\right) },
    \end{align*}
where $q'$ is such that $\frac{1}{q}+\frac{1}{q'}=1$. Therefore,
\[\mathbb{E}\Big[\big\| \pi_{\mathsf{noise}}\cdot A^\epsilon_{\Sigma,\sigma}\big\|^{2p}_{\mathcal{B}_{\mathsf{YM}}^{\alpha,2p,s,\ell}}\Big] \lesssim  t^{p-\frac{\dim G}{2q'}} \sum_{n=0}^\infty 2^{n2p(\ell-\frac{2}{2p})} 2^{n\left(\frac{\dim G }{q'} -p\right) },\]
which is finite if $2p\ell-2+\frac{\dim G}{q'}-p<0$.
\paragraph{The case $g=1$.}\label{CaseGenus1}
    For the case where $g=1$, the partition function verifies $\left\|Z_{\Sigma,t\sigma}\right\|\leq C_2t^{-\frac{\dim G}{2}}$.
    In this case, again for $2p\ell-2+\frac{\dim G}{q'}-p<0$, we have 
    \[\mathbb{E}\Big[\big\| \pi_{\mathsf{noise}}\cdot A^\epsilon_{\Sigma,\sigma}\big\|^{2p}_{\mathcal{B}_{\mathsf{YM}}^{\alpha,2p,s,\ell}}\Big] \lesssim t^{p-\frac{\dim G}{2q'}-\frac{\dim G}{2}}. \]

\end{proof}

\section{Semiclassical limit}\label{semiclas}
In this section we conclude the proof of the semiclassical limit. The proof is divided into two steps. The first one is the tightness of the family of random $1$-forms $\big(A_{\Sigma,t\sigma}\big)_{t>0}$, and the second step consists in identifying the limit. The conclusion follows then from Prokhorov's theorem.
\subsection{Tightness of $\big(A_{\Sigma,t\sigma}\big)_{t>0}$}\label{s7}
Recall that we are considering a compact surface $\Sigma$ endowed with an area measure $\sigma$. We denote by $A_{\Sigma,\sigma}$ the $\mathcal{B}_{\mathsf{YM}}^{\alpha,p,s,\ell}$-valued random variable describing the Yang--Mills measure associated with the area form $\sigma$. For $t>0$, we consider the random variable $A_{\Sigma,t\sigma}$ that is the Yang--Mills measure with area form $t\sigma(\Sigma)$. In this section, we would like to show that the family  $(A_{\Sigma,t\sigma}^\epsilon\big)_{\epsilon>0}$  is tight in the space $\mathcal{B}_{\mathsf{YM}}^{\alpha,p,s,\ell}$ for suitable values of the parameters $\alpha,p,s,\ell$ ($\alpha<\frac 1 2$, $p> 1$, $\ell<0$). More precisely, we have the following theorem.
\begin{thm}\label{tightness}
     For $\alpha<\frac 1 2$, $p$ big enough, $\ell<0$, the family $\big(A_{\Sigma,t\sigma}\big)_{t>0}$ is tight in $\mathcal{B}^{\alpha,p,s,\ell}_{\mathsf{YM}}$.
\end{thm}
\begin{proof}
\par Using the compact injections of the spaces $\mathcal{B}^{\alpha,p,s,\ell}_{\mathsf{YM}}$, it is enough to show that 
\[\sup_{t>0} \mathbb{E}\Big[\big\|A_{\Sigma,t\sigma} \big\|^p_{\mathcal{B}^{\alpha,p,s,\ell}_{\mathsf{YM}}}\Big] <\infty.\]
The fundamental decomposition lemma lets us write 
$A_{\Sigma,t\sigma}=\pi_{\mathsf{noise}}\cdot A_{\Sigma,t\sigma} +\pi_{\mathsf{flat}}\cdot A_{\Sigma,t\sigma},$
 and from Lemma \ref{structureLemma}, we know that 
\[ \sup_{t>0} \mathbb{E}\Big[\big \|\pi_{\mathsf{flat}}\cdot A_{\Sigma,t\sigma}\big\|^p_{\mathcal{B}^{\alpha,p,s,\ell}_\mathsf{YM}}\Big] \leq (\dim G)^p.\]
From Section \ref{BoundNoise}
 \[\mathbb{E}\Big[\big\|\pi_{\mathsf{noise}}\cdot A
 _{\Sigma,t\sigma}\big\|^p_{\mathcal{B}_{\mathsf{YM}}^{\alpha,p,s,\ell}}\Big]\lesssim \begin{cases}
     t^{p-\frac{\dim G}{2q'}} \text{ if } g\geq 2\\
    t^{p-\frac{\dim G}{2q'}-\frac{\dim G}{2}} \text{ if } g= 1.\\
 \end{cases}\] 
    Choose therefore the index $q'$ large enough so that $p>\frac{\dim G}{2q'}$ if $g\geq 2$, or $p>\frac{\dim G}{2q'}+\frac{\dim G}{2}$ if $g=1$. Then we have 
    \[\mathbb{E}\Big[\big\|A_{\Sigma,t\sigma} \big\|^p_{\mathcal{B}^{\alpha,p,s,\ell}_{\mathsf{YM}}}\Big]  \lesssim  \mathbb{E}\Big[\big \|\pi_{\mathsf{noise}}\cdot A_{\Sigma,t\sigma}\big\|^p_{\mathcal{B}^{\alpha,p,s,\ell}_\mathsf{YM}}\Big]+ \mathbb{E}\Big[\big \|\pi_{\mathsf{flat}}\cdot A_{\Sigma,t\sigma}\big\|^p_{\mathcal{B}^{\alpha,p,s,\ell}_\mathsf{YM}}\Big],\]
    which gives $\sup_{t>0}\mathbb{E}\Big[\big\|A_{\Sigma,t\sigma} \big\|^p_{\mathcal{B}^{\alpha,p,s,\ell}_{\mathsf{YM}}}\Big]  <\infty,$ and concludes the proof.
\end{proof}

\subsection{Identification of the limit}\label{s8}
In the previous section, we have showed that the family $(A_{\Sigma,t\sigma})_{t>0}$ is a tight family of $\mathcal{B}^{\alpha,p,s,\ell}_\mathsf{YM}$-valued random variables for $\alpha<\frac 1 2$, $p\geq 1$, $\ell<0$. This means that the family of random variables $(A_{\Sigma,t\sigma})_{t>0}$ is compact in the weak convergence topology, by Prokhorov's theorem. To conclude the proof of Theorem \ref{Thm1} about the convergence of $A_{\Sigma,t\sigma}$ as $t\to 0$, it remains to show that the family $(A_{\Sigma,t\sigma})_{t>0}$ has only one accumulation point. Equivalently, we need to show the following proposition.
\begin{proposition}
    For $\alpha<\frac 1 2$, $p$ big enough\footnote{It is different in the cases $g=1$ and $g\geq 2$.}, $\ell<0$, there  exists a $\mathcal{B}^{\alpha,p,s,\ell}_\mathsf{YM}$-valued random variable $A_{\Sigma,0}$ such that,  for any sequence of decreasing positive real numbers $(s_n)_{n>0}$ converging to $0$ such that $(A_{\Sigma,s_n\sigma})_{n\geq0}$ converges, we have 
   \[A_{\Sigma,s_n\sigma}\xrightarrow[n\to\infty]{ }A_{\Sigma,0}.\]
   Moreover, almost surely we have 
   \[A_{\Sigma,0}\in \mathsf{vect}\{U_a: a\in \mathsf{Crit}_1(f)\},\]
   and it is the push--forward of the normalized Atiyah--Bott--Goldman measure by the map 
   \[ \left\{(a_1,a_2,\dots,a_{2g-1},a_{2g})\in G^{2g} \ : \mathsf{Hol}_{\partial \mathcal{S}}\Bigg[\sum_{i=1}^{2g}\log(a_i)U_{a_i}\Bigg]=1\right\} \rightarrow  \mathsf{vect}\{U_a: a\in \mathsf{Crit}_1(f)\}\]
   defined as 
   \[(a_1,a_2,\dots,a_{2g-1},a_{2g}) \in G^{2g} \mapsto\sum_{i=1}^{2g} \log(a_i)U_{a_i}.\]
\end{proposition}

\begin{proof}
Let  $(s_n)_{n>0}$ be a sequence  of decreasing positive real numbers that converges to $0$ and is such that $(A_{\Sigma,s_n\sigma})_{n\geq0}$ converges to some limit random variable $A$. From Paragraph \ref{BoundNoise},  for $\alpha<\frac 1 2$, $p\geq 1$, $\ell<0$, we have 
\[\mathbb{E}\Big[\big \|\pi_{\mathsf{noise}}\cdot A_{\Sigma,\sigma t}\big \|_{\mathcal{B}^{\alpha,p,s,\ell}_\mathsf{YM}}^p\Big] \xrightarrow[t\to 0]{}0.\]
This gives the following convergence in law $\pi_{\mathsf{noise}}\cdot A_{\Sigma,s_n\sigma} \xrightarrow[n\to \infty]{} 0$,
which means that almost surely, $\pi_{\mathsf{noise}}\cdot A=0$. We immediately get that almost surely, $A\in \mathsf{vect}\{U_a: a\in \mathsf{Crit}_1(f)\}$. This lets us write 
\[A=\sum_{i=1}^{2g} M_iU_{a_i},\]
and the law of $A$ is completely determined by the joint law $(M_1,\cdots, M_{2g})$. We have, for any bounded continuous function $F:\g^{2g}\rightarrow \R$, and since $A_{\Sigma,s_n\sigma}\xrightarrow[n\to\infty]{} A$,
\[ \mathbb{E}\big[F(M_1,\cdots,M_{2g})\big]=\lim_{n\to\infty} \mathbb{E}\Big[F\Big(\mathsf{Ext}_{a_1}(A_{\Sigma,s_n\sigma}),\dots,\mathsf{Ext}_{a_{2g}}(A_{\Sigma,s_n\sigma})\Big)\Big].\]
However, for $r>a_{2g}$
{\small
\begin{align*}
  &\mathbb{E}\Big[F\Big(\mathsf{Ext}_{a_1}(A_{\Sigma,s_n\sigma}),\cdots,\mathsf{Ext}_{a_{2g}}(A_{\Sigma,s_n\sigma})\Big)\Big]\\
  &\hspace{2cm}= \frac{ \mathbb{E}\Big[F\Big(\mathsf{Ext}_{a_1}(A_{S,s_n\sigma}^\mathsf{Free}),\cdots,\mathsf{Ext}_{a_{2g}}(A_{S,s_n\sigma}^\mathsf{Free})\Big)\cdot p_{s_n\sigma\big(\Sigma_{\geq r}\big)}\Big(\mathsf{Hol}_{r}(A^\mathsf{Free}_{S,s_n\sigma})\Big)\Big]}{Z_{\Sigma,s_n\sigma}}\\
&\hspace{2cm}=\big(Z_{\Sigma,s_n\sigma}\big)^{-1} \int_{G^{2g}}F\big(\log(a_1),\dots,\log(a_{2g})\big)\cdot p_{s_n\sigma(\Sigma)}\left (\mathsf{Hol}_{\partial \mathcal{S}}\Bigg[\sum_{i=1}^{2g}\log(a_i)U_{a_i}\Bigg]\right)\,\dd a,
\end{align*}}
which converges to the expectation of $F$ under our representative of the Atiyah--Bott--Goldman measure. In fact, the third step follows the fact that $\mathsf{Hol}_{r}(A^\mathsf{Free}_{S,s_n\sigma})$ is an alternate product of Haar distributed, independent random variables, and independent random variables distributed as the law of a Brownian motion taken at time $s_n\sigma(\square(\theta_i,\theta_{i+1},0,r))$, combined with the semi-group property of the heat kernel.
\par This concludes, using Slutsky's theorem, the proof of the proposition. 
\end{proof}

Combining with the representation of the Atiyah--Bott--Goldman volume in the Morse gauge from Section~\ref{sss:ABGvol}, especially equation~(\ref{eq:FromheattoABG}), we have now concluded the following theorem of this article.

\begin{thm}\label{Thm3}
 For $\alpha<\frac 1 2$, $p\geq 1$, $\ell<0$,  in the sense of weak convergence of measures in $\mathcal{B}_{\mathsf{YM}}^{\alpha,p,s,\ell}$, we have $A_{\Sigma, t\sigma}\xrightarrow[t\to 0]{}A_{\Sigma,0}$.
   Moreover, $A_{\Sigma,0}$ is supported on $\mathsf{vect}\{U_a: a\in \mathsf{Crit}_1(f)\}$  and it is the push--forward of the Atiyah--Bott--Goldman measure by the map 
   \[ \left\{ \mathsf{Hol}_{\partial \mathcal{S}}\Bigg[\sum_{ a\in \mathsf{Crit}_1(f) }\log(g_a)U_{a}\Bigg]=1\right\} \rightarrow  \mathsf{vect}\{U_a: a\in \mathsf{Crit}_1(f)\}\]
   defined as 
 $   (g_a)_{  a\in \mathsf{Crit}_1(f) }  \in G^{2g} \mapsto\sum_{ a\in \mathsf{Crit}_1(f)} \log (g_a)U_{a}$.
\end{thm}

\appendix

\section{A sketch of the construction of the Yang--Mills random connection }
\label{s:justifyformulaDAS}
The purpose of this appendix is to explain how to define the free boundary Yang--Mills measure in such a way that both the anisotropic norms
 $A^{\mathsf{Free}}_{S,\sigma} \mapsto \Vert A^{\mathsf{Free}}_{S,\sigma} \Vert_{\mathcal{B}_{\mathsf{YM}}^{\alpha,p,s,\ell}  } $ 
and the holonomy functional
 $A^{\mathsf{Free}}_{S,\sigma} \mapsto \mathsf{Hol}_{\max f-\epsilon }\big(A^{\mathsf{Free}}_{S,\sigma}\big) $ 
are measurable. In this paragraph, we explain one way of doing this using a Cauchy sequence argument.
\subsection{A version of the Yang--Mills measure in  $\mathcal{B}_{\mathsf{YM}}^{\alpha,p,s,\ell}$}
In this paragraph, we adapt to the spaces $\mathcal{B}_{\mathsf{YM}}^{\alpha,p,s,\ell}$ a construction  given in ~\cite{BCDRT} to the space $H^{-1^-}(\Sigma)$. This gives a direct construction of the Yang--Mills random field in  $\mathcal{B}_{\mathsf{YM}}^{\alpha,p,s,\ell}$. Let us first explore the free boundary case.
\par Choose a real orthonormal basis $(e_n)_{n\in \mathbb{N}}$ of $L^2(\Sigma,\sigma,\mathfrak{g})$ then define, for $N\in\N$, the $N$ resolution approximation of the white noise as
\[
  \xi_N:= \sum_{0\leqslant n\leqslant N}c_n e_n ,
\]
where $(c_n)_n$ are i.i.d standard Gaussian random variables defined on some probability space $(\Omega,\mathcal{F},\mathbb{P})$. Denote by $A_{\mathcal{N},N}$ the approximation $\iota_V\mathcal{L}_V^{-1}\left(\xi_N\sigma \right)$ of the normal component $A_{\mathcal{N}}$.
In \cite{BCDRT}, it is proved that for any oriented curve $\gamma$, everywhere transverse to the vector field $V$, the sequence $\left(\int_\gamma A_{\mathcal{N},N}\right)_{N\in \mathbb{N}} $ is a martingale bounded in $L^2(\Omega,\mathfrak{g})$. Therefore, 
\[\int_\gamma A_{\mathcal{N},N}\xrightarrow[N\to\infty]{L^2} \int_\gamma A_{\mathcal{N}}. \]
Sine the law of each $ \int_\gamma A_{\mathcal{N},N}$ is Gaussian the convergence holds true in $L^p(\Omega,\mathfrak{g})$
for all $p\in [2,+\infty)$.
Localize on any flow box $[-1,1]_x\times [-1,1]_y$ with coordinates $(x,y)$ where $V=\partial_x$. Denote by $W_{\mathcal{N},N}(x,y):=\int_0^y A_{\mathcal{N},N}(x,u)\dd u$
the integral of $A_{\mathcal{N},N}$ on the vertical interval $ \{x\}\times [0,y]$. For given $(x,y)$, $\big(W_{\mathcal{N},N}(x,y)\big)_{N\geq 0}$ is a sequence of $\mathfrak{g}$--valued martingale bounded in $L^2(\Omega)$. Therefore, it converges to some $\mathfrak{g}$--valued Gaussian random variable $W_{\mathcal{N}}(x,y)$ which satisfies 
\begin{align*}
\mathbb{E}\left[\big\langle W_{\mathcal{N}}(x_1,y_1)-W_{\mathcal{N}}(x_2,y_2), v_1\big\rangle_{\mathfrak{g}}\big\langle W_{\mathcal{N}}(x_1,y_1)-W_{\mathcal{N}}(x_2,y_2), v_2\big\rangle_{\mathfrak{g}} \right]   = \sigma\big(\square(x_1,y_1,x_2,y_2)\big)\left\langle v_1,v_2 \right\rangle_{\mathfrak{g}},
\end{align*}
where $\square(x_1,y_1,x_2,y_2)$ is the flow box whose vertices are $(x_1,y_1),(x_1,y_2),(x_2,y_1),(x_2,y_2)$. This means that $W$ has the law of a $\mathfrak{g}$--valued reparametrized Brownian sheet.
For $N_2\geqslant N_1$, we set 
\[W_{\mathcal{N},N_1,N_2}(x,y):=\int_0^y \iota_V\mathcal{L}_V^{-1}\left( \sum_{ N_1\leqslant n\leqslant N_2}c_ne_n \right)(x,u,\dd u),  \]
which is a well--defined Gaussian random variable. For $\alpha\in (\frac{1}{3},\frac{1}{2})$ and $p\geqslant 2$ some even integer, we have the following estimate.
\begin{align*}
&\mathbb{E}\Big[ \Vert A_{\mathcal{N},N_1}-A_{\mathcal{N},N_2} \Vert_{\mathcal{W}^{\alpha,\alpha-1,p}}^p \Big]\\
&\hspace{1cm}=   \int_{[-1,1]^4}  \dd x_1\dd y_1\dd x_2\dd y_2 
\\&\hspace{2cm}\cdot\frac{\mathbb{E}\Big[\big \vert W_{\mathcal{N},N_1,N_2}(x_1,y_1  )-W_{\mathcal{N},N_1,N_2}(x_2,y_1  )-W_{\mathcal{N},N_1,N_2}(x_1,y_2  )+W_{\mathcal{N},N_1,N_2}(x_2,y_2  ) \big\vert_{\mathfrak{g}}^p\Big] }{\vert x_1-x_2 \vert^{1+\alpha p} \vert y_1-y_2  \vert^{1+\alpha p}}   \\
&\hspace{1cm}\leqslant C_p\int_{[-1,1]^4}   \dd x_1\dd y_1\dd x_2\dd y_2 \cdot \frac{ \left(\sum_{ N_1\leqslant n\leqslant N_2}\vert\int_{\square(x_1,y_1,x_2,y_2)} e_n\sigma\vert_{\mathfrak{g}}^2 \right)^{\frac{p}{2}} }{\vert x_1-x_2 \vert^{1+\alpha p} \vert y_1-y_2  \vert^{1+\alpha p}} \xrightarrow[N_1,N_2\to \infty]{} 0,
\end{align*}
since for all $\epsilon>0$,
{\small
\begin{align*}
    &\sum_{ N_1\leqslant n\leqslant N_2}\left \vert\int_{\square(x_1,y_1,x_2,y_2)} e_n\sigma\right \vert_{\mathfrak{g}}^2\\
&\hspace{2cm}=
   \sum_{N_1\geq n\geq N_2} \vert \langle 1_{\square(x_1,y_1,x_2,y_2)},e_n\rangle_{L^2(\sigma)}\vert^2 \\
&\hspace{2cm}=\left(  \sum_{N_1\geq n\geq N_2} \vert \langle 1_{\square(x_1,y_1,x_2,y_2)},e_n\rangle_{L^2(\sigma)}\vert^2\right)^\varepsilon \left(  \sum_{N_1\geq n\geq N_2} \vert \langle 1_{\square(x_1,y_1,x_2,y_2)},e_n\rangle_{L^2(\sigma)}\vert^2\right)^{1-\varepsilon}\\
&\hspace{2cm}=\left(  \sum_{N_1\geq n\geq N_2} \vert \langle 1_{\square(x_1,y_1,x_2,y_2)},e_n\rangle_{L^2(\sigma)}\vert^2\right)^\varepsilon \left(  \Vert 1_\sigma \Vert_{L^2(\Sigma)}^2  \right)^{1-\varepsilon}\\
    &\hspace{2cm}\leq o(1) \left( \vert x_1-y_1 \vert\vert x_2-y_2 \vert \right)^{1-\varepsilon},
\end{align*}
}
where we have used Pythagora's Theorem, and the fact that far from cirtical points, \[\sigma\big(\square(x_1,y_1,x_2,y_2)\big)=\mathcal{O}(\vert x_1-y_1\vert \vert x_2-y_2\vert).\]
This shows that $A_{\mathcal{N},N}$ is a Cauchy sequence in \[L^p\Big(\Omega, \mathcal{W}^{\alpha-\epsilon,\alpha-1-\epsilon,p}\big([-1,1]^2\big)\Big)\] for all $\epsilon>0$.
\par Since the weighted norms $\mathcal{W}^{\alpha,\alpha-1,p,s,\ell}(\mathcal{S})$
are only constructed from countably local $\mathcal{W}^{\alpha,\alpha-1,p}(\square)$-norms, for some \emph{countable family of flow boxes}, we can easily adapt the above argument and show prove that for every $\alpha\in (\frac{1}{3},\frac{1}{2})$, for weights $s<-1$, $\ell<0$ and every $p\geqslant 2$ large enough:
\[
\mathbb{E}\Big[ \Vert A_{\mathcal{N},N_1}- A_{\mathcal{N},N_2} \Vert^p_{\mathcal{W}^{\alpha,\alpha-1,p;s,\ell}} \Big]  \xrightarrow[N_1,N_2\to \infty]{} 0. 
\]

\subsection{Control of the holonomy along Cauchy sequences}

The construction of the closed measure relies on conditioning with respect to \(\mathsf{Hol}_{\partial \mathcal{S}}(A_{\Sigma,\sigma}^{\mathsf{Free}})\). We explain here how to define the holonomy of \(A_{\Sigma,\sigma}^{\mathsf{Free}}\). It suffices to define the holonomy of its normal component, as the holonomy of \(A_{\Sigma,\sigma}^{\mathsf{Free}}\) is then obtained by multiplicativity.

We show that \(\mathsf{Hol}_{\partial \mathcal{S}}(A_{\mathcal{N},N})\) forms a Cauchy sequence of random variables, and that \(\mathsf{Hol}_{\partial \mathcal{S}}\) can be realized as a measurable function of the limiting random variable \(A_{\mathcal{N}}\). One way to establish this is via rough path or controlled path techniques.

Let \(\gamma : [0,1] \to \mathcal{S}\) denote the closed oriented curve corresponding to \(\partial \mathcal{S}\), or more generally to a level set between the last critical point and the maximum of $f$.

\paragraph{Regularity of area functional.}
We denote by \(V^\perp\) the smooth vector field obtained by rotating \(V\) by an angle of \(\frac{\pi}{2}\) in each fiber. Given a \(C^1\) curve \(\gamma\), we need to control the regularity of the map
\[
r \mapsto \sigma\big(\angle(\gamma)(s,r)\big),
\]
where \(\angle(\gamma)\) denotes the curve obtained by joining \(\min f\) to \(\gamma(0)\) and \(\min f\) to \(\gamma(1)\) along flow lines of \(f\).

This regularity is formulated in the following lemma, whose proof can be found in~\cite{BCDRT,DN}.
\begin{lemma}\label{lem:regularityarea}
Let \(\gamma: [0,1]\to \Sigma\) be an oriented \(C^1\) curve such that, for all \(t\in [0,1]\), the scalar product
\[
\left\langle \gamma'(t), V^\perp(\gamma(t)) \right\rangle_{T_{\gamma(t)}\Sigma}
\]
has a constant sign. Denote by \(t_0\) the first time at which \(\gamma\) meets an unstable curve \(W^{u}(a)\), with \(a\in \mathrm{Crit}(f)_1\), and define
\[
\tau(t) = \ell\big(\gamma_{[t, t_0]}\big).
\]
Then \(\tau\) is strictly monotone and \(C^1\) for all \(t\) sufficiently close to \(t_0\), with \(t \le t_0\). For \(\delta>0\) small enough, the map
\[
[0,\delta]\ni \tau \mapsto \sigma\big(\angle(\gamma_{[t(\tau), t_0]})\big)
\]
belongs to \(C^{1-}\), where \(t(\tau)\) denotes the inverse function of \(t\mapsto \tau(t)\).

More precisely, as \(t\to t_0\), one has the asymptotic estimate
\begin{equation}
\sigma\big(\angle(\gamma_{[t, t_0]})\big) = \mathcal{O}\big(\tau(t)\log \tau(t)\big).
\end{equation}
In particular, the derivative \(\partial_t \sigma\big(\angle(\gamma_{[t, t_0]})\big)\) exhibits a logarithmic singularity in \(\tau\), namely
\[
\partial_t \sigma\big(\angle(\gamma_{[t, t_0]})\big) = \mathcal{O}\big(\log \tau(t)\big),
\]
and satisfies \(\partial_t \sigma\big(\angle(\gamma_{[t, t_0]})\big) \in L^p([0,t_0])\) for all \(1 \le p < \infty\).
\end{lemma}
Any transverse $C^1$ curve or any level curve satisfies the assumptions of the above Lemma.
\par We will now construct the holonomy $\mathsf{Hol}_{\gamma}(A_{\mathcal{N}})$ on piecewise $C^1$ curves
by rough path theory.

\paragraph{Identification with rough path and formal calculations of the second order term.}

We recall the definition of rough paths in our setting, following~\cite[p.~13]{friz_hairer}, using adapted notation. Let \(\gamma \in C^\infty([0,1],\Sigma)\) be a smooth curve everywhere transverse to \(V\), contained in a level set of \(f\), and such that it intersects \(\bigcup_a W^u(a)\) at most once, and if so only at its endpoints. For \(0 \le s \le t \le 1\), set
\[
W_{N,s,t} := \int_s^t \gamma^* A_{\mathcal{N},N} \in \mathfrak{g}.
\]

The Cauchy sequence argument showing that \(A_{\mathcal{N},N}\) converges in the anisotropic spaces as \(N \to +\infty\) also implies that \(\lim_{N\to+\infty} W_{N,s,t}\) exists as a random variable for all such level curves \(\gamma\), and that this limit depends measurably on  \(A_{\mathcal{N}}\) in the anisotropic space.

It is essential at this stage that \(\gamma\) is a level curve, so that the above line integral depends measurably on \(A\) in the anisotropic topology, since the definition of the anisotropic norms itself involves line integrals of  \(A_{\mathcal{N}}\)along such curves. In a second step, we will define the second-order term measurably in terms of the line integrals \(\int_s^t \gamma^* A_{\mathcal{N}}\) along level curves.

To extend integration to more general curves, called \emph{approximable} in~\cite{BCDRT}, one must argue via Cauchy sequences of curves, as in~\cite{BCDRT}. However, for general approximable curves, the second-order term cannot be defined in this way.

\begin{definition}[A \(\mathfrak{g}\)-valued rough path]
A rough path on the interval \([0,1]\) with values in the Lie algebra \(\mathfrak{g}\) consists of a continuous function
\[
W : [0,1] \to \mathfrak{g},
\]
together with a continuous function
\[
\mathbb{W} : [0,1]^2 \to \mathfrak{g} \otimes \mathfrak{g},
\]
called the \emph{second-order process}, satisfying the algebraic relation
\begin{equation}\label{eq:RoughPathsEq}
\mathbb{W}_{s,t} - \mathbb{W}_{s,u} - \mathbb{W}_{u,t} = W_{s,u} \otimes W_{u,t},
\end{equation}
for all \(0 \le s \le u \le t \le 1\). We denote the pair by \(\mathbf{W} = (W,\mathbb{W})\).
\end{definition}
The reader may simply view a rough path as the data of a pair \((W,\mathbb{W})\), together with an algebraic constraint expressing the compatibility of the second-order term \(\mathbb{W}\) with the first-order term \(W\).

We now recall the notion of rough paths with H\"older regularity \(\alpha\), following~\cite[Definition 2.1, p.~14]{friz_hairer}.

\begin{definition}[\(\alpha\)-H\"older rough paths]
Let \(\alpha \in (\tfrac{1}{3}, \tfrac{1}{2}]\). The space of \(\alpha\)-H\"older rough paths over \(\mathfrak{g}\), denoted by \(\mathcal{C}^\alpha([0,1],\mathfrak{g})\), consists of pairs \((W,\mathbb{W}) =: \mathbf{W}\) such that
\begin{align*}
\|W\|_\alpha &:= \sup_{s \neq t \in [0,1]} \frac{|W_{s,t}|_{\mathfrak{g}}}{|t-s|^\alpha} < \infty, \\
\|\mathbb{W}\|_{2\alpha} &:= \sup_{s \neq t \in [0,1]} \frac{|\mathbb{W}_{s,t}|_{\mathfrak{g} \otimes \mathfrak{g}}}{|t-s|^{2\alpha}} < \infty,
\end{align*}
and such that the algebraic constraint~\eqref{eq:RoughPathsEq} holds.
\end{definition}

Note that a rough path is a \emph{deterministic} object. We now explain how rough paths arise in our construction. Recall that
\[
W_{N,s,r} := \int_s^r \gamma^* A_{\mathcal{N},N}
\]
is a Cauchy sequence of \(\mathfrak{g}\)-valued Gaussian random variables, whose limit has covariance  (in entries $a$ and $b$)
\[
\mathbb{E}\!\left[ W^a_{s,r} \otimes W^b_{s,r} \right]
= \delta_{ab}\, \cdot \Big\| \Big[\angle\big(\gamma([s,r])\big)\Big] \Big\|_{L^2(\Sigma)}^2.
\]
Here \([\angle(\gamma([s,r]))]\) is the current \(\iota_V \mathcal{L}_V^{-1}\big(\gamma([s,r])\big)\). In the case $\gamma([s,r]) $ intersects no unstable curve then $[\angle(\gamma([s,r]))]$ is nothing but the indicator function on the triangular domain bounded by $\gamma([s,r])$ and the two flowlines connecting from $\min(f)$ to $\gamma(s) $ and from $\min(f)$ to $\gamma(r)$. Now, we decompose \(W_{s,r}\in \mathfrak{g}\) with respect to some orthonormal basis of \(\mathfrak{g}\). This yields independent centered Gaussian components \((W^a_{s,r})_{a=1}^{\dim(\mathfrak{g})}\).

This provides the first piece of data for our rough path. We now turn to the \emph{second-order process}.

There is a subtle point concerning the second-order term: it does not appear to be directly defined along the Cauchy sequence \(A_{\mathcal{N},N}\), \(N \ge 0\). To define it, we first let \(N \to +\infty\) along the Cauchy sequence and then mollify once more along the curve \(\gamma\), introducing a new parameter \(\varepsilon\). In fact, we construct the second-order term as a measurable function of the integrals \(W^{\mathfrak{g}}_{s,r}\), via a suitable mollification procedure. This will imply that the second-order term is measurable with respect to \(A_{\mathcal{N}}\).

\paragraph{The second-order process for smooth transverse curves.}

Assume that \(\gamma\) is \(C^\infty\) and transverse to \(V\), and that the map
\[
r \mapsto \sigma\big(\angle\gamma([s,r])\big)
\]
is \emph{strictly increasing}\footnote{In the terminology of \cite{BCDRT}, this means that the curve \(\gamma\) is either of type I or type II only.}. Then the identity
\begin{align*}
\mathbb{E}\!\left[ W_{\gamma,s,r_1}^a\, W_{\gamma,s,r_2}^a \right]
=
\inf\Big\{
\sigma\big(\angle\gamma([s,r_1])\big),
\sigma\big(\angle\gamma([s,r_2])\big)
\Big\}
\end{align*}
shows that the process \(r \mapsto W^{\mathfrak{g}}_{\gamma,s,r}\) is nothing but a \(\mathfrak{g}\)-valued Brownian motion reparametrized by the time $
\sigma\big(\angle\gamma([s,r])\big).$

To describe the geometric rough path enhancement of the initial process \(W^{\mathfrak{g}}_{\gamma,s,r}\), we need to define the second-order term~\cite[Eq.~(2.2), p.~14]{friz_hairer}. In the particular case where \(\gamma\) is of type I, meaning that \(\gamma\) is transverse to \(V\) and does not intersect any unstable curve, we may define the second-order term by
\begin{equation}
\mathbb{W}_{s,t}
:= \int_s^t W^{\mathfrak{g}}_{s,r} \otimes \mathrm{d}_r W^{\mathfrak{g}}_{s,r}
\in \mathfrak{g}\otimes\mathfrak{g},
\end{equation}
where the differential \(\mathrm{d}_r W^{\mathfrak{g}}_{\gamma,r}\) is taken in the \emph{Stratonovich} sense with respect to the variable \(r\); see~\cite[§3.3, pp.~32--34]{friz_hairer} for the construction of the rough path
\[
\mathbf{B}^{\mathsf{Strat}} := \left(B^t, \int_s^t B_{s,r} \circ \mathrm{d}B_r\right).
\]
The stochastic integral is well defined because \(r \mapsto \sigma\big(\angle\gamma([s,r])\big)\) is \(C^\infty\) and strictly increasing, so that \(W^{\mathfrak{g}}_{\gamma,r}\) is a semimartingale. However, there is a problem when \(\gamma\) intersects an unstable curve at an endpoint\footnote{In the terminology of \cite{BCDRT}, this is a curve of type II.}: the function \(r \mapsto \sigma\big(\angle\gamma([s,r])\big)\) is no longer \(C^1\) in \(r\), so classical stochastic calculus cannot be used to define the second-order term \(\mathbb{W}\). We will return to this case later.

A component of the second-order term is given by
\[
\mathbb{W}_{s,t}^{a,b}
:= \int_s^t W^a_{s,r} \circ \mathrm{d}W_r^b,
\qquad (a,b)\in \{1,\dots,\dim(\mathfrak{g})\}^2.
\]
\paragraph{Constructing the second-order process for piecewise \(C^1\) curves.}

When the curve \(\gamma\) is only admissible or piecewise \(C^1\), it is no longer possible to construct the second-order process simply by invoking Stratonovich differentials. Instead of postulating the second-order object, we use the probabilistic origin of \(W^{\mathfrak{g}}_\gamma\) as the integral of a random connection \(A_{\mathcal{N}}\) along \(\gamma\). We will show that, after a suitable mollification, the sequence of second-order terms \(\mathbb{W}^{a,b}_{\varepsilon,s,r}\) converges as \(\varepsilon \to 0^+\) to a limiting random function with the expected properties.

At this point, we recall the definition of geometric rough paths~\cite[§2.2, p.~16]{friz_hairer}.

\begin{definition}[Geometric rough paths]
Let \(\mathbf{W} := (W,\mathbb{W})\) be a \(\mathfrak{g}\)-valued rough path on the interval \([0,1]\). We denote its components by \(W^i\) and \(\mathbb{W}^{j,k}\), where \(i,j,k \in \{1,\dots,\dim(\mathfrak{g})\}\) are Lie algebra indices corresponding to a chosen orthonormal basis of \(\mathfrak{g}\). Then \(\mathbf{W}\) is said to be a \emph{geometric rough path} if
\begin{equation}
\mathbb{W}_{s,t}^{i,j} + \mathbb{W}_{s,t}^{j,i} = W^i_{s,t} W^j_{s,t}
\end{equation}
for all \((i,j) \in \{1,\dots,\dim(\mathfrak{g})\}^2\) and all \((s,t) \in [0,1]^2\).
\end{definition}

The second-order term obtained from the above limiting procedure automatically satisfies
\begin{equation}
\mathbb{W}_{s,t}^{i,j} + \mathbb{W}_{s,t}^{j,i} = W^i_{s,t} W^j_{s,t},
\end{equation}
which is the defining relation for geometric rough paths~\cite[Eq.~(2.5), p.~16]{friz_hairer}. In terms of Wiener chaos expansion, \(\mathbb{W}_{s,t}^{i,j}\) is uniquely given by
\begin{align*}
\mathbb{W}_{s,t}^{i,j}
:= \frac{1}{2}\left( W^i_{s,t} W^j_{s,t}
+ \int_s^t \underbrace{:\! W^i_{\gamma,s,r}\,\mathrm{d}_r W^j_{\gamma,r}
- W^j_{\gamma,s,r}\,\mathrm{d}_r W^i_{\gamma,r}\!:}_{\text{antisymmetric term}} \right),
\end{align*}
where the underbraced antisymmetric term (the \emph{Lévy area} term), whenever it makes sense, lies in the second homogeneous Wiener chaos.\footnote{Beware that the Wick square is ill defined except when the curve is transverse to \(V\) and does not intersect unstable curves; in that case, the Stratonovich differential defines the Wick square. But when \(\gamma\) intersects unstable curves, it is no longer clear how to make sense of the Wick square.}

\subsection{Existence of the second order term for smooth curves}
We denote by \(V^\perp\) the smooth vector field obtained by rotating \(V\) by an angle of \(\frac{\pi}{2}\) in each fiber. Assume that \(\gamma\) is a \(C^\infty\) curve transverse to \(V\) and that it satisfies the following important property: for every \(t\in [0,1]\), the scalar product
\begin{equation}\label{eq:constantsign}
\left\langle \gamma'(t), V^\perp(\gamma(t)) \right\rangle_{T_{\gamma(t)}\Sigma}
\end{equation}
never vanishes and does not change sign. Thus, we reduce to the case where \(\gamma\) satisfies condition~\eqref{eq:constantsign}, which means that the area function
\[
r \mapsto \sigma\big(\angle\gamma[s,r]\big)
\]
is strictly monotone; this will play a crucial role later.

Define \(W_{\varepsilon;s,r}^i := W_{s,r}^i \star \chi_\varepsilon\), where
\[
\chi_\varepsilon(r) = \varepsilon^{-1}\chi(r\varepsilon^{-1}), \qquad \chi \in C_c^\infty(\mathbb{R}), \ \chi \ge 0, \ \int \chi = 1,
\]
and where the convolution product is taken in the \(r\)-variable. The corresponding mollified second-order term is
\begin{align*}
\mathbb{W}_{\varepsilon; s,t}
:= \int_s^t W_{\varepsilon; s,r}\otimes \mathrm{d}_r W_{\varepsilon,r}
\in \mathfrak{g}\otimes \mathfrak{g},
\end{align*}
where the right-hand side is well defined since everything is smooth.

Using orthogonality in \(\mathfrak{g}\otimes \mathfrak{g}\), Wick's theorem, and the independence of the different components of \(W\), we obtain the following identity:
\begin{align*}
\mathbb{E}\!\left[ \left\langle \mathbb{W}_{\varepsilon; s,t}, \mathbb{W}_{\varepsilon; s,t} \right\rangle_{\mathfrak{g}\otimes \mathfrak{g}} \right]
&= \sum_{i,j} \mathbb{E}\!\left[\left(\mathbb{W}^{i,j}_{\varepsilon;s,t}\right)^2\right] \\
&= \sum_{i,j} \mathbb{E}\!\left[\left(\int_s^t W_{\varepsilon;s,r_1}^i\, \mathrm{d}_r W^j_{\varepsilon;s,r_1}\right)
\left(\int_s^t W_{\varepsilon;s,r_2}^i\, \mathrm{d}_r W^j_{\varepsilon;s,r_2}\right)\right] \\
&= \sum_{i,j}\int_s^t\!\!\int_s^t
\mathbb{E}\!\left[ W_{\varepsilon;s,r_1}^i W_{\varepsilon;s,r_2}^i \right]
\mathbb{E}\!\left[ \mathrm{d}_{r_1}W_{\varepsilon;s,r_1}^j\, \mathrm{d}_{r_2}W_{\varepsilon;s,r_2}^j \right] \\
&\quad + \sum_i \left(\int_s^t \mathbb{E}\!\left[ W_{\varepsilon;s,r_1}^i\, \mathrm{d}_{r_1}W^i_{\varepsilon;s,r_1} \right]\right)
\left(\int_s^t \mathbb{E}\!\left[ W_{\varepsilon;s,r_2}^i\, \mathrm{d}_{r_2}W^i_{\varepsilon;s,r_2} \right]\right).
\end{align*}

Let us stop here and comment on the previous identity. We decomposed the scalar product \(\left\langle \mathbb{W}_{\varepsilon; s,t}, \mathbb{W}_{\varepsilon; s,t} \right\rangle_{\mathfrak{g}\otimes \mathfrak{g}}\) with respect to an orthonormal basis of \(\mathfrak{g}\otimes \mathfrak{g}\). We then applied Wick's theorem to each expectation
\[
\mathbb{E}\!\left[\left( \int_s^t W_{\varepsilon;s,r_1}^i\, \mathrm{d}_r W^j_{\varepsilon;s,r_1} \right)
\left( \int_s^t W_{\varepsilon;s,r_2}^i\, \mathrm{d}_r W^j_{\varepsilon;s,r_2} \right)\right],
\]
since the \(W_\varepsilon^i\) are Gaussian random variables. We also used the fact that the components \(W^i\) and \(W^j\) are independent for \(i\neq j\), which eliminates many terms in the sum. Finally, note that the differentials in the above integrals are ordinary differentials, not stochastic differentials, since everything has been mollified.

Now, by definition of $W^i_\varepsilon$, we can calculate each term more explicitly as
\begin{align*}
&\mathbb{E}\left[  W_{ \varepsilon; s,r_1}^iW_{\varepsilon;s,r_2}^i  \right]= 
\int\limits_{[s,t]^2} \inf\left\{ \sigma\big(\angle \gamma([s,u_1])\big),\sigma\big(\angle \gamma([s,u_2])\big) \right\}\chi_\varepsilon(r_1-u_1)\chi_\varepsilon(r_2-u_2)\dd u_1\dd u_2, 
\end{align*}
hence
\begin{align*}
&\mathbb{E}\left[\dd_{r_1}  W_{ \varepsilon; s,r_1}^i\dd_{r_2}W_{\varepsilon;s,r_2}^i  \right]\\
&\hspace{2cm}=   
\int\limits_{[s,t]^2} \inf\left\{ \sigma\big(\angle \gamma([s,u_1])\big),\sigma\big(\angle \gamma([s,u_2])\big) \right\}\partial_{r_1}\chi_\varepsilon(r_1-u_1)\partial_{r_2}\chi_\varepsilon(r_2-u_2)\dd u_1\dd u_2, 
\end{align*}
and
also
\begin{align*}
& \int_s^t \mathbb{E}\left[ W_{\varepsilon;s,r_1}^i \dd_{r_1}W^i_{\varepsilon;s,r_1} \right]\\
&\hspace{2cm}=    \int\limits_{[s,t]^2} \inf\left\{ \sigma\big(\angle \gamma([s,u_1])\big),\sigma\big(\angle \gamma([s,u_2])\big) \right\}\chi_\varepsilon(r_1-u_1)\partial_{r_1}\chi_\varepsilon(r_1-u_2)\dd u_1\dd u_2.  
\end{align*}
Let us first deal with the following term. 
\begin{align*}
&\int_s^t\int_s^t
\mathbb{E}\left[  W_{ \varepsilon; s,r_1}^iW_{\varepsilon;s,r_2}^i  \right]\mathbb{E}\left[  \dd_{r_1}W_{\varepsilon;s,r_1}^j \dd_{r_2} W_{\varepsilon;s,r_2}^j  \right]\\
&\hspace{2cm}= \int\limits_{[s,t]^2} \dd r_1 \dd r_2 \\
&\hspace{4cm}\cdot \int\limits_{[s,t]^2} \inf\left\{ \sigma\big(\angle \gamma([s,u_1])\big),\sigma\big(\angle \gamma([s,u_2])\big) \right\}\chi_\varepsilon(r_1-u_1)\chi_\varepsilon(r_2-u_2)\dd u_1\dd u_2 \\
&\hspace{6cm}\cdot\partial_{r_1}\partial_{r_2} \int \dd v_1\dd v_2 \chi_\varepsilon(r_1-v_1)\chi_\varepsilon(r_2-v_2) \\
&\hspace{10cm}\cdot \inf\left\{\sigma\big(\angle \gamma([s,v_1])\big),\sigma\big(\angle \gamma([s,v_2])\big) \right\}
\end{align*}
Now let us make the crucial observations.

First, the function
\[
(r_1,r_2)\mapsto
\inf\!\big\{ \sigma\big(\angle \gamma([s,r_1])\big),\sigma\big(\angle \gamma([s,r_2])\big) \big\}
\]
is well defined and of bounded variation on \([s,1]^2\); applying \(\mathrm{d}_{r_1}\otimes\mathrm{d}_{r_2}\) to this function yields a well-defined current of degree \(2\). Owing to the monotonicity of the area functional, the distribution
\[
\partial_{r_1}\partial_{r_2}\,
\inf\!\big\{\sigma\big(\angle \gamma([s,r_1])\big),\sigma\big(\angle \gamma([s,r_2])\big)\big\}
\]
is a signed measure of finite mass supported on a subset of the diagonal
\[\{(r,r)\in[s,1]^2:\partial_r\sigma(\angle\gamma([s,r]))\neq0\}\subset[s,1]^2.\]
To see this, note that away from the diagonal the infimum is equal locally to one of the two smooth functions \(\sigma(\angle\gamma([s,r_1]))\) or \(\sigma(\angle\gamma([s,r_2]))\), so the mixed derivative vanishes there by the chain rule.

By Lemma~\ref{lem:regularityarea}, the derivatives of the area have at worst logarithmic singularities on \([0,1]\); hence
\[
r\mapsto \partial_r\sigma\big(\angle\gamma([s,r])\big)\in L^p([s,1])
\quad\text{for all }1\le p<\infty,
\]
uniformly in \(s\in[0,1]\). Mollification preserves this integrability: for any \(\psi\in L^p\) and the mollifier \(\chi_\varepsilon\),
Young's inequality gives
\[
\|\chi_\varepsilon\star\psi\|_{L^p}\le\|\chi_\varepsilon\|_{L^1}\,\|\psi\|_{L^p}=\|\psi\|_{L^p},
\]
so the mollified derivatives remain uniformly in \(L^p\).

This implies that 
\[\partial_{r_1}\partial_{r_2} \int_{v_1,v_2} \chi_\varepsilon(r_1-v_1)\chi_\varepsilon(r_2-v_2) \inf\left\{ \sigma\big(\angle \gamma([s,v_1])\big),\sigma\big(\angle \gamma([s,v_2])\big) \right\}\dd v_1\dd v_2 \]
converges when $\varepsilon\rightarrow 0^+$
to
 $\left(\partial_r\sigma\big(\angle \gamma([s,r_1])\big) \right) ^2 \delta(r_1-r_2) $  where $\left(\partial_r\sigma\big(\angle \gamma([s,.])\big) \right) ^2$ belongs to all $L^p$ spaces for $p\in [1,+\infty)$.

Therefore
letting $\varepsilon\rightarrow 0^+$,
the expression $\mathbb{E}\left[ \left\langle \mathbb{W}_{\varepsilon; s,t}, \mathbb{W}_{\varepsilon; s,t} \right\rangle_{\mathfrak{g}\otimes \mathfrak{g}}\right]$ converges and simplifies as 
\begin{align*}
&\mathbb{E}\left[ \left\langle \mathbb{W}_{s,r}, \mathbb{W}_{s,r} \right\rangle_{\mathfrak{g}\otimes \mathfrak{g}}\right]=C_1 \int_s^t \sigma\big(\angle \gamma([s,r])\big) \left(\partial_r\sigma\big(\angle \gamma([s,r])\big) \right) ^2 \dd r +C_2 \sigma\big(\angle \gamma([s,t])\big)^2
\end{align*}
where $C_1,C_2>0$ are combinatorial constants coming from Wick's Theorem.
\paragraph{Kolmogorov criterion for rough paths.}

Now that we have proved the convergence of the second-order process \(\mathbb{W}\), it remains to study its regularity. Our goal is therefore to bound
\[
\mathbb{E}\!\left[\left\langle \mathbb{W}_{s,t}, \mathbb{W}_{s,t} \right\rangle_{\mathfrak{g}\otimes\mathfrak{g}}\right].
\]

By Hölder's inequality, for every \(p\in[1,\infty)\), the area increments satisfy
\begin{align*}
\sigma\big(\angle\gamma([s,t])\big)
&\le \int_s^t \left|\partial_r \sigma\big(\angle\gamma([s,r])\big)\right|\,\mathrm{d}r \\
&\le |t-s|^{\frac{p-1}{p}} \big\|\partial_r \sigma\big(\angle\gamma([s,r])\big)\big\|_{L^p},
\end{align*}
which yields
\begin{align*}
\mathbb{E}\!\left[\left\langle \mathbb{W}_{s,t}, \mathbb{W}_{s,t} \right\rangle_{\mathfrak{g}\otimes\mathfrak{g}}\right]
&\lesssim
\left(\int_s^t \big(|s-r|^{\frac{p-1}{p}}\big)^{\frac{p}{p-1}} \mathrm{d}r\right)^{\frac{p-1}{p}}
\left(\int_s^t \big(\partial_r \sigma(\angle\gamma([s,r]))\big)^{2p}\,\mathrm{d}r\right)^{\frac{1}{p}} \\
&\qquad + |t-s|^{\frac{2(p-1)}{p}} \\
&\lesssim |t-s|^{2\frac{p-1}{p}} + |t-s|^{\frac{2(p-1)}{p}}
\lesssim |t-s|^{2\frac{p-1}{p}}.
\end{align*}
Here we used again Hölder's inequality and the fact that \(\partial_r \sigma\big(\angle\gamma([s,r])\big)\) belongs to \(L^p\) for every \(p\in[1,\infty)\).

Observe that \(W\) and \(\mathbb{W}\) are polynomial functionals of the white noise of degrees \(1\) and \(2\), respectively, so they belong to the second Wiener chaos. By hypercontractivity, for every \(q\ge 2\),
\begin{equation}\label{eq:Kolmo1}
\mathbb{E}\!\left[|W_{s,t}|_{\mathfrak{g}}^{2q}\right]^{\frac{1}{2q}}
\lesssim
\mathbb{E}\!\left[|W_{s,t}|_{\mathfrak{g}}^{2}\right]^{\frac{1}{2}}
\lesssim |t-s|^{\frac{\beta}{2}},
\end{equation}
and
\begin{equation}\label{eq:Kolmo2}
\mathbb{E}\!\left[|\mathbb{W}_{s,t}|_{\mathfrak{g}\otimes\mathfrak{g}}^{4q}\right]^{\frac{1}{4q}}
\lesssim
\mathbb{E}\!\left[|\mathbb{W}_{s,t}|_{\mathfrak{g}\otimes\mathfrak{g}}^{2}\right]^{\frac{1}{2}}
\lesssim |t-s|^{\beta},
\end{equation}
for every \(\beta\in(0,1)\) with \(\beta>\frac{1}{q}\) (it is enough to choose \(\frac{p-1}{p}\ge \beta\)).

It follows from the Kolmogorov criterion for rough paths~\cite[Thm.~3.1, p.~28]{friz_hairer} that \(\mathbf{W}:=(W,\mathbb{W})\in \mathcal{C}^\alpha\) for all \(\alpha\in(0,\tfrac12)\), and we obtain estimates of the form
\begin{equation*}
\mathbb{E}_{\mathsf{Free}}\left[
\sup_{s\neq t\in[0,1]} \frac{|W_{s,t}|_{\mathfrak{g}}}{|t-s|^\alpha}
+
\sup_{s\neq t\in[0,1]} \frac{|\mathbb{W}_{s,t}|_{\mathfrak{g}\otimes\mathfrak{g}}}{|t-s|^{2\alpha}}
\right]
<\infty.
\end{equation*}
\subsection{Differential equations driven by rough paths}

Rough path theory provides an almost deterministic framework for studying stochastic differential equations driven by rough signals, which in our setting define the holonomies of the random connection. Before stating the main result on differential equations driven by rough paths, we recall a key tool of the theory: the integration theorem against controlled paths, due to Gubinelli.

We begin with the notion of \emph{controlled paths}, which describes the space of \(\mathfrak{g}\)-valued functions \(Y\) on \([0,1]\) whose singularities are modelled on those of a given reference rough path \(\mathbf{W}\)~\cite[Def.~4.6, p.~56]{friz_hairer}.
\begin{definition}[Controlled paths]
Given a rough path \(\mathbf{W}\in \mathcal{C}^\alpha([0,1],\mathfrak{g})\), we say that a path \(Y\) is controlled by \(\mathbf{W}\) if there exists \(Y' \in \mathcal{C}^\alpha([0,1],\mathrm{End}(\mathfrak{g}))\) such that
\begin{equation}\label{eq:controlledRP}
Y_{s,t}=Y'_s W_{s,t}+R^Y_{s,t},
\end{equation}
where the remainder satisfies \(\|R^Y\|_{2\alpha}<\infty\). This defines the space of \emph{controlled rough paths}
\[
(Y,Y')\in \mathcal{D}^{2\alpha}_{\mathbf{W}}([0,1],\mathfrak{g}).
\]
\end{definition}

Note that the relation between \(Y\), \(Y'\), and \(R^Y\) is not necessarily unique; what matters is the existence of such a decomposition rather than uniqueness of the representatives. Another important point is that the remainder term \(R^Y\) is more regular, which is crucial for applications.

We also need a notion of integration against rough paths due to Gubinelli~\cite[Thm.~4.10, p.~57]{friz_hairer}:
\begin{thm}[Integration against rough paths]\label{thm:gubintegral}
Let \(Y\in \mathcal{D}^{2\alpha}_{\mathbf{W}}\) be controlled by a driving rough path \(\mathbf{W}\in \mathcal{C}^\alpha\). Define
\begin{equation*}
\int_0^t Y_s\,\mathrm{d}\mathbf{W}_s
:= \lim_{\mathcal{P}\to 0^+} \sum_{[s,t]\in \mathcal{P}} \Big( Y_s W_{s,t}+Y'_s\mathbb{W}_{s,t} \Big).
\end{equation*}
 The above  limit\footnote{This formula shows that the second-order term \(\mathbb{W}\) must be taken into account.} exists, and the corresponding linear map
\begin{equation*}
\mathcal{D}^{2\alpha}_{\mathbf{W}}([0,1],\mathrm{End}(\mathfrak{g}))
\longrightarrow
\mathcal{D}^{2\alpha}_{\mathbf{W}}([0,1],\mathfrak{g}),
\qquad
(Y,Y')\longmapsto \left(\int_0^{\cdot} Y_s\,\mathrm{d}\mathbf{W}_s,\, Y\right),
\end{equation*}
is continuous with respect to the relevant Banach norms.
\end{thm}

One should also keep in mind that It\^o and Stratonovich stochastic integrals are particular cases of Gubinelli's integration when the driving rough path is \(\mathbf{B}^{\mathsf{Ito}}\) or \(\mathbf{B}^{\mathsf{Strat}}\), respectively, as discussed in~\cite[Chapter 5]{friz_hairer}.

The main theorem used to integrate ODEs driven by rough signals reads as follows; set \(\beta>\frac{1}{3}\):
\begin{thm}\label{thm:roughpathODE}
Let \(a\in M_N(\mathbb{C})\) be an initial condition and let \(\mathbf{W}\in \mathcal{C}^\beta(\mathbb{R}_+,\mathfrak{g})\) be a driving rough path, with \(\beta\in \left(\frac{1}{3},\frac{1}{2}\right)\). Then there exists a unique element \((Y,Y')\in \mathcal{D}^{2\alpha}_{\mathbf{W}}([0,1],\mathfrak{g})\), controlled by the driving rough path \(\mathbf{W}\), which solves the ODE in integral form
\begin{equation}
Y_t=a+\int_0^t Y_s\,\mathrm{d}\mathbf{W}_s
\end{equation}
for all \(t<\tau\), for some \(\tau>0\). Here the integral against \(\mathrm{d}\mathbf{W}_s\) is understood in the sense of Theorem~\ref{thm:gubintegral}.
\end{thm}
\subsection{Level sets holonomy process}

Recall that in the previous part, we only allowed smooth curves $\gamma$ which are transverse to $V$ and eventually intersect unstable curves at most once and only at endpoints. Let us call such curves \emph{elementary curves}. Since our application is for a level set $\gamma$
of $f$, we need to decompose such $\gamma$ as concatenation of finite number of elementary curves.

\paragraph{Solving along elementary pieces of level curves.}

Let us summarize what we have obtained so far. The random connection $A_{\mathcal{N}}$ is constructed in the anisotropic spaces. For all $C^\infty$ elementary curve $\gamma$ everywhere transverse to $V$ and such that $\gamma$ intersects $\cup_{a\in \mathsf{Crit}_1(f)} W^u(a)$ at most once and it should be at endpoints (this ensures the monotonicity of $r\mapsto \sigma\big(\angle \gamma([0,r])\big)$), we constructed a geometric rough path $\mathbf{W}_\gamma$ as 
above such that it satisfies the two above estimates (\ref{eq:Kolmo1}, \ref{eq:Kolmo2}) prior to applying Kolmogorov. The rough path $\mathbf{W}_\gamma$ satisfies the input estimates of Kolmogorov  (\ref{eq:Kolmo1}, \ref{eq:Kolmo2}). 
So we just defined a map $\mathbf{W}$ which depends on our input random connection $A_{\mathcal{N}}$ in a measurable way and indexed by the above smooth curves $\gamma$ which is valued 
in the space of geometric rough paths of H\"older regularity $\alpha\in (\frac{1}{3},\frac{1}{2})$ (endowed with its natural distance)
and for given smooth $\gamma\in C^\infty([0,1],\Sigma)   $.

Therefore for each such curve $\gamma$ there is a process $\mathbf{W}_\gamma$ which is a rough path of regularity 
$\mathcal{C}^\alpha, \alpha\in (\frac{1}{3},\frac{1}{2})$ and using this H\"older rough path, we may solve the rough differential equation written in integral form
$$U(t)=\mathrm{Id}+\int_0^t U(s) \dd\mathbf{W}_s    $$
where the initial data is the identity matrix $\mathrm{Id}$.
The \emph{parallel transport} of $A_{\mathcal{N}}$ along $\gamma$ is the solution $U(1)$ of the above rough differential equation which 
depends measurably on $A_{\mathcal{N}}$.

\paragraph{Decomposing arbitrary transverse curves into elementary pieces and including jumps.}

For a general closed level curve $\gamma:[0,1] \mapsto \Sigma $ intersecting the union $\cup_{a\in \mathsf{Crit}_1(f)} W^u(a)$ transversally, this applies to any level set $\{x; f(x)=\max(f)-\varepsilon\}\subset \Sigma$ for $\varepsilon>0$ small enough,  we first 
denote by $0=t_1< \dots <t_k=1$ the times at which our curve $\gamma$ intersects 
$\cup_{a\in \mathsf{Crit}_1(f)} W^u(a)$, for $i\in \{1,\dots,k\}$, the point $ \gamma(t_i) $ belongs to the unstable curve $W^u(a_{j_i})$ indexed by the saddle point $a_{j_i}\in \mathsf{Crit}_1(f)$ and where $\gamma(0)=\gamma(1)$.
For each piece $\gamma([t_i,t_{i+1}])$, $i\in \{ 1,\dots,k-1 \}$ of curve, the previous construction yields a parallel transport matrix denoted by $U_{t_i,t_{i+1}}$ since each piece of curve $\gamma([t_i,t_{i+1}])$ is the concatenation of two elementary curves.
Then the holonomy process $\mathsf{Hol}_\gamma$ is defined as
\begin{equation}
\mathsf{Hol}_\gamma(A^{\mathsf{Free}}_{\Sigma,\sigma}):= g_{j_k}  U_{[t_{k-1},t_k]} \dots  g_{j_2}  U_{[t_1,t_2]}  g_{j_1} U_{[t_0,t_1]}  .  
\end{equation}

\begin{remark}
 The above discussion applies almost verbatim to piecewise $C^1$ curves obtained by finite composition of either flowlines or $C^1$ curve everywhere transverse to $V$ and therefore allows to define a holonomy $\mathsf{Hol}_\gamma$ for such curve. 
\end{remark}

\begin{corollary}
There exists a measurable map $ A\mapsto \mathsf{Hol}_\gamma(A)$ which is obtained as follows.
\begin{itemize}
    \item For every $C^\infty$ curve $\gamma$ everywhere transverse to $V$, define
    $\mathsf{Hol}\left( \gamma \right)
     := g_{t=1}$, where the $G$-valued path $\left( g_t \ ; \ 0 \leq t \leq 1 \right)$ solves the ODE (in the rough path sense)
     $$
     \dd g_t = g_t \circ \dd W_t^\gamma \ ,
     $$
     and $W_t^\gamma := W( \gamma_{[0,t]} )$. Because we are considering the time $1$ result of the flow driven by $W(\gamma)$, the notation $\mathsf{Hol}\left( \gamma \right) = \exp_{\mathrm{Curves}}\left( W(\gamma) \right)$ seems natural.
    \item The map $ \mathbf{W}(\gamma) \mapsto \mathsf{Hol}\left( \gamma \right)$ is continuous.
\end{itemize}
\end{corollary}
\begin{proof}
For any admissible curve $\gamma$, the pair $(W_{\gamma,s,t},\mathbb{W}_{\gamma,s,t})$ forms a geometric rough path of regularity $\mathcal{C}^\alpha$ for all $\alpha\in (0,\frac{1}{2})$.
Therefore the above theorem~\ref{thm:roughpathODE} yields a deterministic integration scheme in order to continuously map
$
   \mathbf{W}( \gamma )
   \mapsto \mathsf{Hol}_\gamma \ .
$
The continuity follows from the Ito--Lyons stability estimates, see ~\cite[Thm 8.5 p.~141]{friz_hairer} which states that solutions have Lipschitz dependence on the driving rough paths in the rough path topology (see the Lipschitz estimate in~\cite[on top of  p.~142]{friz_hairer}).
In the particular case when $\gamma$ is transverse and does not intersect unstable curves, the Holonomy process boils down to solving SDE in the Stratanovich sense. 
\end{proof}

Another consequence is that we have a well--defined of holonomy for the Yang--Mills measure on the closed surface which is no longer Gaussian since for any closed curve $\gamma$ obtained by finite composition of either flowlines or $C^1$ curve everywhere transverse to $V$, we have
\begin{align*}
&\mathbb{E}_{\mathsf{free}}\left(\sup_{s\neq t\in [0,1]} \frac{ \vert W_{s,t} \vert_{\mathfrak{g}} }{ \vert t-s\vert^\alpha  }  + \sup_{s\neq t \in [0,1] } \frac{\vert \mathbb{W}_{s,t} \vert_{\mathfrak{g}\otimes \mathfrak{g}}}{ \vert t-s\vert^{2\alpha} }  \right)<+\infty \\
&\implies
\mathbb{E}_{\mathsf{YM}}\left(\sup_{s\neq t\in [0,1]} \frac{ \vert W_{s,t} \vert_{\mathfrak{g}} }{ \vert t-s\vert^\alpha  }  + \sup_{s\neq t \in [0,1] } \frac{\vert \mathbb{W}_{s,t} \vert_{\mathfrak{g}\otimes \mathfrak{g}}}{ \vert t-s\vert^{2\alpha} }  \right)<+\infty.
\end{align*}

This concludes our justification that the Yang--Mills measure has a version in
the space $\mathcal{B}_{{\rm YM}}^{\alpha,p,s,\ell}$.

\section{Tangent space at regular points and the Atiyah--Bott--Lefschetz trace formula}\label{appendix:Atiyah-Bott-Lefschetz}
Let us sketch the proof of the following fact. At a regular point $\rho$ of $\mathcal{M}_g$, denote by $\nabla$ the corresponding flat connection (the one whose holonomy process represents $\rho$). We have \[ \dim T_\rho(\mathcal{M}_g)= \dim H^1 \left(\Omega^\bullet(\Sigma,\operatorname{End}(E)),\mathsf{ad}_\nabla\right) =(2g-2)\dim(G). \] 
Consider any vector field $V$ on our surface $\Sigma$ with discrete, nondegenerate critical points. Using any flat connection $\nabla$ on $E$, we may lift the action of the flow generated by $V$ on functions to an action 
\[e^{-t\mathcal{L}_V^{\mathsf{ad}_\nabla}}: C^\infty\big (\operatorname{End}(E)\big )\rightarrow C^\infty\big (\operatorname{End}(E)\big ),\]
where $\mathcal{L}_V^{\mathsf{ad}_\nabla}=[\mathsf{ad}_\nabla,\iota_V]$,
on sections of $C^\infty(\operatorname{End}(E))$ using $\mathsf{ad}_\nabla$. For fixed $t>0$, the operator $e^{-t\mathcal{L}_V^{\mathsf{ad}_\nabla}}$ is a \emph{geometric endomorphism} acting on the complex $\left(\Omega^\bullet(\Sigma,\operatorname{End}(E)),\mathsf{ad}_\nabla\right)$ in the terminology of Atiyah--Bott~\cite{AB67}. Its induced action on the cohomology $H^{\bullet}\left(\Omega^\bullet(\Sigma,\operatorname{End}(E)), \mathsf{ad}_\nabla \right)$ is the identity map since  $e^{-t\mathcal{L}_V^{\mathsf{ad}_\nabla}}$ is isotopic to the identity. By the Atiyah--Bott--Lefschetz fixed point formula, we find that 
\begin{align*}
  \chi\bigl(\Omega^\bullet(\Sigma, \operatorname{End}(E)), \mathsf{ad}_\nabla\bigr)
  &= \operatorname{Tr}\bigl(\mathrm{Id}\big|_{H^\bullet \to H^\bullet}\bigr) \\
  &= \sum_{a \in \operatorname{Fix}(V)} (-1)^{\operatorname{ind}(a)} \operatorname{Rank}(\operatorname{End}(E)) \\
  &= (2 - 2g) \dim(\mathsf{SU}(N)),
\end{align*}
 where we used the Poincar\'e--Hopf formula $\sum_{a\in \text{Fix}(V)} (-1)^{\ind(a)}=2-2g $. But since both \[
H^0\bigl(\Omega^\bullet(\Sigma,\operatorname{End}(E)),\mathsf{ad}_\nabla\bigr)
\quad\text{and}\quad
H^2\bigl(\Omega^\bullet(\Sigma,\operatorname{End}(E)),\mathsf{ad}_\nabla\bigr)
\]
vanish for $\nabla$ irreducible, the Euler characteristic verifies 
\[
\chi\bigl(\Omega^\bullet(\Sigma,\operatorname{End}(E)),\mathsf{ad}_\nabla\bigr)
=
-\dim\Bigl(H^1\bigl(\Omega^\bullet(\Sigma,\operatorname{End}(E)),\mathsf{ad}_\nabla\bigr)\Bigr),
\]
which concludes our argument.

\section{Symbolic index} \label{symb}

We collect in this appendix commonly used symbols of the article, together
with their meaning and, if relevant, the page where they first occur.

 \begin{center}
\renewcommand{\arraystretch}{1.1}
\begin{longtable}{lll}
\toprule
Symbol & Meaning & Page\\
\midrule
\endfirsthead
\toprule
Symbol & Meaning & Page\\
\midrule
\endhead
\bottomrule
\endfoot
\bottomrule
\endlastfoot
$\Sigma$ & Compact surface &\\
$\pi_1(\Sigma)$& Fundamental group of $\Sigma$ & \\
$\sigma$& Area measure, smooth volume form \\
$\partial \mathcal{S} $ & Boundary of $S$\\
$G$ & A compact connected semi-simple Lie group &\\
$1_G$& Identity on G&\\
$\g$ & Lie algebra of $G$ equipped with bi-invariant inner product&\\
$f$&  A Morse function & \\
$\nabla f$&  Gradient of $f$ & \\
$W^{s/u}(a)$& Stable and unstable manifolds of $a\in \mathsf{Crit}(f)$&\pageref{StableUnstable}\\
$\mathsf{Crit}_1(f)$&The set of critical points of index $1$&\pageref{Crit}\\
$\mathsf{Crit}(f)$&The set of all critical points of $f$&\pageref{Crit}\\
$[U_a]$& Unstable current associated to the critical point $a\in\mathsf{Crit}_1(f)$&\\
$\mathsf{Hol}_{\partial\mathcal{S}}$ & Combinatorial holonomy& \pageref{def:combihol}\\
$\mathcal{M}_g$& Moduli space of flat connections in Morse gauge for surfaces of genus $g$& \pageref{ModuSpace}\\
$T_{(g_a)_{a\in\mathsf{Crit}_1(f)}}\mathcal{M}_g$& Tangent cone to $\mathcal{M}_g$ at $(g_a)_{a\in\mathsf{Crit}_1(f)}$& \pageref{TangCone}\\
$\omega$& ABG symplectic form in Morse gauge& \pageref{ABGSymp}\\
$\mathcal{W}^{p;\alpha,\alpha-1}(\square)$& A space of anisotrpic $1$-currents on a flow box $\square$&\pageref{FlowboxSpace}\\
$\mathcal{W}^{\alpha,p,s,\ell}$& Weighted  anisotropic noise space& \pageref{NoiseSpace}\\
$\mathcal{B}^{\alpha,p,s,\ell}_{\mathsf{YM}}$& &\pageref{YmSpace}\\
$\pi_\mathsf{noise}$& Extraction of noise component  &\pageref{prop:keyprop} \\
$\pi_\mathsf{flat}$& Extraction of flat component &\pageref{prop:keyprop} \\
$\mathsf{Ext}_a$& $\mathsf{Ext}_a(M_b[U_b])=M_b\delta_{a}(b)$ & \pageref{prop:keyprop}\\
$A_{S,\sigma}^\mathsf{Free}$ &The Yang--Mills measure on $(S,\sigma)$ with free boundary condition &\pageref{FreeYMM}\\
$A_{\Sigma,\sigma}$ &The Yang--Mills measure on $(\Sigma,\sigma)$ &\pageref{ClosedYMM}\\
$A_{\Sigma,0}$& The limit of the Yang--Mills measure in $0$ area & \\
$\Sigma_{\leq r}$& The subset $\{x\in \Sigma : f(x) \leq r\}$ \\
$\Sigma_{\geq r}$& The subset $\{x\in \Sigma : f(x) \geq r\}$ \\
$Z_{\Sigma,\sigma}$& Partition function of the Yang--Mills measure on $(\Sigma,\sigma)$\\
$Z_{S,\sigma}(g)$& Partition function with boundary condition  $\mathsf{Hol}(\partial \mathcal{S} )=[g]$\\
$\gamma_r$& The curve $\{f=r\}$ \\

$\mathsf{Hol}_r(A)$& The holonomy of $\gamma_r$ under $A$
\end{longtable}
\end{center}

\end{document}